\documentclass[11pt]{article}

\usepackage{amsmath,amsthm,amssymb,fancyhdr,graphicx,bbm,cancel,color,mathrsfs,todonotes,xypic, pinlabel,graphbox}
\usepackage[all]{xy}

\usepackage{hyperref}

\usepackage{bm}

\usepackage{lipsum}

\usepackage{tabto}
\newcommand\Tab{\tab \hspace{-3cm}}

\let\OLDthebibliography\thebibliography
\renewcommand\thebibliography[1]{
  \OLDthebibliography{#1}
  \setlength{\parskip}{0pt}
  \setlength{\itemsep}{0pt plus 0.3ex}
}

\newtheorem{thm}{Theorem}[section]
\newtheorem{mainthm}{Theorem}

\newtheorem{lem}[thm]{Lemma}
\newtheorem{prop}[thm]{Proposition}
\newtheorem{cor}[thm]{Corollary}
\theoremstyle{definition}
\newtheorem{defn}[thm]{Definition}
 
\newtheorem{example}[thm]{Example}

\newtheorem{rmk}[thm]{Remark}
\theoremstyle{remark}

\newcommand{\dmo}{\DeclareMathOperator}

\newcommand{\R}{\mathbb{R}}
\newcommand{\Q}{\mathbb{Q}}
\newcommand{\co}{\mathbb{C}}\newcommand{\Z}{\mathbb{Z}}

\newcommand{\al}{\alpha}\newcommand{\be}{\beta}\newcommand{\ga}{\gamma}\newcommand{\de}{\delta}\newcommand{\ep}{\epsilon}
\newcommand{\Om}{\Omega}\newcommand{\De}{\Delta}\newcommand{\ka}{\kappa}\newcommand{\la}{\lambda}
\newcommand{\Ga}{\Gamma}
\newcommand{\what}{\widehat}\newcommand{\wtil}{\widetilde}\newcommand{\Si}{\Sigma}\newcommand{\Lam}{\Lambda}
\newcommand{\cd}{\cdots}
\newcommand{\sbs}{\subset}\newcommand{\pa}{\partial}

\newcommand{\xra}{\xrightarrow}

\newcommand{\ra}{\rightarrow}
\newcommand{\hra}{\hookrightarrow}\newcommand{\onto}{\twoheadrightarrow}
\newcommand{\bb}[1]{\mathbb{#1}}\newcommand{\ca}[1]{\mathcal{#1}}\newcommand{\un}[1]{\underline{#1}}\newcommand{\ov}[1]{\overline{#1}}

\newcommand{\fr}[2]{\frac{#1}{#2}}

\newcommand{\rt}{\rtimes}

\newcommand{\ti}{\times}

\newcommand{\pair}[1]{\langle #1 \rangle}
\dmo{\sgn}{sign}\dmo{\Span}{span}
\dmo{\we}{\wedge}
\dmo{\ind}{ind}\dmo{\Ind}{Ind}
\dmo{\bop}{\bigoplus}\dmo{\pic}{Pic}
\dmo{\coker}{coker}\dmo{\vol}{Vol}\dmo{\gal}{Gal}\dmo{\perm}{Perm}
\dmo{\tor}{Tor}\dmo{\ext}{Ext}\dmo{\Ext}{Ext}
\dmo{\aut}{aut}
\dmo{\Aut}{Aut}
\dmo{\inn}{Inn}\dmo{\var}{Var}
\dmo{\dep}{depth}\newcommand{\rest}[2]{#1\bigr\vert_{#2}}
\dmo{\ad}{ad}\dmo{\curl}{curl}
\dmo{\hy}{\bb H}\dmo{\Sl}{SL}
\dmo{\SO}{SO}\dmo{\psl}{PSL}
\dmo{\isom}{Isom}\dmo{\Isom}{Isom}
\dmo{\conf}{Conf}
\dmo{\stab}{Stab}\dmo{\Jac}{Jac }
\dmo{\diam}{diam}\dmo{\fix}{Fixed}\dmo{\Fix}{Fix}
\dmo{\injR}{injRad}\dmo{\Ad}{Ad}
\dmo{\esv}{ess-vol}\dmo{\out}{Out}\dmo{\Out}{Out}
\dmo{\nil}{Nil}\dmo{\sol}{Sol}
\dmo{\Div}{div}
\dmo{\SU}{SU}
\dmo{\SP}{SP}
\dmo{\Sp}{Sp}
\dmo{\rk}{rk}
\dmo{\rank}{rank}
\dmo{\psp}{PSp}\dmo{\psu}{PSU}
\dmo{\PU}{PU}\dmo{\pgl}{PGL}
\dmo{\Mod}{Mod}\dmo{\range}{Range}
\dmo{\eu}{eu}\dmo{\mi}{mi}
\dmo{\Log}{Log}\dmo{\supp}{supp}
\dmo{\maps}{Maps}\dmo{\Gr}{Gr}
\dmo{\Pin}{Pin}
\dmo{\Spin}{Spin}\dmo{\Str}{Str}
\dmo{\Sq}{Sq}\dmo{\Symp}{Symp}
\dmo{\pd}{PD}\dmo{\PD}{PD}\dmo{\sig}{Sig}
\dmo{\Set}{Set}\dmo{\Top}{Top}
\dmo{\ev}{ev}\dmo{\St}{St}
\dmo{\Pt}{Pt}\dmo{\pt}{pt}

\dmo{\colim}{colim }\dmo{\Pl}{PL}
\dmo{\String}{String}\dmo{\smear}{smear}

\dmo{\dev}{dev}
\dmo{\met}{Met}\dmo{\contact}{Contact}
\dmo{\teich}{Teich}\dmo{\Teich}{Teich}\dmo{\qi}{QI}
\dmo{\der}{Der}
\dmo{\cl}{Cliff}\dmo{\Cl}{Cl}
\dmo{\Pf}{Pf}
\dmo{\GL}{GL}
\dmo{\PSL}{PSL}
\dmo{\ch}{ch}\dmo{\diag}{diag}
\dmo{\grad}{grad}\dmo{\Char}{char}
\dmo{\spec}{Spec}\dmo{\Arg}{Arg}
\dmo{\rad}{rad}\dmo{\im}{Im}
\dmo{\Hom}{Hom}\dmo{\End}{End}
\dmo{\tr}{tr}\dmo{\id}{Id}
\dmo{\gl}{GL}
\dmo{\sym}{Sym}\dmo{\Sym}{Sym}
\dmo{\com}{Comm}
\dmo{\Lk}{Lk}
\dmo{\CAT}{CAT}
\dmo{\Rep}{Rep}
\dmo{\Res}{Res}
\dmo{\Conf}{Conf}
\dmo{\PConf}{PConf}
\dmo{\Push}{Push}
\dmo{\Cont}{Cont}
\dmo{\sm}{\setminus}
\dmo{\vn}{\varnothing}
\dmo{\disk}{\mathbb D}
\dmo{\Trd}{Trd}\dmo{\Mat}{Mat}
\dmo{\Riem}{Riem}
\dmo{\Diffn}{\Diff_0}\dmo{\diff}{diff}
\dmo{\Diff}{Diff}\dmo{\homeo}{Homeo}
\dmo{\Homeo}{Homeo}\dmo{\Fr}{Fr}
\dmo{\rot}{rot}\dmo{\Emb}{Emb}
\dmo{\Ham}{Ham}\dmo{\Met}{Met}
\dmo{\Ein}{Ein}\dmo{\CP}{\co P}
\dmo{\Per}{Per}\dmo{\Ric}{Ric}
\dmo{\Nrd}{Nrd}
\dmo{\Comp}{Comp}\dmo{\PSC}{PSC}
\dmo{\Cent}{Cent}\dmo{\Orb}{Orb}
\dmo{\aind}{a-ind}\dmo{\tind}{t-ind}
\dmo{\constant}{constant}
\dmo{\Td}{Td}
\dmo{\LMod}{LMod}
\dmo{\SMod}{SMod}
\dmo{\SDiff}{SDiff}
\dmo{\Br}{Br}
\dmo{\csch}{csch}
\dmo{\triv}{triv}
\dmo{\genus}{genus}
\dmo{\Homeq}{HomEq}
\dmo{\PP}{\mathbb{P}}
\dmo{\U}{U}
\dmo{\Gal}{Gal}
\dmo{\BDiff}{\wtil{\Diff}}
\dmo{\BAut}{\wtil{\Aut}}
\dmo{\Iso}{Iso}
\dmo{\SL}{SL}
\dmo{\Cone}{Cone}
\dmo{\codim}{codim}
\dmo{\II}{II}
\dmo{\I}{I}
\dmo{\InjRad}{InjRad}
\dmo{\Inn}{Inn}
\dmo{\sys}{sys}
\dmo{\Comm}{Comm}
\dmo{\PO}{PO}
\dmo{\vertex}{Vert}
\dmo{\POm}{P\Om}
\dmo{\ab}{ab}
\dmo{\PSO}{PSO}
\dmo{\CRS}{CRS}

\usepackage{enumerate,pdfpages,stmaryrd} 
\usepackage[margin=1.5in]{geometry}

\usepackage{tikz}
\usepackage{dsfont}
\usetikzlibrary{matrix}

\begin{document}

\providecommand{\keywords}[1]{\textit{Key words and phrases.} #1}

\title{Convex-compact subgroups of the Goeritz group}

\author{Bena Tshishiku}



\date{\today}

\maketitle

\begin{abstract}
Let $G<\Mod_2$ be the Goeritz subgroup of the genus-2 mapping class group. We show that finitely-generated, purely pseudo-Anosov subgroups of $G$ are convex cocompact in $\Mod_2$, addressing a case of a general question of Farb--Mosher. We also give a simple criterion to determine if a Goeritz mapping class is pseudo-Anosov, which we use to give very explicit convex-cocompact subgroups. In our analysis, a central role is played by the primitive disk complex $\ca P$. In particular, we (1) establish a version of the Masur--Minksy distance-formula for $\ca P$, (2) classify subsurfaces $X\subset S$ that are infinite-diameter holes of $\ca P$, and (3) show that $\ca P$ is quasi-isometric to a coned-off Cayley graph for $G$. 
\end{abstract}

\setcounter{tocdepth}{1}
\tableofcontents


\section{Introduction}

Let $S^3=V\cup_S \what V$ be the genus-2 Heegaard splitting of $S^3$ (so $V,\what V$ are handlebodies and $S=\pa V=\pa \what V$ is a closed surface of genus 2). The \emph{Goeritz group} $\bb G<\Mod(S)$ is the group of isotopy classes of homeomorphisms of $S$ that extend to both handlebodies $V,\what V$. The algebraic structure of $\bb G$ is well understood \cite{goeritz,scharlemann, akbas, cho}. In particular, $\bb G$ is generated by the elements $\alpha,\beta,\gamma,\delta$ pictured in Figure \ref{fig:generators}, and with respect to this generating set, $\bb G$ has a simple amalgamated structure: 
\begin{equation}\label{eqn:splitting}
\begin{array}{ccccccccccccccc}
\bb G&\cong \big[&(\Z_2&\times&\Z_2)&\rtimes&\Z&\big]*_{\Z_2\times\Z_2}\big[&(\Z_3&\rtimes&\Z_2)&\times& \Z_2&\big]\\
&&\alpha&&\gamma&&\beta&&\delta&&\gamma&&\alpha
\end{array}\end{equation}

In this paper we study geometric aspects of $\bb G$ as a subgroup of $\Mod(S)$. Our main results follow.

\begin{mainthm}[Purely pseudo-Anosov implies convex cocompact]\label{thm:ccc}
Finitely-generated, purely pseudo-Anosov subgroups of the genus-$2$ Goeritz group $\bb G$ are convex cocompact in the genus-$2$ mapping class group $\Mod(S)$. 
\end{mainthm}

Farb--Mosher \cite[Question 1.5]{farb-mosher} ask if \emph{purely pseudo-Anosov} implies \emph{convex cocompact} for finitely-generated (free) subgroups of mapping class groups. This question is a special case of Gromov's hyperbolicity question; see \cite[Question 1.1]{bestvina-questions} and \cite{kent-leininger-subgroups}. In particular, by \cite{farb-mosher,hamenstadt}, Theorem \ref{thm:ccc} implies that for any purely pseudo-Anosov subgroup $G$ of the Goeritz group $\bb G$, the associated extension group $\wtil G$ 
\[1\ra\pi_1(S)\ra\wtil G\ra G\ra 1\]
is Gromov hyperbolic.

The question of Farb--Mosher remains open in general, but the answer is ``yes" in many special cases of geometric interest, including Veech groups, certain hyperbolic 3-manifold subgroups \cite{KLS,dowdall-kent-leininger}, and certain right-angled Artin subgroups \cite{mangahas-taylor,koberda-mangahas-taylor}. Theorem \ref{thm:ccc} is a new direction in this family of results.

\paragraph{Characterizing pseudo-Anosov elements of the Goeritz group.} 
We complement Theorem \ref{thm:ccc} by giving a simple criterion to determine if $g\in\bb G$ is pseudo-Anosov. 

\begin{mainthm}[Pseudo-Anosov characterization]\label{thm:pseudo}
Let $g\in\bb G<\Mod(S)$ be a mapping class in the genus-$2$ Goeritz group. Then $g$ is pseudo-Anosov if and only if $g$ is not conjugate into any of the following subgroups, where the elements $\alpha,\beta,\gamma,\delta$ are shown in Figure \ref{fig:generators}.
\begin{enumerate}
\item[$(i)$] \emph{(primitive-disk stabilizer)} $\pair{\alpha,\beta,\gamma\delta}\cong(\Z_2\ti\Z)*_{\Z_2}(\Z_6)$
\item[$(ii)$] \emph{(reducing-sphere stabilizer)} $\pair{\alpha,\beta,\gamma}\cong(\Z_2\ti\Z)\rt\Z_2$
\item[$(iii)$] \emph{(primitive pants-decomposition stabilizer)} $\pair{\alpha,\gamma,\delta}\cong \Z_2\times S_3$
\item[$(iv)$] \emph{(figure-8 knot stabilizer)} $\pair{\al,\beta\delta\beta^{-1}\delta,\ga\de}\cong\Z_2\ti(\Z\rt\Z_2)$
\end{enumerate} 
\end{mainthm}

\begin{figure}[h!]
\labellist
\small
\pinlabel $\alpha \text{ : hyperelliptic}$ at 300 600
\pinlabel $\beta \text{ : half-twist}$ at 800 540
\pinlabel $\gamma$ at 1200 560
\pinlabel $\delta$ at 1580 560
\pinlabel $\pi$ at 560 710
\pinlabel $\pi$ at 1070 670
\pinlabel $\pi$ at 1340 500
\pinlabel $4\pi/3$ at 1700 460
\endlabellist
\centering
    \includegraphics[scale=.2]{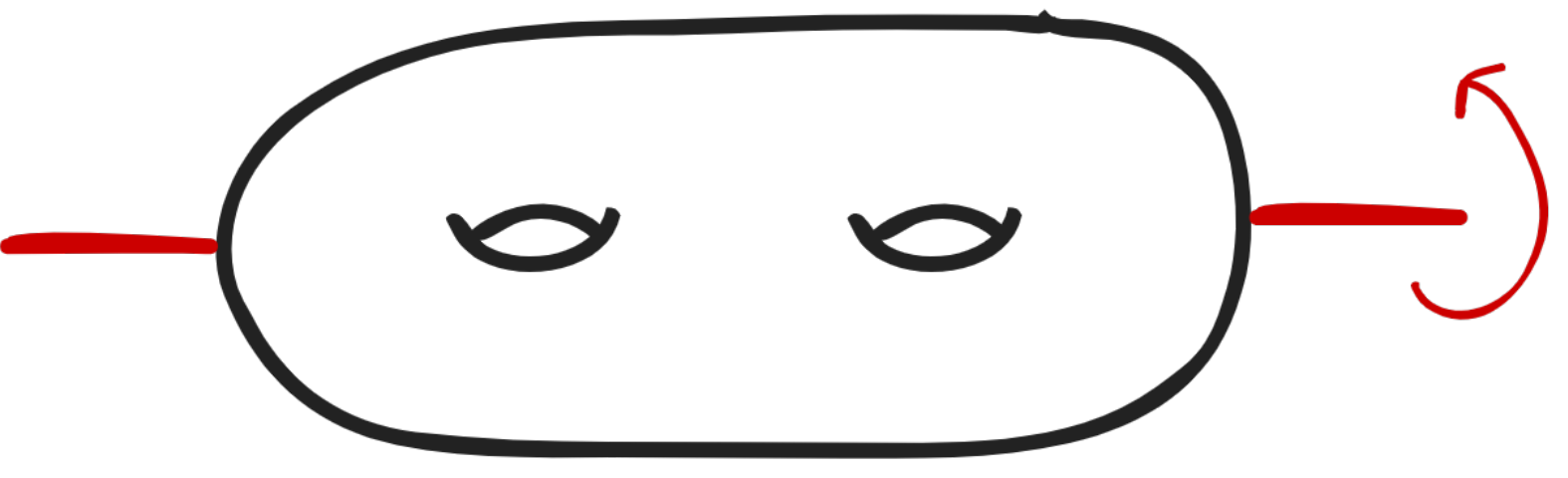}\hspace{.2in}
        \includegraphics[align=c,scale=.2]{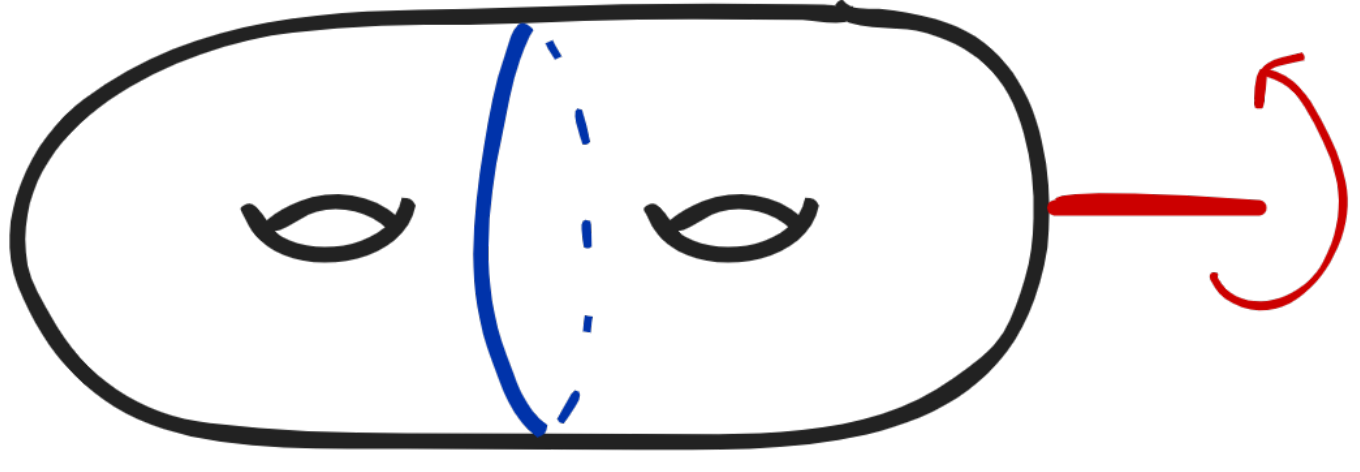}\hspace{.2in}
            \includegraphics[align=c,scale=.2]{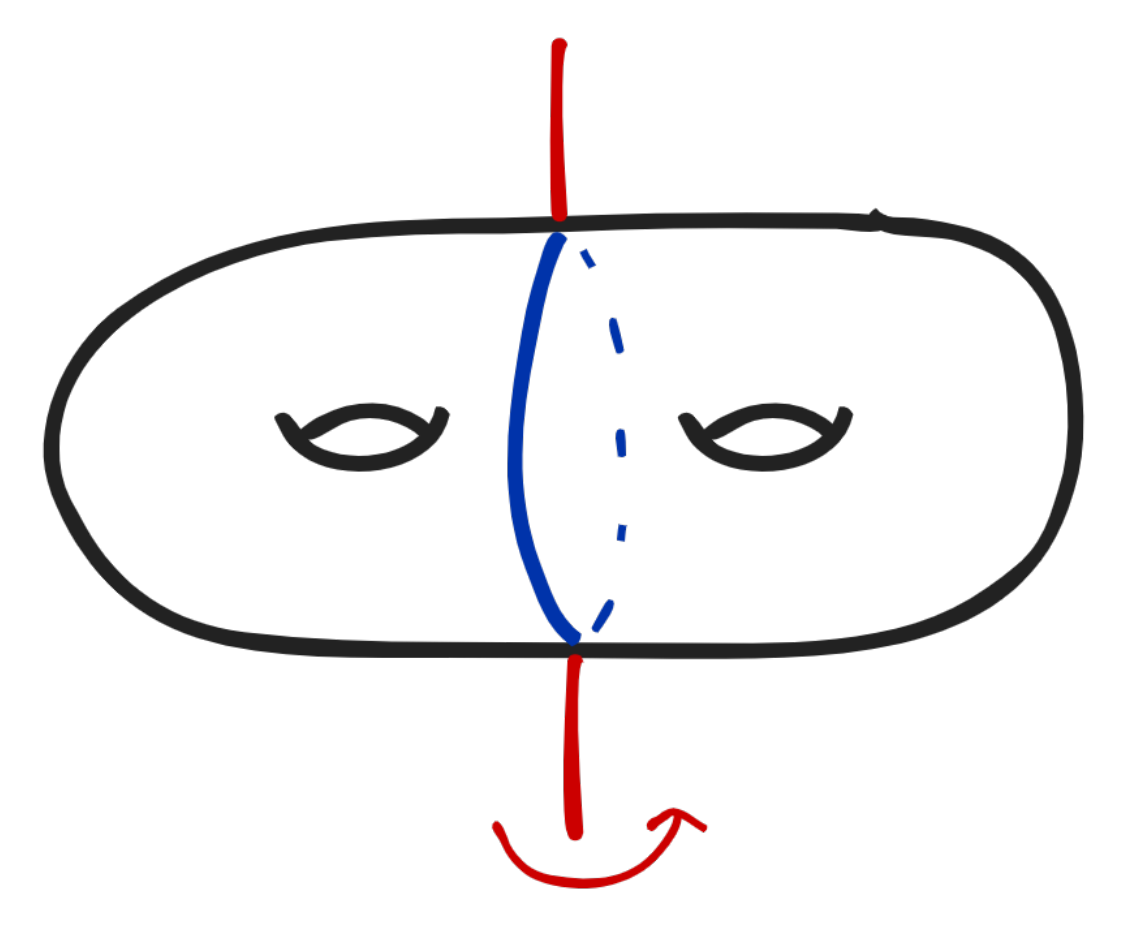}\hspace{.2in}
                \includegraphics[align=c,scale=.2]{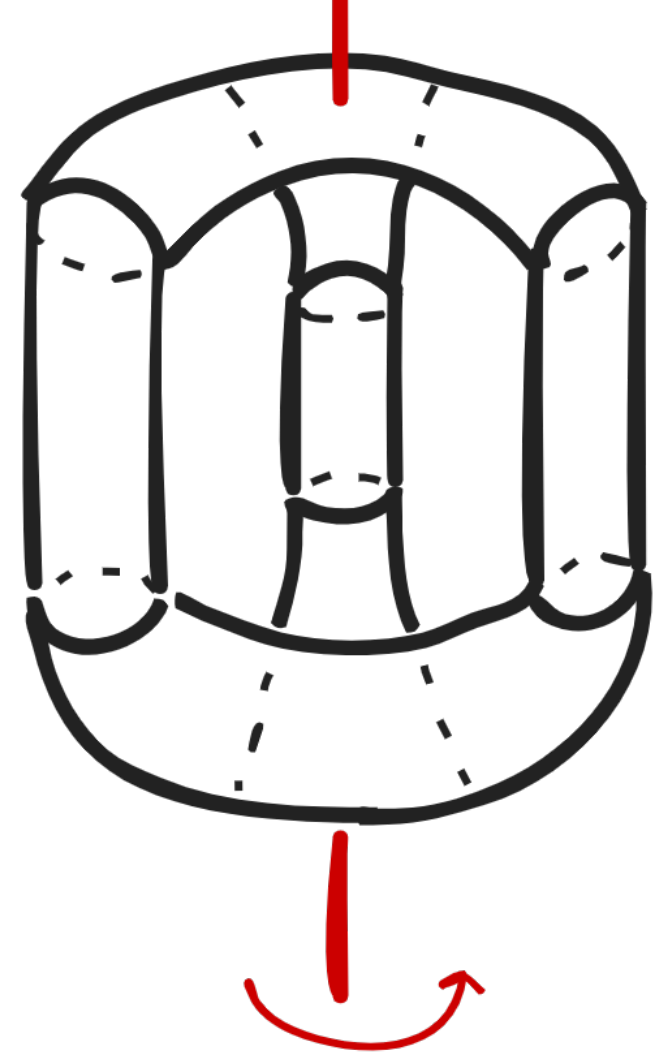}
\caption{Generators $\alpha,\beta,\gamma,\delta$ of the Goeritz group $\bb G$. }
\label{fig:generators}
\end{figure}
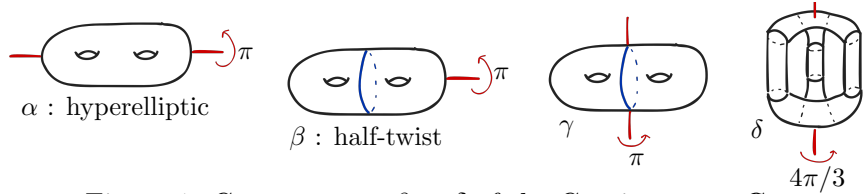

In practice, it is easy to use Theorem \ref{thm:pseudo} and the presentation (\ref{eqn:splitting}) to decide the Nielsen--Thurston type (finite-order, reducible, or pseudo-Anosov) of a mapping class $g\in \bb G$ given as a word in these generators. 

\begin{rmk} The ``figure-8 knot stabilizer" in Theorem \ref{thm:pseudo} comes from the fibering of the figure-8 knot complement
\begin{equation}\label{eqn:fig8-fiber}(T^2\setminus \pt)\ra S^3\setminus K\ra S^1.\end{equation} The union of two fibers of (\ref{eqn:fig8-fiber}) with $K$ is a genus-2 Heegaard surface in $S^3$; see Figure \ref{fig:fig8}. The figure-8 knot stabilizer is the subgroup of the Goeritz group that fixes $K$. For example, the monodromy of (\ref{eqn:fig8-fiber}) yields an infinite-order element of the figure-8 knot stabilizer that is conjugate to $\be\de\be^{-1}\de$. \end{rmk}

\begin{figure}[h!]
\labellist
\endlabellist
\centering
\includegraphics[scale=.5]{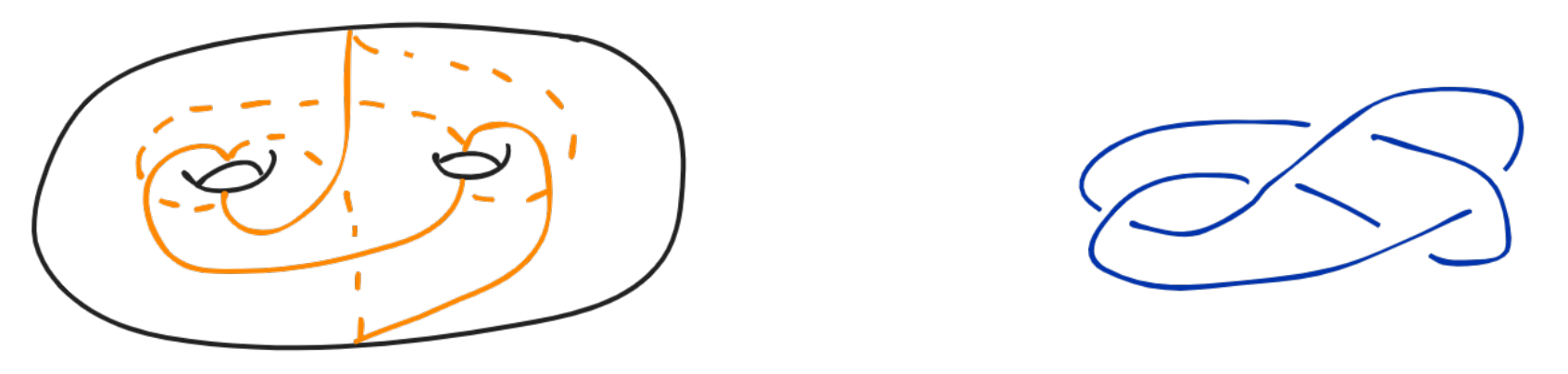}
\caption{The figure-8 knot embedded on the standard genus-2 Heegaard surface.}
\label{fig:fig8}
\end{figure}

\paragraph{Refined Nielsen--Thurston classification.} Theorem \ref{thm:pseudo} also classifies the ways an element of the Goeritz group can be reducible (in terms of what structure is preserved). This provides a refinement of the Nielsen--Thurston classification for reducible elements of the Goeritz group. We illustrate this with the following result. 

\begin{cor}[Classification of canonical reduction systems for Goeritz elements]\label{cor:CRS}
Let $g\in\bb G$ be an infinite-order, reducible element. Then the canonical reduction system $\CRS(g)$ is either
\begin{enumerate}
\item[(i)] (weakly reducing pair) a pair of curves $c$ and $\what c$ that bound primitive disks in $V$ and $\what V$ respectively, 
\item[(ii)] (reducing curve) a curve $c$ that bounds disks in both $V$ and $\what V$, or 
\item[(iii)] (figure-8 curve) a curve $c$ whose embedding in $S^3$ is the figure-8 knot.
\end{enumerate} 
\end{cor}


\paragraph{Explicit convex cocompact examples.} Using Theorems \ref{thm:ccc} and \ref{thm:pseudo}, one may give explicit examples of convex cocompact subgroups of $G<\Mod(S)$. We illustrate this with the following corollary, which is motivated by Problem 3.6 of \cite{mosher-survey} that asks for new constructions/examples.

\begin{cor}[Explicit convex cocompact subgroups]\label{cor:explicit}
Let $\beta,\delta\in\bb G$ be the elements pictured in Figure \ref{fig:generators}. For each $n\ge2$, the subgroup 
\[G_n=\pair{\be^n\de,\de\be^n}<\Mod(S)\] is purely pseudo-Anosov and hence convex cocompact by Theorem \ref{thm:ccc}. 
\end{cor}

\begin{proof}[Proof of Corollary \ref{cor:explicit}]
By the presentation (\ref{eqn:splitting}), the subgroup $\pair{\beta,\delta}\sbs\bb G$ is isomorphic to $\Z*\Z_3$. This implies that nontrivial products of the generators of $G_n$ are cyclically reduced with no cancellation. From this one quickly deduces that no word in $G_n$ is conjugate into any of the subgroups listed in Theorem \ref{thm:pseudo}. Therefore $G_n$ is purely pseudo-Anosov. 
\end{proof}

Explicit examples are rare in the study of convex cocompact subgroups of mapping class groups. More general constructions of convex cocompact groups, e.g.\ using ping-pong as in \cite{mosher-schottky,min}, involve replacing a given collection of pseudo-Anosov elements by large (non-explicit!) powers of these elements. Other explicit examples in genus 2, different from the ones we produce, appear in \cite[\S7]{mangahas-taylor} and \cite[\S6]{whittlesey} (the latter examples are convex cocompact by \cite{KLS}). 

\paragraph{Techniques.} 

The main technical results used in the proofs of Theorems \ref{thm:ccc} and \ref{thm:pseudo} are Theorems \ref{thm:distance} and \ref{thm:holes} below. 

To show a group $G<\Mod(S)$ is convex cocompact, it suffices to show that the orbit map (with respect to some basepoint) of $G$ acting on the curve complex $\ca C(S)$ is a quasi-isometric embedding \cite{kent-leininger-ccc}. For studying subgroups $G<\bb G$, it is convenient to choose the basepoint so that the orbit is contained in the primitive disk complex $\ca P(V)\sbs\ca C(S)$ associated to the Heegaard splitting $S^3=V\cup_S\what V$ (see \S\ref{sec:background} for the definition). Theorem \ref{thm:distance} relates the distance in $\ca P(V)$ to the distance in $\ca C(S)$ and more generally $\ca C(X)$, where $X\subset S$ is a subsurface whose complement does not support a primitive disk; such subsurfaces are called \emph{holes} for $\ca P(V)$; see \S\ref{sec:background}. 

\begin{mainthm}[Distance formula]\label{thm:distance} 
Given $c>0$, there is $K>0$ so that for any two vertices $D,E$ of $\ca P(V)$, 
\begin{equation}\label{eqn:distance}
\frac{1}{K}\sum_{X}\{d_X(D,E)\}_c-K\le d_{\ca P(V)}(D,E)\le K\sum_{X}\{d_X(D,E)\}_c+K,
\end{equation}
where the sums range over subsurfaces $X\sbs S$ that are \emph{holes} for $\ca P(V)$, $d_X(D,E)$ is defined by \emph{subsurface projection}, and $\{\cdot\}_c$ is the \emph{cutoff function}. Definitions appear in \S\ref{sec:background}. 
\end{mainthm}

Distance formulas as in (\ref{eqn:distance}) originate in the work of Masur--Minsky \cite{masur-minsky2}. We arrive at (\ref{eqn:distance}) by closely following the work of Masur--Schleimer \cite{masur-schleimer}, who prove an analogous result for the disk complex $\ca D(V)$ and provide an axiomatic approach to distance formulas more generally. 

To apply Theorem \ref{thm:distance}, it is necessary to classify the holes $X\subset S$ with the property that the image of the subsurface projection $\pi_X:\ca P(V)\rightarrow\ca C(X)$ has infinite diameter. For brevity we say that $X$ \emph{has infinite diameter} (with respect to $\ca P(V)$). 

\begin{mainthm}[Classification of large holes for $\ca P(V)$]\label{thm:holes} 
Let $X\subset S$ be a hole for $\ca P(V)$ with diameter $\ge 61$. Then either $X=S$ or $X$ is a genus-$1$ Seifert surface for the figure-$8$ knot. In either case there exists $g\in\bb G$ that preserves $X$ and so that $\rest{g}{X}$ is pseudo-Anosov. Consequently, every hole of diameter $\ge 61$ has infinite diameter. 
\end{mainthm}

\begin{rmk}[Genus-1 fibered knots]
The proof of Theorem \ref{thm:holes} uses the classification of genus-1 fibered knots in $S^3$. The only such knots are the figure-8 and the (left/right-handed) trefoil \cite{fico}. The appendix contains a modern proof of this result. 
\end{rmk}

\begin{rmk}[Relation to the work of Masur--Schleimer] Our proofs of the distance formula and classification of holes for $\ca P(V)$ (Theorems \ref{thm:distance} and \ref{thm:holes}) build on the work of \cite{masur-schleimer}, who prove similar results for the disk complex $\ca D(V)$. The primitive disk complex $\ca P(V)$ is a subcomplex of $\ca D(V)$. As such, the techniques of \cite{masur-schleimer} are useful for this paper; however, we note that the results of \cite{masur-schleimer} cannot be applied directly to prove our main results, since e.g.\ there are holes for $\ca P(V)$ that are not holes for $\ca D(V)$, and there are holes for $\ca D(V)$ that have infinite diameter with respect to $\ca D(V)$, but finite diameter with respect to $\ca P(V)$. See Remark \ref{rmk:D-vs-P}. Many of the arguments of \cite{masur-schleimer} ``break" when restricting only to primitive disks, and some of the work in the current paper involves finding ways to ``repair" these arguments. 
\end{rmk}

\begin{rmk}[Holes for the reducing sphere complex]
Using our analysis one could also classify holes for the complex $\ca R(V,\what V)$ of reducing spheres for the Heegaard splitting. This problem was posed by Schleimer \cite[\S7]{schleimer-cc}. One might try to use $\ca R(V,\what V)$ to prove Theorem \ref{thm:ccc} (indeed, this is one motivation for Schleimer's problem), but it seems more convenient to work with the primitive disks instead of reducing spheres. For example, it's easier to construct paths between primitive disks using surgery. 
\end{rmk}

\begin{rmk}[The assumption $g=2$]
Our approach to proving each of Theorems \ref{thm:ccc} -- \ref{thm:holes} fails to extend to genus $g\ge3$. A fundamental difference between $g=2$ and $g\ge3$ is that in genus 2, every surgery path between a pair of primitive disks is a path of \emph{primitive} disks, but this is not true in higher genus; see \cite[Thm.\ 1.1]{CKL}.
\end{rmk}
 
\paragraph{About the proof of Theorem \ref{thm:ccc}.} The proof of Theorem \ref{thm:ccc} is given in \S\ref{sec:ccc}. Here we explain the main ideas using the following diagram. 

\[\begin{xy}
(-10,0)*+{\bb G}="A";
(20,0)*+{\Cone(\Ga_{\bb G},\bb G_D)}="B";
(-10,-15)*+{G}="C";
(20,-15)*+{\ca P(V)}="D";
(20,-30)*+{\ca C(S)}="E";
{\ar"A";"B"}?*!/_3mm/{};
{\ar@{^{(}->} "C";"A"}?*!/^5mm/{};
{\ar "C";"B"}?*!/_3mm/{};
{\ar "C";"D"}?*!/_3mm/{};
{\ar "C";"E"}?*!/_3mm/{};
{\ar@{^{(}->} "D";"E"}?*!/_3mm/{};
{\ar@{<->} "B";"D"}?*!/_3mm/{};
\end{xy}\]

Assuming $G<\bb G$ is purely pseudo-Anosov, to show $G$ is convex cocompact in $\Mod(S)$, it suffices to show that the orbit map $G\rightarrow\ca C(S)$ is a q.i.\ embedding \cite{kent-leininger-ccc}. First we use Theorem \ref{thm:distance} (distance formula) and results of \cite{BBKL} to reduce to showing that the orbit map $G\rightarrow\ca P(V)$ is a q.i.\ embedding (see Prop.\ \ref{prop:contains-reducible}). This step is nontrivial because the inclusion $\ca P(V)\hra\ca C(S)$ is not a q.i.\ embedding by Theorems \ref{thm:distance} and \ref{thm:holes}. Next we show, using work of Cho \cite{cho}, that $\ca P(V)$ is quasi-isometric to a cone-off of the Cayley graph $\Ga_{\bb G}$ with respect to the stabilizer $\bb G_D$ of a primitive disk (see Prop.\ \ref{prop:cone}). Finally $G\hra\bb G$ is a quasi-isometric embedding because $G$ is virtually 
free, and to prove that the composition $G\hra\bb G\hra\Cone(\Ga_{\bb G},\bb G_E)$ is a quasi-isometric embedding, we use a criterion of Abbott--Manning \cite{abbott-manning} (see Prop.\ \ref{prop:manning-abbott}). 

\paragraph{About the proof of Theorem \ref{thm:pseudo}.} Given Theorem \ref{thm:holes} (classification of holes), the proof of Theorem \ref{thm:pseudo} is a straightforward argument using the canonical reduction system and some results of Oertel \cite{oertel} about homeomorphisms of handlebodies. The details are in \S\ref{sec:nielsen-thurston}.


\paragraph{Acknowledgements.} The author thanks J.\ Manning for helpful conversations about \cite{abbott-manning} and for observing that $\ca P(V)$ is quasi-isometric to a coned-off Cayley graph. Thanks also to C.\ Leininger for comments on a draft of this paper. 

\section{Background and auxiliary results} \label{sec:background} 

In this section we quickly summarize the necessary terminology and results that we will need in the rest of the paper. At the end of the paper we have included a guide to the notation for easy reference. 


\subsection{Surfaces and curves}

We use $\Sigma_{g,b}$ to denote the compact oriented surface of genus $g$ with $b$ boundary components. A \emph{simple closed curve} (usually referred to simply as a \emph{curve}) $c$ on a surface $\Sigma$ is an embedded circle. A curve $c$ is \emph{essential} if it does not bound a disk and \emph{non-peripheral} if it is not isotopic to a boundary component. We will often use the same notation for a curve and its isotopy class. For a curve $c\sbs\Sigma$, we use $n(c)\subset\Sigma$ to denote a small regular neighborhood of $c$. For a pair of isotopy classes $c_1,c_2$, the \emph{geometric intersection number} is denoted $i(c_1,c_2)$. 


\subsection{Primitive disks and curve complexes} \label{sec:complexes}


Fix a genus-2 Heegaard splitting $S^3=V\cup_S\what V$. Some of the definitions of this section can be stated more generally, but several of the results are specific to genus 2, so we focus on that case. 


\paragraph{Primitive disks.} A properly embedded disk $D\sbs V$ is called \emph{primitive} if there exists a properly embedded disk $\what D\sbs \what V$ so that $i(\pa D,\pa\what D)=1$. 

For example, in Figure \ref{fig:handlebody}, the disks $\what E_1,\what E_2$ illustrate that $E_1,E_2\subset V$ are primitive disks. We call $E_1,E_2$ and $\what E_1,\what E_2$ a \emph{dual pair} (these are the curves for the standard genus-2 Heegaard diagram for $S^3$).

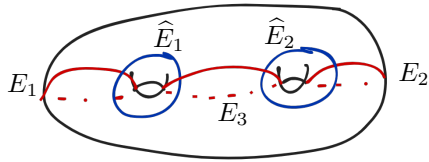
\begin{figure}[h!]
\labellist
\small
\pinlabel $E_1$ at 540 220
\pinlabel $E_2$ at 875 230
\pinlabel $E_3$ at 720 195
\pinlabel $\what E_1$ at 665 262
\pinlabel $\what E_2$ at 760 265
\endlabellist
\centering
\includegraphics[scale=.45]{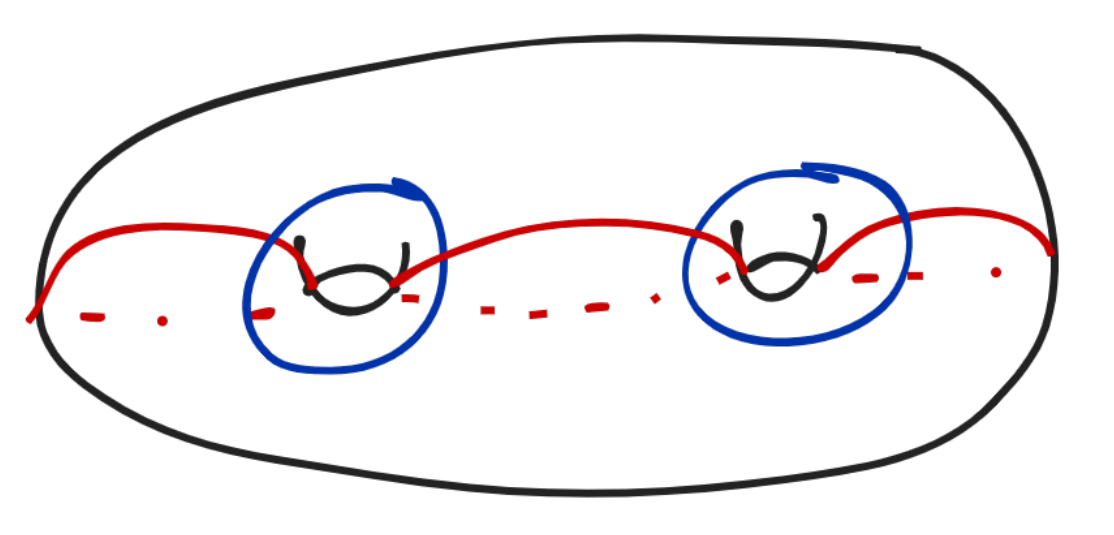}
\caption{Standard Heegaard diagram for the Heegaard splitting $S^3=V\cup_S \what V$. Here $E_1,E_2,E_3$ are disks in $V$ (the ``inside" handlebody), and $\what E_1,\what E_2$ are disks in $\what V$.}
\label{fig:handlebody}
\end{figure}

\begin{rmk}[Primitivity testing]\label{rmk:primitive-test}
A disk $D\subset V$ is primitive if and only if the conjugacy class of $\pa D$ in $\pi_1(\what V)\cong F_2$ is primitive. More precisely, if we orient $\pa D$, $\pa\what E_1$, and $\pa\what E_2$, then the intersection pattern of $\pa D$ with $\what E_1$ and $\what E_2$ determines a word $w_D$ in $x_1$ and $x_2$, say. Then disk $D$ is primitive if and only if the word $w_D$, when cyclically reduced, is part of a free basis for the free group $F_2=\pair{x_1,x_2}$. See \cite{zieschang} and also \cite{gordon}. 





There is an easy algorithm to decide if a word in $F_2$ is primitive \cite{piggott}. A simple obstruction to a cyclically reduced word $w\in F_2$ being primitive is if $w$ contains both $x_1$ and $x_1^{-1}$ (or $x_2$ and $x_2^{-1}$). For example, Figure \ref{fig:primitive} gives an example of a curve $c\sbs S$ that bounds a disk, but does not bound a primitive disk. 

\begin{figure}[h!]
\labellist
\pinlabel $c$ at 210 115
\endlabellist
\centering
\includegraphics[scale=.45]{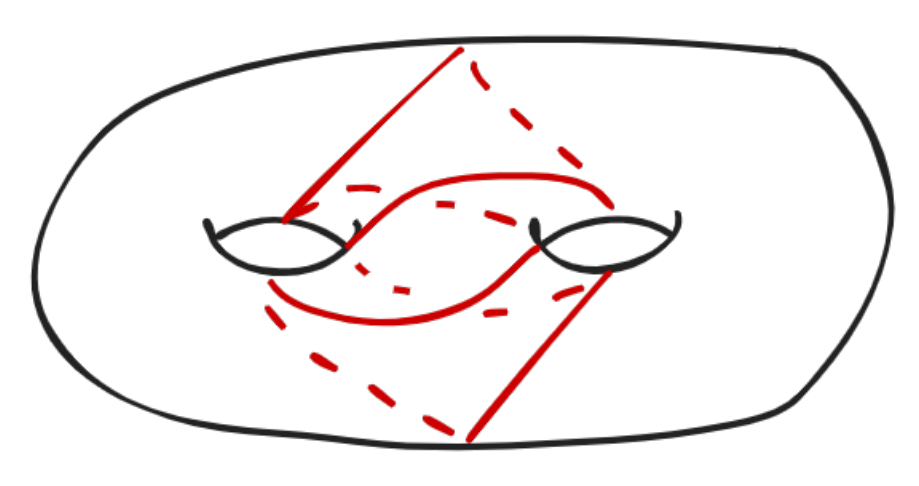}
\caption{The curve $c$ is nonseparating and bounds a disk in $D\subset V$, but $D$ is not primitive. Indeed, (for a particular choice of orientations) the word $w_D$ is conjugate to $x_1x_2^{-1}x_1x_2x_1^{-1}x_2$, which is not primitive.}
\label{fig:primitive}
\end{figure}
\end{rmk}

\paragraph{Primitive disk complex.} For a compact surface $X$, we use $\ca C(X)$ to denote the curve complex, which is the simplicial complex with a vertex for each isotopy class of essential, non-peripheral simple closed curve on $X$, and a $k$-simplex for each $(k+1)$-tuple of isotopy classes that can be realized disjointly. We define the distance $d_{\ca C(X)}(a,b)$ between two vertices of $\ca C(X)$ as the fewest edges in an edge path in $\ca C(X)$ connecting them. 



The key complex of interest in this paper is the \emph{primitive disk complex} $\ca P(V)$, which is the full subcomplex of the curve complex $\ca C(S)$ whose vertices are isotopy classes of simple closed curves on $S$ that bound primitive disks in $V$. This complex was defined in \cite{cho}. Sometimes we abuse notation and refer to $D$ (instead of $\pa D$) as a vertex of $\ca P(V)$. 

By definition, $\ca P(V)$ is a subcomplex of the (more familiar) disk complex $\ca D(V)\sbs\ca C(S)$ whose vertices are curves on $S$ that bound disks in $V$. The disk complex will be helpful for our study of $\ca P(V)$ because analogues of Theorems \ref{thm:distance} and \ref{thm:holes} are known for $\ca D(V)$ \cite{masur-schleimer}. 

To understand $\ca P(V)$, we will use two other complexes: the reducing sphere complex $\ca R(V,\what V)$ and a new complex that we call the Heegaard marking complex $\ca M(V,\what V)$. 

\paragraph{Reducing sphere complex.} The \emph{reducing sphere complex} $\ca R(V,\what V)$ is the subcomplex of $\ca C(S)$ whose vertices are curves that bound disks in both $V$ and $\what V$ (we call these \emph{reducing curves}; the union of the two disks is called a \emph{reducing sphere}). Edges correspond to curves with geometric intersection number 4 (this is the minimal possible intersection since vertices of $\ca R(V,\what V)$ are in particular separating curves on $S$). See Figure \ref{fig:reducing-sphere} for an example. The complex $\ca R(V,\what V)$ is studied in \cite{scharlemann,akbas}. 

\begin{figure}[h!]
\labellist
\small
\pinlabel $P$ at 885 755
\pinlabel $Q$ at 970 580
\endlabellist
\centering
\includegraphics[scale=.25]{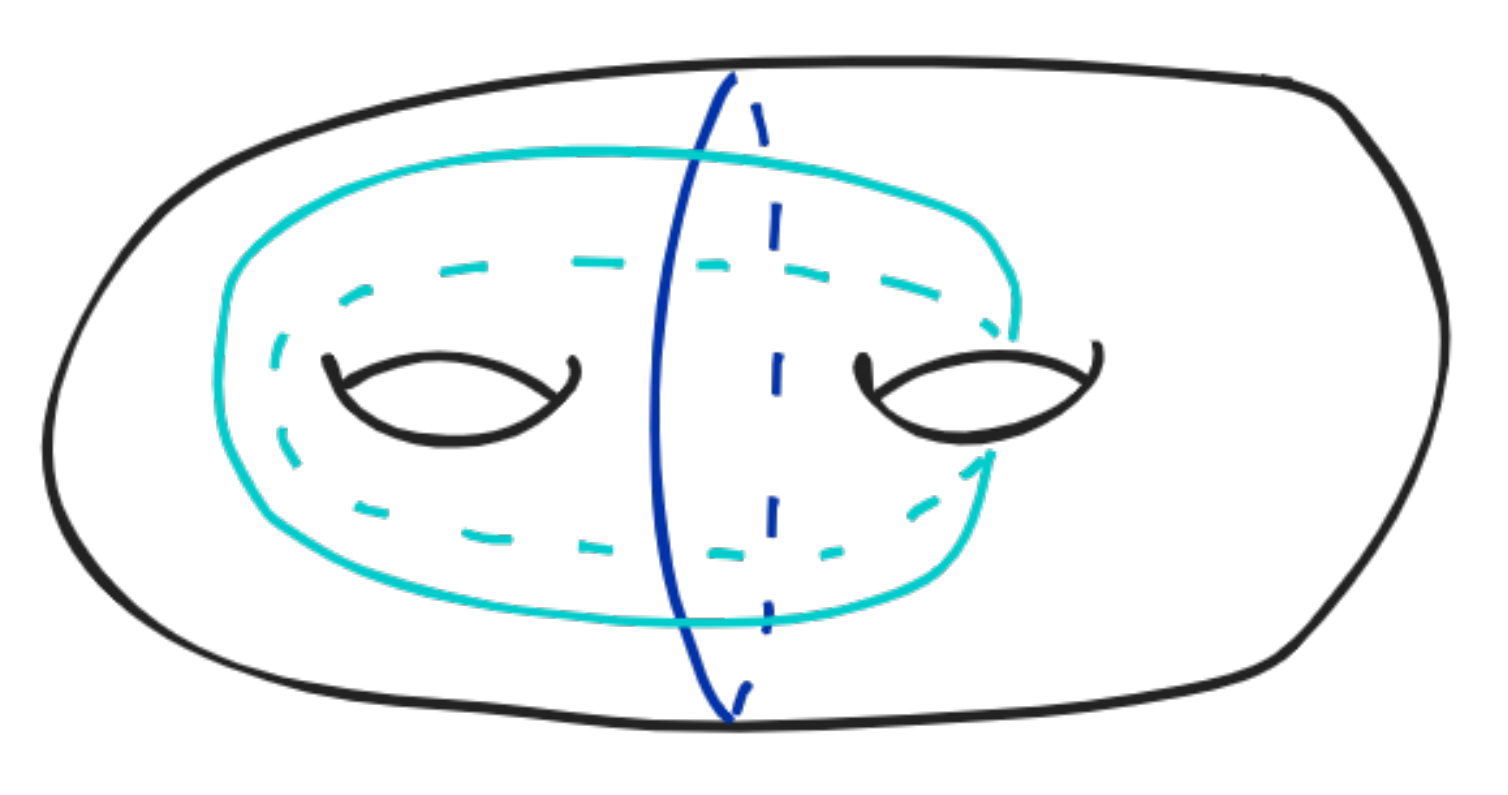}
\caption{Adjacent vertices of $\ca R(V,\what V)$.}
\label{fig:reducing-sphere}
\end{figure}

There is an embedding $\ca R(V,\what V)\hra \ca P(V)$ that appears in \cite{cho}. First observe that there is a bijection between vertices of $\ca R(V,\what V)$ and edges of $\ca P(V)$: given a reducing sphere $R$, there is a unique pair of disjoint primitive disks $D_1^R,D_2^R\sbs V$ that are disjoint from $R$; conversely, a pair of disjoint primitive disks $D_1,D_2$ in $V$ determines a unique reducing curve disjoint from $D_1\cup D_2$ by \cite[Lem.\ 2.2]{cho}. Then there is a map from vertices of $\ca R(V,\what V)$ to $\ca P(V)$ that sends $R$ to the midpoint of the edge $\{D_1^R,D_2^R\}$. As explained in \cite[\S6]{cho}, this map extends to an embedding $\ca R(V,\what V)\hra\ca P(V)$, where each triangle of $\ca R(V,\what V)$ maps into a unique triangle of $\ca P(V)$, as pictured in Figure \ref{fig:primitive-reducing}. 

\begin{figure}[h!]
\labellist
\endlabellist
\centering
\includegraphics[scale=.3]{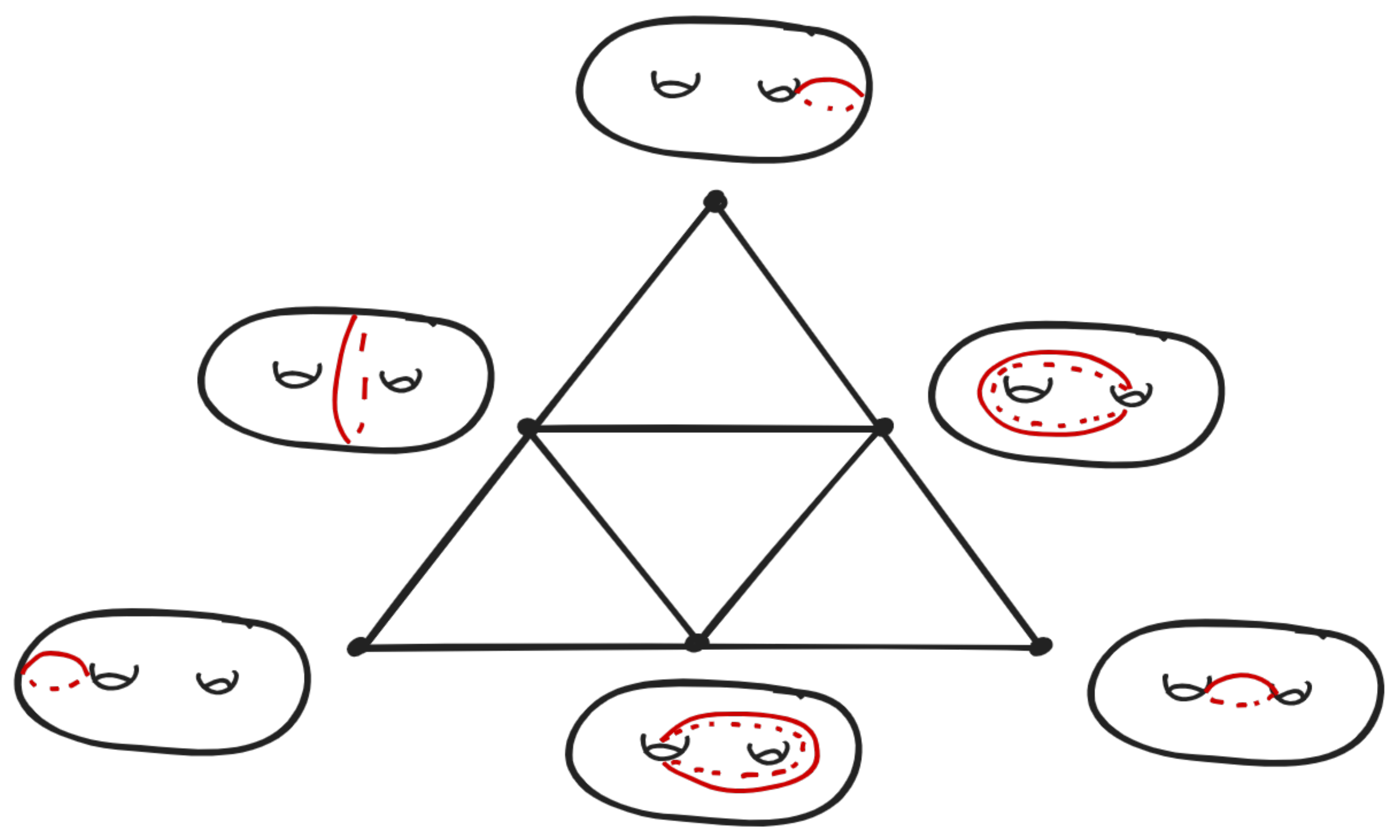}
\caption{A triangle of $\ca R(V,\what V)$ embedded in a triangle of $\ca P(V)$.}
\label{fig:primitive-reducing}
\end{figure}

\begin{rmk}\label{rmk:cone}
Using the preceding paragraph, we find that $\ca P(V)$ is obtained from $\ca R(V,\what V)$ by adding, for each primitive disk $D$, a vertex $v_D$ and the cone $v_D*\ca R_D$, where $\ca R_D\sbs\ca R(V,\what V)$ is the subgraph spanned by reducing spheres that are disjoint from $D$. 
\end{rmk}

\paragraph{Heegaard marking complex.} We define the \emph{Heegaard marking complex} $\ca M(V,\what V)$ as the graph with
\begin{itemize}
\item a vertex for each pair $(R,D)$, where $R$ is a reducing curve, and $D\sbs V$ is a primitive disk that is disjoint from the primitive disks $D',D''\subset V$ that are disjoint from $R$.
\item an edge between pairs of vertices if they differ by one of the following moves: 
\begin{itemize}
\item twist: replace $(R,D)$ by $(R,\beta_R(D))$, where $\beta_R\in\bb G$ is a half-twist about $R$ (analogous to $\beta$ in Figure \ref{fig:generators}).
\item 3-cycle: replace $(R,D)$ by $(\rho(R),\rho(D))$, where $\rho\in\bb G$ is the order-3 element defined as follows. Let $D',D''\subset V$ be the two primitive disks that are disjoint from $R$. Let $\rho\in\bb G$ be a mapping class that cyclically permutes $D,D', D''$ and preserves the two components $S\setminus (\pa D\cup\pa D'\cup\pa D'')$ (analogous to $\delta$ in Figure \ref{fig:generators}). 
\end{itemize} 
\end{itemize}

A vertex of $\ca M(V,\what V)$ naturally determines a clean, complete marking in the sense of \cite[\S2.5]{masur-minsky2}. To see this, fix a vertex $(R,D)$ of $\ca M(V,\what V)$. Let $D',D''\subset V$ and $\what D',\what D''\subset\what V$ be the primitive disks that are disjoint from $R$ (with the respective pairs $D',\what D'$ and $D'',\what D''$ dual). Then 
\[\mu=\{(R,D), (D',\what D'), (D'',\what D'')\}\] 
is a clean, complete marking. We refer to either $(R,D)$ or $\mu$ as a \emph{Heegaard marking}.

Recall \cite[\S2.5]{masur-minsky2} that the marking complex $\ca M(S)$ has vertices for clean, complete markings, and a vertex for markings that differ by a twist or a flip. 

\begin{lem}\label{lem:marking-comparison}
Fix adjacent vertices $(R_1,D_1)$ and $(R_2,D_2)$ of $\ca M(V,\what V)$, and let $\mu_1,\mu_2$ be the corresponding vertices of $\ca M(S)$. Then $d_{\ca M(S)}(\mu_1,\mu_2)\le 2$. 
\end{lem}

\begin{proof}
If $(R_1,D_1)$ and $(R_2,D_2)$ differ by a twist, then $\mu_1$ and $\mu_2$ also differ by a twist, so they are adjacent in $\ca M(S)$. 

Suppose that $(R_1,D_1)$ and $(R_2,D_2)$ differ by a 3-cycle move. Up  to homeomorphism, we can assume that the pairs are the ones pictured in Figure \ref{fig:3cycle}. Then $\rho$ is the standard generator $\delta$. Figure \ref{fig:marking-path} illustrates a path of length-2 between $\mu_1$ and $\mu_2$ in $\ca M(S)$. 
\end{proof}

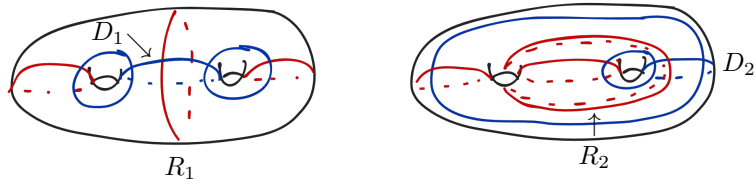
\begin{figure}[h!]
\labellist
\small
\pinlabel $R_1$ at 270 790
\pinlabel $\searrow$ at 230 908
\pinlabel $D_1$ at 200 920
\pinlabel $D_2$ at 810 890
\pinlabel $\uparrow$ at 670 832
\pinlabel $R_2$ at 670 800
\endlabellist
\centering
\includegraphics[scale=.4]{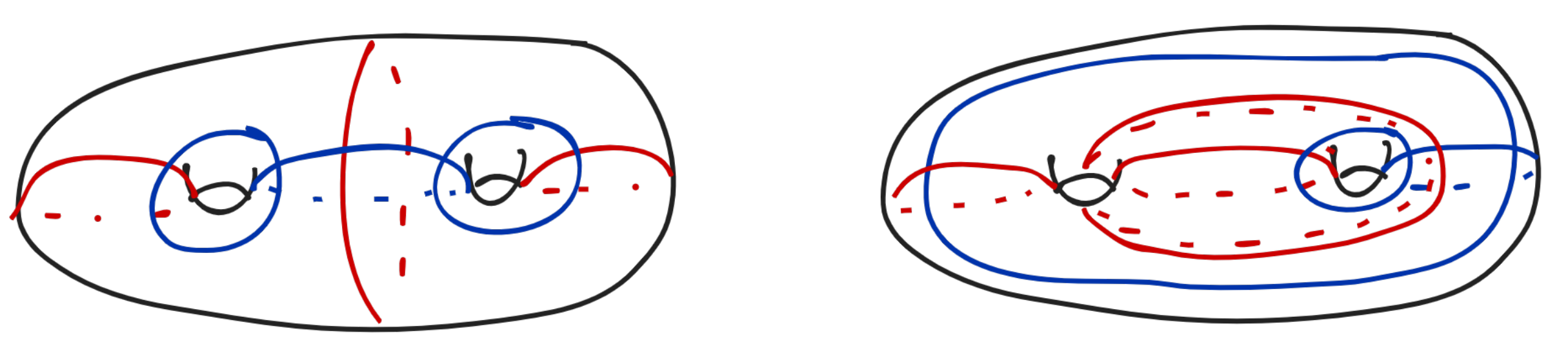}
\caption{Heegaard markings that differ by a 3-cycle move. }
\label{fig:3cycle}
\end{figure}

\begin{figure}[h!]
\labellist
\small
\pinlabel $\boxed{1}$ at 130 580
\pinlabel $\text{flip}$ at 450 700
\pinlabel $\text{clean}$ at 820 700
\pinlabel $\boxed{3}$ at 1100 570
\pinlabel $\text{flip}$ at 230 520
\pinlabel $\text{clean}$ at 620 525
\pinlabel $\boxed{5}$ at 990 450
\endlabellist
\centering
\includegraphics[scale=.35]{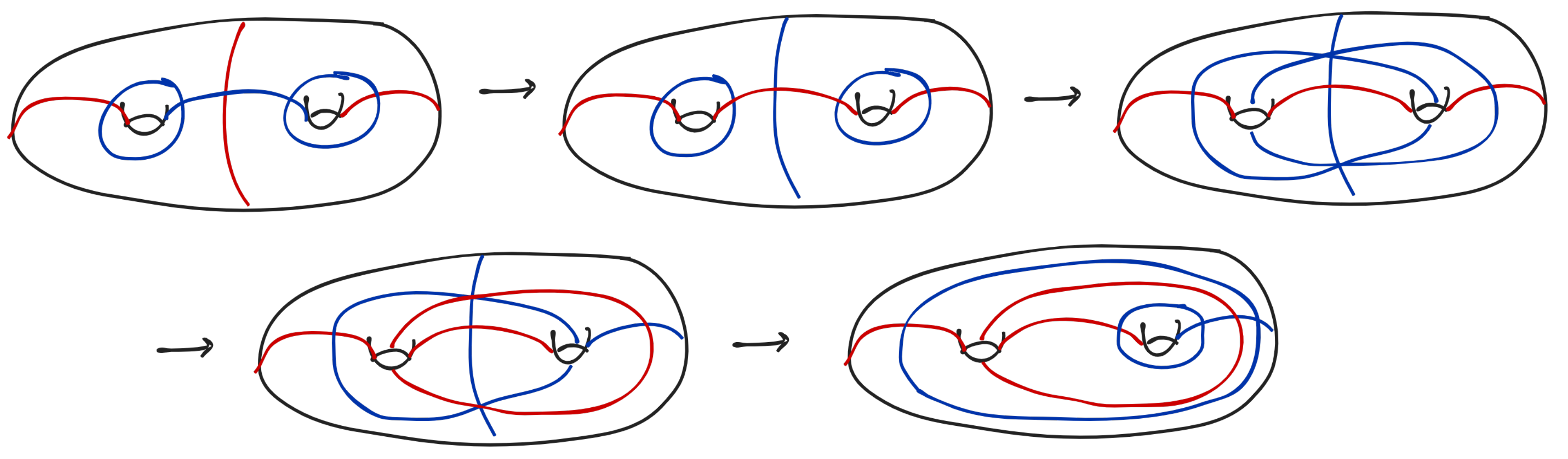}
\caption{Path $\boxed{1}\ra\boxed{3}\ra\boxed{5}$ of complete, clean markings obtained by doing two flip moves. After performing a flip move, and additional step is needed to replace the resulting marking by a clean marking, c.f.\ \cite[\S2.5]{masur-minsky2}.}
\label{fig:marking-path}
\end{figure}

\begin{lem}\label{lem:marking-qi-G}
$\bb G$ is quasi-isometric to $\ca M(V,\what V)$. 
\end{lem}

\begin{proof}
It is easy to see that $\bb G$ acts transitively on vertices of $\ca M(V,\what V)$. The stabilizer of the standard pair $(P,E_3)$ (Figure \ref{fig:handlebody} and Figure \ref{fig:reducing-sphere}) is the subgroup $\pair{\alpha,\gamma}\sbs\bb G$. Since $\pair{\alpha,\gamma}\cong\Z/2\Z\times\Z/2\Z$ is finite, the lemma follows. 
\end{proof}

\paragraph{Goeritz group action.} The Goeritz group $\bb G$ obviously acts on all of our complexes. The following computation will be used in \S\ref{sec:ccc}.

\begin{prop}[Primitive disk stabilizer]\label{prop:primitive-stabilizer} 
Let $E=E_2$ be the primitive disk pictured in Figure \ref{fig:handlebody}. The stabilizer $\bb G_E<\bb G$ is the subgroup generated by $\alpha,\beta,\gamma\delta$. 
\end{prop}

\begin{proof}
The following proof is based on the computation in  \cite{scharlemann} of a generating set for $\bb G$. Let $\ca R_E\sbs\ca R(V,\what V)$ denote the subcomplex spanned by reducing spheres that are disjoint from $E$. Let $P,Q$ be the reducing curves pictured in Figure \ref{fig:reducing-sphere}. 

We prove the proposition by showing that (1) $\ca R_E$ is connected, (2) $\bb G_E$ acts transitively on vertices and edges of $\ca R_E$, and (3) $\bb G_E$ contains an edge inversion $P\leftrightarrow Q$. Once we show this, it is a basic result of geometric group theory that $\bb G_E$ is generated by the stabilizer of $P$ and the edge inversion \cite[Lem.\ 4.10]{farb-margalit}. The stabilizer of $P$ in $\bb G$ is $\pair{\alpha,\beta,\gamma}\cong(\Z_2\ti\Z)\rt\Z_2$; see \cite[\S2]{scharlemann}. It is easy to see that the intersection of this group with $\bb G_E$ is $\pair{\alpha,\beta}$. In addition $\gamma\delta$ inverts $P$ and $Q$, so it only remains to show that $\ca R_E$ is connected, and that $\bb G_E$ acts transitively on its vertices and edges.

One can show that $\ca R_E$ is connected in the same way that Scharlemann shows that $\ca R(V,\what V)$ is connected. Given $R\in\ca R_E$, Scharlemann \cite[\S3]{scharlemann} gives a surgery procedure that replaces $R$ by an adjacent curve with fewer intersections with $P$. This construction does not create intersections with $E$. (The main input of the surgery operation is an arc of $R$ on one side of $P$ whose slope (defined in \cite[\S3]{scharlemann}) is $\infty$ and an arc on the other side that has slope 0 and is disjoint from $R$. By our choice of $E$, the arc of slope $\infty$ is not on the same side of $P$ as $E$. Since $E$ is disjoint from arcs of slope 0, the surgery procedure does not create any new intersections with $E$. See \cite{scharlemann} for details.)

Next we show that $\bb G_E$ acts transitively on vertices $\ca R_E$. Fix a vertex $R\in\ca R_E$ we show that there is $k\in\bb G_E$ with $k(P)=R$. By Scharlemann's result, there exists $g\in\bb G$ with $g(P)=R$. There are two possibilities for $g(E)$ (each of the solid tori components of $V\setminus R$ contains exactly one primitive disk). If $g(E)\neq E$, then $(g\circ\gamma)(E)=E$, and so either $k=g$ or $k=g\circ\gamma$ has the desired properties $k(E)=E$ and $k(P)=R$ (note that $\ga(P)=P$). The fact that $\bb G_E$ acts transitively on edges follows by observing that $\beta\in \bb G_E$ acts transitively on reducing spheres that are adjacent to $P$ in $\ca R_E$. 
\end{proof}

To summarize, we have a sequence of forgetful maps
\[\ca M(V,\what V)\ra \ca R(V,\what V)\ra\ca P(V),\]
and each is distance non-increasing. Up to quasi-isometry, each map is the inclusion of a space in a coned-off space. Identifying $\ca M(V,\what V)$ with $\bb G$ (up to quasi-isometry), the space $\ca R(V,\what V)$ is quasi-isometric to the cone off of $\bb G$ by cosets of $\bb G_P=\pair{\alpha,\beta,\gamma}$, and the space $\ca P(V)$ is quasi-isometric to the cone off of $\bb G$ by cosets of $\bb G_E=\pair{\alpha,\beta,\gamma\delta}$. Since a finite-index subgroup of $\bb G_P$ is contained in $\bb G_E$, we can also view $\ca P(V)$ as a cone off of $\ca R(V,\what V)$, up to quasi-isometry. This is discussed further in \S\ref{sec:ccc}. 

To end this section, we record the following lemma for later use. The proof is elementary. 
\begin{lem}\label{lem:contraction}
Let $G$ be a finitely generated group. Fix complexes $\ca X$ and $\ca Y$ with isometric $G$ actions, and assume there is a distance non-increasing $G$-equivariant map $\pi:\ca X\ra\ca Y$. If the orbit map $G\ra\ca Y$ is a quasi-isometric embedding, then so is the orbit map $G\ra\ca X$. 
\end{lem} 


\subsection{Surgery of primitive disks}\label{sec:surgery}

In this section we explain the construction of surgery paths in $\ca P(V)$. 

\paragraph{Surgery.} Primitive disks are manipulated by surgery along bigons (boundary compressions). Some care must be taken to ensure that the disks resulting from surgery are primitive. Below we mostly follow the terminology in \cite[\S8]{masur-schleimer}.

Fix a primitive disk $D$. A \emph{surgery bigon} for $D$ is a triple $(B,a,b)$, where $B\subset V$ is a closed disk whose boundary is decomposed into two arcs $a,b$ with $a=B\cap D$, $b=B\cap S$, and $B\cap\pa D=\pa a=\pa b$. Sometimes we refer to the surgery bigon as $B$, leaving $a$ and $b$ implicit. 

To surger $D$ along $(B,a,b)$, observe that $a$ decomposes $D$ into closed disks $D=D_1\cup D_2$ whose intersection is $a$. The disks obtained from $D$ by surgery along $B$ are $D'=D_1\cup B$ and $D''=D_2\cup B$. After an isotopy, $D,D',D''$ are pairwise disjoint. Furthermore, $D',D''$ are not isotopic, since otherwise $\pa D\sbs S$ would be separating (hence not primitive). 

A surgery bigon $(B,a,b)$ for $D$ is called a \emph{boundary compression} if $b$ is an essential arc in $S\setminus\pa D$. This implies that $\pa D'$ and $\pa D''$ are essential in $S$. 

We call a boundary compression $(B,a,b)$ for $D$ \emph{primitive} if there exists $\what D_b\sbs \what V$ so that $\pa D,\pa \what D_b$ are disjoint and $b,\pa \what D_b$ intersect transversely in a single point. If $B$ is a primitive boundary compression, then $D',D''$ are both primitive disks; this follows from the definitions. Below we will only ever be interested in primitive boundary compressions.

\begin{figure}[h!]
\labellist
\pinlabel $E$ at 59 735
\pinlabel $B$ at 190 680
\pinlabel $b$ at 160 745
\pinlabel $a$ at 59 760
\pinlabel $E_1$ at 400 725
\pinlabel $E_2$ at 457 725
\endlabellist
\centering
    \includegraphics[scale=.6]{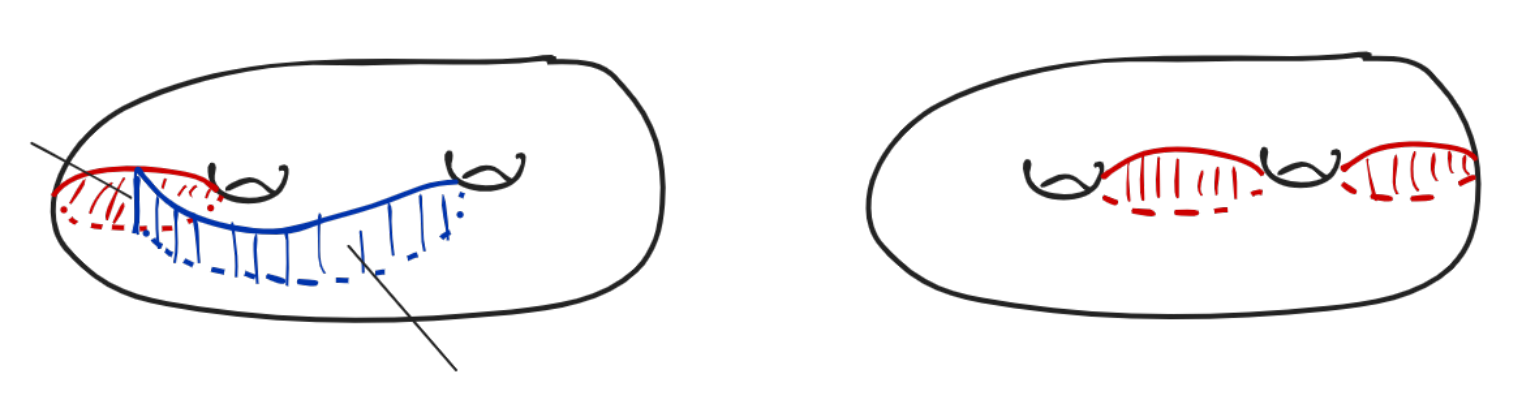}
\caption{Left: a primitive disk $E$ and a surgery bigon $(B,a,b)$. Right: the result of surgery is two primitive disks $E_1,E_2$.} 
\label{fig:surgery}
\end{figure}
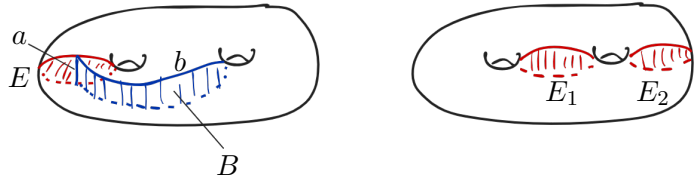

\begin{example}\label{ex:outermost}
The following example gives a basic but important construction of primitive boundary compressions. Fix primitive disks $D,E\sbs V$ with $i(\pa D,\pa E)\neq0$. We assume $\pa D,\pa E$ are in minimal position and $D,E$ intersect transversely. Since $V$ is irreducible, we can isotope $D,E$ so that the 1-manifold $E\cap D$ has no closed components, i.e.\ each component of $D\cap E$ is a proper arc. An outermost arc $a$ of $D\cap E\subset E$ cuts off a bigon $(B,a,b)$, which is always a primitive boundary compression for $D$ by \cite[Thm.\ 2.3]{cho}. The surgered disks $D',D''$ satisfy 
\[i(\pa D',\pa E)+i(\pa D'',\pa E)<i(\pa D,\pa E),\] which implies that neither of $D',D''$ is isotopic to $D$. Furthermore, neither of $D',D''$ is isotopic to $E$ because $D',D''$ are disjoint from $D$ and $i(\pa D,\pa E)\neq0$. 
\end{example}

\begin{rmk}\label{rmk:multidisk} We can repeat the above discussion with a single primitive disk $D$ replaced by a collection $\Delta\sbs S$ of $k\ge1$ disjoint primitive disks. A primitive boundary compression for $\Delta$ is a primitive boundary compression $B$ for some component $D\sbs \Delta$ such that $B$ is disjoint $\Delta\setminus D$. Surgering $D$ along $B$ results in a collection $\Delta'$ of $k+1$ disjoint primitive disks. Furthermore, generalizing Example \ref{ex:outermost}, if $E$ is a primitive disk with nonempty minimal intersection with $\De$, then an outermost bigon of $E\setminus (\De\cap E)$ is a primitive boundary compression for $\Delta$, and $i(\Delta',E)<i(\Delta,E)$.  
\end{rmk}

\paragraph{Subsurfaces and surgery sequences.} A subsurface $X\subset S$ is \emph{essential} if every component of $\pa X$ is essential in $S$. By convention, we will always assume our subsurfaces are connected. 

We say a primitive disk $D$ is \emph{supported in} $X$ if $\pa D\sbs X$ and $\pa D$ is not parallel to a component of $\pa X$. More generally, $D$ \emph{cuts} $X$ if $\pa D$ cannot be isotoped to be disjoint from $X$. 

If there is a primitive disk supported in $X$, we say $X$ is \emph{primitively compressible}. Otherwise, $X$ is \emph{primitively incompressible}. Similarly, if $D$ cuts $X$ and there exists a primitive boundary compression $(B,a,b)$ for $D$ with $b\sbs X$, then we say that $D$ is \emph{primitively boundary compressible} into $X$. 

\begin{defn}[Surgery sequences]
Fix a subsurface $X\sbs S$ and primitive disk $D$, and assume $\pa X,\pa D$ are in minimal position and have nonempty intersection. A \emph{primitive compression sequence} is a sequence $\{\De_k\}_{k=1}^n$, where (i) $\De_1=\{D\}$, (ii) $\De_{k+1}$ is a collection of $k+1$ disjoint disks obtained from $\De_k$ by a primitive boundary compression supported in $S\sm \pa X$, and (iii) for each $k$, each $D_k\in \De_k$ intersects $\pa X$ nontrivially. 
\end{defn}

\begin{rmk}\label{rmk:surgery-end}
If $S\setminus n(\pa X)$ supports a primitive disk $E$ and $\{\De_k\}_{k=1}^n$ is a surgery sequence for $X$, then either (i) $E$ is disjoint from $\De_n$ or (ii) $\De_n$ can be primitively boundary compressed into $S\setminus n(\pa X)$; cf.\ Example \ref{ex:outermost} and Remark \ref{rmk:multidisk}. This motivates the following definition, which agrees with \cite[Definition 11.2]{masur-schleimer}. 
\end{rmk}

\begin{defn}\label{defn:maximal}
A primitive compression sequence $\{\De_k\}_{k=1}^n$ for $X,D$ is \emph{maximal} if either $X$ supports a primitive disk that is disjoint from $\De_n$ or there is no primitive boundary compression for $\De_n$ supported in $S\setminus n(\pa X)$. In the former case, the surgery sequence is said to \emph{end in $S\setminus n(\pa X)$} and in the latter case it is said to \emph{end essentially}. 
\end{defn}

Whether or not a maximal primitive compression sequence for $(X,D)$ ends essentially or ends in $S\setminus n(\pa X)$ depends on whether or not $S\setminus \pa X$ is primitively compressible. By Remark \ref{rmk:surgery-end}, if $S\setminus n(\pa X)$ is primitively compressible, then every maximal surgery sequence ends in $S\setminus n(\pa X)$, and the converse is also true, trivially. Then if $S\setminus \pa X$ is primitively incompressible, and $\{\De_k\}_{k=1}^n$ is a maximal primitive compression sequence, then $\De_n$ cannot be primitively boundary compressed into $S\setminus n(\pa X)$. 

Any surgery sequence can be extended to a maximal sequence. If $\{\De_k\}_{k=1}^n$ is not maximal then, there exists a primitive boundary compression for $\De_n$ supported in $S\setminus n(\pa X)$, which produces a collection $\De_{n+1}$. Since the given sequence is not maximal, every disk in $\De_{n+1}$ intersects $\pa X$, which extends the sequence. Since each $D_k\in\De_k$ intersects $\pa X$ and the total intersection $i(\De_k,\pa X)$ is bounded above by $i(\pa D,\pa X)$, every surgery sequence is finite and can be extended to a maximal sequence.

\paragraph{General position.}\label{rmk:standard-position}

For various arguments it is helpful to put curves on $S$ or disks in $V$ in a ``general position". Precise notions follow. 

Recall that simple closed curves $a,b\sbs S$ are said to be in \emph{minimal position} if they meet transversely and $S\setminus(a\cup b)$ has no bigon components. Equivalently, $a,b$ realize the geometric intersection number for the corresponding isotopy classes \cite[\S1.2]{farb-margalit}.  

Disks $D,E\sbs V$ that intersecting transversely can be isotoped so that $D\cap E$ is a union of arcs. This is because $V$ is irreducible: if $c\subset D\cap E$ is a circle, then $c$ bounds a disk in both $D$ and $E$, giving a 2-sphere in $V$ that can be filled by a 3-ball; using these 3-balls, circular components of $D\cap E$ can be removed by a sequence of isotopies. Then we say that $D\cap E$ is \emph{circle free}.

If given simple closed curves $a,b\sbs S$ and a subsurface $X\sbs S$, we can isotope $a,b$ so that $a,b,\pa X$ are (pairwise) in minimal position. Then $S\setminus (a\cup b\cup\pa X)$ generally will contain triangular components, and we can always isotope further so that these triangle components are all in $X$ or all in $S\setminus X$ (the same observation appears in \cite[\S10]{masur-schleimer}). 

Combining the preceding paragraphs, we say $D,E,X$ are in \emph{standard position} if (1) each pair of $\pa D,\pa E,\pa X$ are in minimal position, (2) the triangles of $S\setminus(\pa D\cup \pa E\cup\pa X)$ are either all in $X$ or all in $S\setminus X$, and (3) $D\cap E$ is circle free. Any triple $D,E,X$ can be put in standard position. 

%

\subsection{Subsurface projection}\label{sec:subsurface-projection}

In this section we recall the terminology needed to state the distance formula.

\paragraph{Subsurface projections, holes, and diameter.} We use subsurface projections to understand the geometry of $\ca P(V)$. For a subsurface $X\sbs S$ and a curve $a\sbs S$, the \emph{subsurface projection} $\pi_X(a)\subset\ca C(X)$ is defined as the union of nontrivial isotopy classes represented by the boundary of $n\big((a\cap X)\cup \pa X\big)$. 
See Figure \ref{fig:projection} for an example. 

\begin{figure}[h!]
\labellist
\pinlabel $a$ at 100 600
\pinlabel $b$ at 180 550
\pinlabel $\pi_{X}(a)$ at 310 600
\pinlabel $\pi_Y(a)$ at 540 600
\pinlabel $b$ at 435 550
\endlabellist
\centering
    \includegraphics[scale=.6]{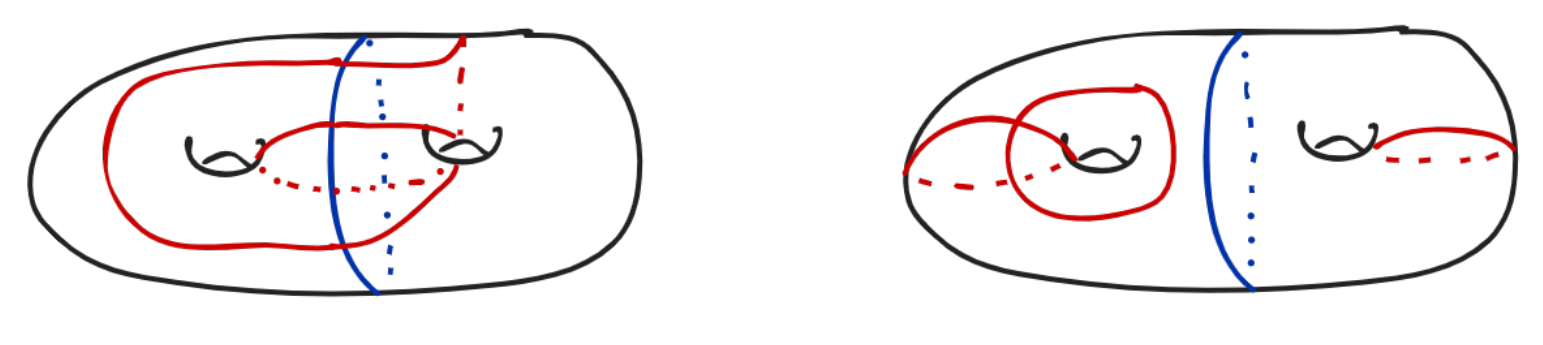}
\caption{A curve $a$ and its projections to the subsurfaces $X$ and $Y$ on the left and right sides of the curve $b$.}
\label{fig:projection}
\end{figure}
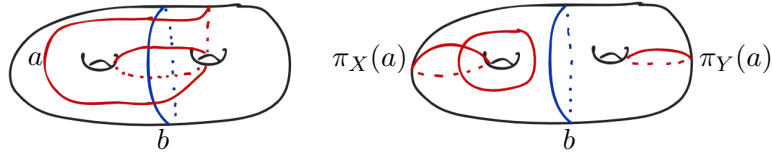

For curves $a,b\sbs S$, the \emph{subsurface-projection distance} is defined as
\[d_X(a,b)=\diam_{\ca C(X)}(\pi_X(a)\cup\pi_X(b).\]

For $X\cong\Si_{0,2}$ it is useful to (re)define $\ca C(X)$ to be the arc complex. This requires a different definition of the subsurface projection; see \cite[\S2.4]{masur-minsky2} or \cite[\S4.1]{masur-schleimer}. 

The function $d_X$ is not a distance function, since it isn't necessarily true that $d_X(a,a)=0$; on the other hand, $d_X$ is symmetric and satisfies the triangle inequality. 
If $d_X$ is bounded on $\ca P(V)$, then its maximum value is called the \emph{diameter} of $X$ (with respect to $\ca P(V)$). If $d_X$ is unbounded, we say $X$ has \emph{infinite diameter}.

One says $X$ is \emph{hole} for $\ca P(V)$ if every primitive disk $D\sbs V$ cuts $X$. 

\begin{rmk}\label{rmk:D-vs-P} One can similarly define holes for the disk complex $\ca D(V)$ and the diameter for $X\sbs S$ with respect to $\ca D(V)$. Since $\ca P(V)\sbs\ca D(V)$, the following statements follow from the definitions. 
\begin{itemize}
\item If $X$ is a hole for $\ca D(V)$, then $X$ is a hole for $\ca P(V)$.
\item If $X$ has infinite diameter with respect to $\ca P(V)$, then $X$ has infinite diameter with respect to $\ca D(V)$. 
\end{itemize} 
The converses of these statements are not true. For example, if $D\in\ca D(V)$ is not primitive and $\pa D$ is nonseparating, then $X=S\setminus n(\pa D)$ is a hole for $\ca P(V)$, but not a hole for $\ca D(V)$. Regarding the diameter, there exists a genus-1 surface $X\subset S$ that is a Seifert surface for an embedding of the trefoil knot, and $X$ has infinite diameter with respect to $\ca D(V)$, but finite diameter with respect to $\ca P(V)$ (see \S\ref{sec:trefoil}).  
\end{rmk}

\paragraph{Subsurface projection of a surgery sequence.} For the disks in a primitive compression sequence for $(X,D)$, the subsurface projections to $X$ do not change much. By \cite[Lem.\ 11.5]{masur-schleimer}, if $\{\De_k\}$ is a compression sequence for $(X,D)$, then there exists $D_n\in\Delta_n$ and arcs $a\sbs \pa D\cap X$ and $b\sbs\pa D_n\cap X$ so that $a$ and $b$ are disjoint.\footnote{The statement of \cite[Lem.\ 11.5]{masur-schleimer} includes the hypothesis that the sequence is maximal, but this is not used in the proof. This is noteworthy because our surgery sequences can be maximal as sequences of primitive disks without being maximal as a sequence of all disks.} As a consequence, one obtains the following result, proved in \cite[Lem.\ 11.7]{masur-schleimer}. 

\begin{lem}\label{lem:surgery-sequence}
Fix a subsurface $X\subset S$, and a primitive disk $D$ that cuts $X$. Then there exists a primitive disk $E$ such that $d_X(D,E)\le 6$ and either $E$ supported in $X$ (if $X$ is primitively compressible) or $E$ cannot be primitively compressed into $S\setminus n(\pa X)$ (if $X$ is primitively incompressible). 
\end{lem}

\subsection{$I$-bundles} \label{sec:I-bundles}

An $I$-bundle $T\xra{p} B$ is a fiber bundle with fiber $I=[0,1]$. In this paper $B$ will always be a compact surface $B$. The \emph{horiztonal boundary} $\pa_hT\sbs T$ is the union of endpoints of each fiber, i.e.\ it is the total space of the induced bundle with fiber $\partial I$; the restriction $p:\pa_hT\ra B$ is a 2-fold covering map. The \emph{vertical boundary} $\pa_vT$ is defined as $p^{-1}(\pa B)$. A subset of $T$ is called \emph{vertical} if it is a union of fibers. 

If $T$ is orientable and $\pa_hT$ is connected, then the base $B$ is non-orientable. Indeed, if $\pa_hT$ is connected, then there exists a curve $c\sbs B$ so that $p^{-1}(c)$ is a M\"obius band. In order for $T$ to be orientable, $c$ must be 1-sided. From this it follows that $\pa_hT\ra B$ is the orientation double cover of $B$. 


For example, if $\pa_hT=\Si_{0,2n}$, then $B=\R P^2\setminus(\bigcup_nD^2)$, and if $\pa_hT=\Si_{1,2n}$, then $B=(\R P^2\#\R P^2)\setminus(\bigcup_nD^2)$.

\subsection{Fibered knots and Heegaard splittings} \label{sec:fibered-link}

\begin{lem}\label{lem:fibered-link}
Let $L\subset S^3$ be a fibered link
\[\mathring F\ra S^3\setminus L\xra{\pi} S^1.\]
The union of two fibers with $L$ is a Heegaard surface. 
\end{lem}

\begin{proof}
Identify $S^1$ with $\R/\Z$. For $\theta\in S^1$, set $\mathring F_\theta=\pi^{-1}(\theta)$, and denote $F_\theta=F_\theta\cup L$ (a compact surface with boundary $L$). We show that 
\[\Sigma= F_0\cup F_{1/2}\]
is a Heegaard surface. 

Consider the interval $J=[0,1/2]\subset S^1$. Then $\pi^{-1}(J)\cong\mathring F\times J$. We want to show that $W:=(\mathring F\times J)\cup K$ is a handlebody. 

First observe that $F\times J$ is a handlebody. Choose a maximal collection of disjoint essential arcs $a_1,\ldots,a_m\subset F$ whose complement is a disk. Then $a_1\times J,\ldots,a_m\times J$ is a collection of disjoint disks whose complement in $F\times J$ is a 3-ball. 

Next observe that $W$ is homeomorphic to $F\times J\cup C(q)$, where $C(q)$ is the mapping cylinder of the projection $q:\partial F\times J\rightarrow \pa F$. Observe that $C(q)$ is homeomorphic to $\partial F\times C(q')$, where $q':J\rightarrow\{*\}$ (constant map). It is easy to show that $C(q')$ is homeomorphic to $J\times [0,1]$, and it follows that $W\cong F\times J\cup C(q)$ is homeomorphic to $F\times J$. 

The argument works the same for $J=[1/2,1]$. Therefore, $\Sigma\subset S^3$ is a Heegaard surface. 
\end{proof}

\begin{rmk}[Monodromy]\label{rmk:monodromy}
Using Lemma \ref{lem:fibered-link}, we obtain a description for the gluing map of the Heegaard splitting associated to a fibered link. This is pictured schematically in Figure \ref{fig:fibered-link}; see also Figure \ref{fig:I-bundles}. Here $\phi\in\Mod(F)$ is a lift of the monodromy of the fibration $\mathring F\ra S^3\setminus L\rightarrow S^1$ with respect to the short exact sequence
\[1\ra\Z^{|L|}\ra\Mod(F)\ra\Mod(\mathring F)\ra1.\]
Any lift defines a Heegaard splitting, but changing $\phi$ by an element of the kernel (a boundary twist) may result in a different manifold ($\neq S^3$). This is discussed more in the Appendix, where we give a proof of the classification of genus-1 fibered knots. 

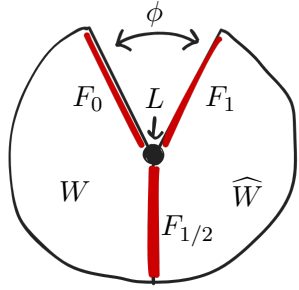
\begin{figure}[h!]
\labellist
\pinlabel $F_0$ at 130 300
\pinlabel $F_{1/2}$ at 205 200
\pinlabel $F_1$ at 230 300
\pinlabel $L$ at 180 300
\pinlabel $W$ at 120 230
\pinlabel $\what W$ at 250 230
\pinlabel $\phi$ at 180 364
\endlabellist
\centering
\vspace{.1in}
\includegraphics[scale=.5]{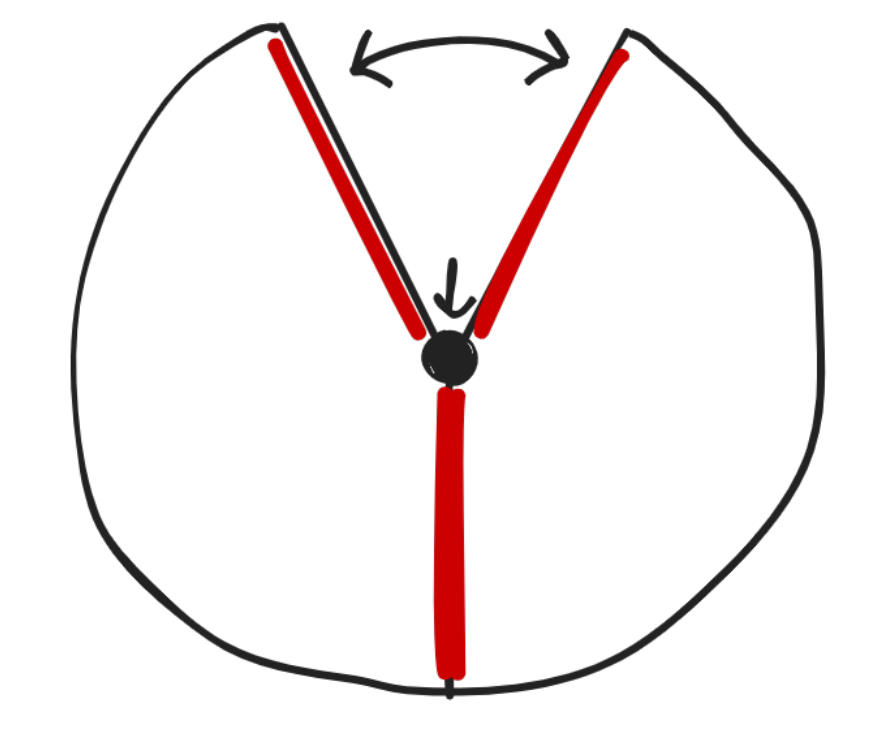}
\vspace{.1in}
\caption{$M_\phi$ is obtained by gluing two handlebodies $F\ti I$ and $\what F\ti I$.}
\label{fig:fibered-link}
\end{figure}

\end{rmk}

\paragraph{A converse.} Lemma \ref{lem:fibered-link} and Remark \ref{rmk:monodromy} have an obvious converse. Fix two copies of a compact, oriented surface $F=\what F$. Let $M$ be a 3-manifold with a decomposition $M=(F\ti I)\cup(\what F\ti I)$, where the gluing $h:\pa (F\ti I)\rightarrow\pa (\what F\ti I)$ matches the horizontal boundary components $h(F\ti\{j\})=\what F\ti \{j\}$ for $j=0,1$. This is a Heegaard splitting and $\pa F\ti\{1/2\}$ is a fibered link in $M$. 

The homeomorphism type of $M$ is determined by the gluing map $h$ and more precisely by a single mapping class $[\phi]\in\Mod(F)$. After an isotopy of $h$, we can assume that $\rest{h}{(\pa F)\ti I}$ is the identity. Define a homeomorphism $\phi:F\ra F$ by the composition
\[\phi: F= F\ti 0\cong F\ti 1\xra{h} \what F\ti 1\cong  \what F\ti0\xra{h^{-1}} F\ti 0= F.\]
The homeomorphism type of $M$ depends only on the isotopy class of $\phi$, and we denote $M=M_\phi$. 

\begin{rmk}\label{rmk:gluing}Since we have a fixed identification $F=\what F$, we can (and will for convenience) choose $h$ so that $\rest{h}{ F\ti0}=\id$. Then $\phi=\rest{h}{F\ti 1}$. See Figure \ref{fig:I-bundles}. 
\end{rmk}


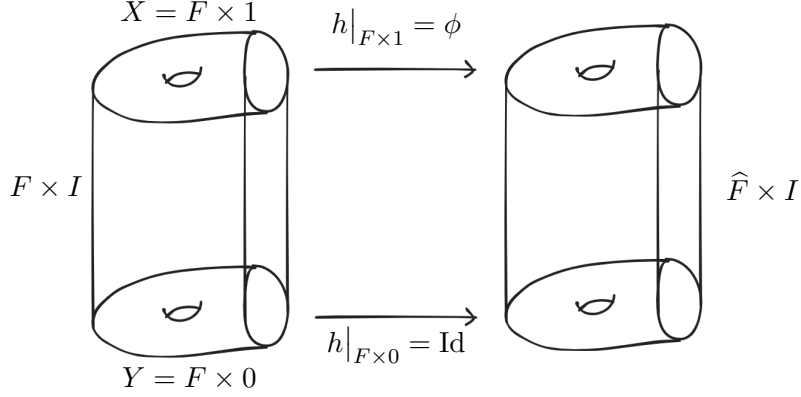
\begin{figure}[h!]
\labellist
\pinlabel $F\ti I$ at 20 660
\pinlabel $\what F\ti I$ at 470 660
\pinlabel $X=F\ti 1$ at 110 770
\pinlabel $Y=F\ti 0$ at 110 540
\pinlabel $\rest{h}{F\ti 1}=\phi$ at 240 760
\pinlabel $\rest{h}{F\ti 0}=\id$ at 240 560
\endlabellist
\centering
\vspace{.1in}
\includegraphics[scale=.6]{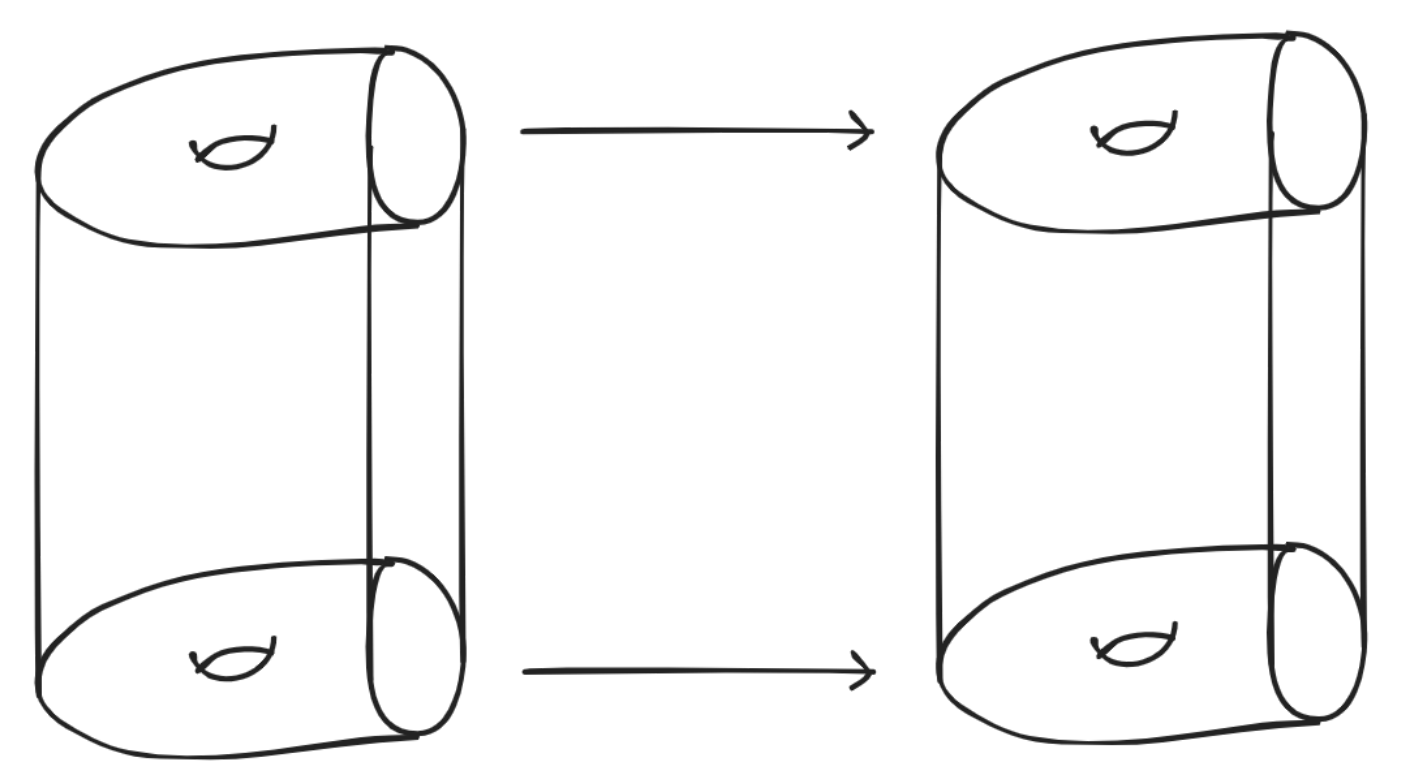}
\vspace{.1in}
\caption{$M_\phi$ is obtained by gluing two handlebodies $F\ti I$ and $\what F\ti I$.}
\label{fig:I-bundles}
\end{figure}

\begin{rmk}[Orientations] 
To view $M$ as an oriented 3-manifold, fix an orientation on $F=\what F$, orient $F\ti I$ so that the positive normal direction to $F\ti\{1/2\}$ points toward $F\ti\{1\}$, and orient $\what F\ti I$ in the opposite way, so that the positive normal direction to $\what F\ti\{1/2\}$ points toward $\what F\ti\{0\}$. With this choice, $h$ is orientation reversing, so $M$ inherits its orientation from $F\ti I$ and $\what F\ti I$. 
\end{rmk}

\begin{rmk}[Fiber-preserving Goeritz elements]\label{rmk:fiber-preserving} If $M_\phi=(F\ti I)\cup(\what F\ti I)$ as above, and $\wtil\phi\in\Homeo(F,\pa F)$ represents $\phi\in\Mod(F)$, then the product homeomorphism $\phi\times\id$ of $F\ti I$ and $\what F\ti I$ induces a homeomorphism of $M_\phi$ that preserves the Heegaard splitting. By construction, the isotopy class of this homeomorphism restricted to $F\ti 1$ is $\phi$.
\end{rmk}

\begin{prop}[Fiber-preserving homeomorphisms]\label{prop:fiber-preserving}
Let $F$ be a compact surface, and fix $M=(F\ti I)\cup_\phi(\what F\ti I)$ as above. Fix an orientation-preserving homeomorphisms $g:M\ra M$ that preserves this Heegaard splitting. If $g$ preserves $F\ti0\cup F\ti 1$, then $g$ is isotopic, through homeomorphisms preserving the Heegaard splitting, to a homeomorphism whose restriction to each of $F\ti I$ and $\what F\ti I$ is fiber preserving. 
\end{prop}

\begin{proof}
There are two cases: either the components of $F\ti 0\cup F\ti 1$ are preserved, or they are interchanged. We explain the case each component is preserved; the other case is similar. 

Since a homeomorphism of a handlebody that is the identity on the boundary is isotopic to the identity \cite[\S1]{hatcher-handlebody}, it suffices to show that there exists a fiber-preserving homeomorphism $f$ so that $\rest{f^{-1}\circ g}{\Sigma}$ is isotopic to the identity, where $\Sigma=\pa(F\ti I)$. 

Set $g_i=\rest{g}{F\ti\{i\}}$ for $i=0,1$. 
Observe that $g_0$ and $g_1$ are isotopic homeomorphisms of $F$, since they induce the same (outer) automorphism of $\pi_1(F)$ (identifying $\pi_1(F\ti 1)=\pi_1(F\ti I)=\pi_1(F\ti0)$. Note that this isotopy is not necessarily the identity on $\pa F$. In any case, by extension we can isotope $g$ to a homeomorphism (still denoted $g$) satisfying $g_0=g_1$. 

Let $f=g_0\ti \id$ on $F\ti I$ and $\what F\ti I$. By construction, $f^{-1}\circ g$ is the identity on $F\ti\{0,1\}$, so $\rest{f^{-1}\circ g}{\Sigma}$ is supported on $(\pa F)\ti I$. Then the isotopy class of $\rest{f^{-1}\circ g}{\Sigma}$ is a multitwist. 

By \cite[Thm.\ 1.11]{oertel}, a multitwist about $c\sbs\Sigma$ extends to a handlebody only if the curves in $c$ bound a union of essential disks and annuli. Since no curve in $c=\pa F\ti\{1/2\}$ bounds a disk in $F\ti I$, and no two curves of $c$ bound an annulus in $F\ti I$, we conclude that $\rest{f^{-1}\circ g}{\Sigma}$ is isotopic to the identity. 
\end{proof}

\begin{rmk}\label{rmk:relation}
Assume that $f$ is fiber preserving on $F\ti I$ and $\what F\ti I$. If $f$ preserves the components of $F\ti 0\cup F\ti 1$, then we can write $f=\psi\ti\id$ on $F\ti I$ and $f=\what\psi\ti\id$ on $\what F\ti I$, and the gluing $h:\pa(F\ti I)\xra{\cong}\pa(\what F\ti I)$ forces the relation $\psi=\what\psi$, and $\psi\>\phi=\phi\>\psi$ in $\Mod(F)$. Together these imply that $\psi\>\phi\>\psi^{-1}=\phi$. 

Suppose instead that $f$ interchanges the components of $F\ti 0\cup F\ti 1$. Write $f(x,t)=(\psi(x),1-t)$ on $F\ti I$ and $f(x,t)=(\what\psi(x),1-t)$ on $\what F\ti I$. Here the gluing forces the relations $\psi=\what\psi\>\phi$ and $\what\psi=\phi\>\psi$. Together these imply $\psi\>\phi\>\psi^{-1}=\phi^{-1}$. 


We also observe that the homeomorphism $(x,t)\leftrightarrow(x,1-t)$ on $F\ti I$ and $\what F\ti I$ defines an orientation-reversing homeomorphism between $M_\phi$ and $M_\phi^{-1}$. 
\end{rmk}

\subsection{Heegaard splittings and Mayer--Vietoris}

Here we make a simple observation that will be useful in multiple arguments. 

\begin{lem}\label{lem:mayer-vietoris} 
Let $M$ be a $3$-manifold with a Heegaard splitting $M=W\cup_\Sigma\what W$. Suppose there exists a multicurve $c\sbs\Sigma$ that bounds a surface in $W$ and $\what W$. If $c$ can be oriented so that the associated homology class $[c]\in H_1(\Sigma)$ is nonzero, then $H_2(M)\neq0$. 
\end{lem}
\begin{proof}
The Mayer--Vietoris sequence associated to the Heegaard splitting contains the exact sequence 
\[H_2(M)\xra{d} H_1(\Si)\xra{j} H_1(W)\oplus H_1(\what W)\]
By assumption there is a nonzero class $H_1(\Sigma)$ in $\ker(j)=\im(d)$, so $H_2(M)\neq0$. 
\end{proof}

\section{Classification of holes for $\ca P(V)$} \label{sec:holes}


This section is devoted to the proof of Theorem \ref{thm:holes}. Our analysis is based on work of Masur--Schleimer \cite[\S9-12]{masur-schleimer}, who characterize large-diameter holes for the disk complex $\ca D(V)$. Theorem \ref{thm:holes} does not follow easily from \cite{masur-schleimer}, but we do use several of the same ideas. 

\paragraph{Overview of the proof.} 

Let $X\subset S$ be an essential subsurface, and assume $X$ is a hole for $\ca P(V)$. It is easy to show that if $X=S$, then $X$ supports a pseudo-Anosov in $\bb G$ (see Figure \ref{fig:thurston-pseudo}), so $X$ has infinite diameter. Therefore, we assume that $X\subsetneq S$ is a proper subsurface. We divide into three cases: $X\cong\Sigma_{0,2}$ is an annulus , $X\neq \Sigma_{0,2}$ is primitively compressible, and $X\neq \Sigma_{0,2}$ is primitively incompressible. In the first two cases, we will show that every hole $X$ has diameter $\le13$ (Thm.\ \ref{thm:annuli} and Prop.\ \ref{prop:compressible}). The last case is the most complicated and interesting. When $X\neq\Sigma_{0,2}$ is primitively incompressible with diameter $\ge61$, we construct homeomorphisms 
\[V\cong\Sigma_{1,1}\times I\cong\what V\]
such that $X$ is a component of the horizontal boundary of each $I$-bundle (Theorem \ref{thm:incompressible}). In particular, this gives a fibering 
\begin{equation}\label{eqn:fibering}X\setminus\pa X\ra S^3\setminus\partial X\ra S^1.\end{equation} Therefore, $\partial X$ is a genus-1 fibered knot. It is well known that this implies that $\partial X$ is either the figure-8 knot or the right/left-handed trefoil knot (this is discussed more in the Appendix). It is easy to show that $X$ has infinite diameter when $\partial X$ is the figure-8 knot (since this knot has hyperbolic monodromy, c.f.\ Remark \ref{rmk:fiber-preserving}). We show that if $\partial X$ is a trefoil knot, then $X$ has finite diameter (\S\ref{sec:trefoil}). Altogether this proves Theorem \ref{thm:holes}. 

\begin{figure}[h!]
\labellist
\small
\pinlabel $c$ at 170 680
\pinlabel $d$ at 300 750
\endlabellist
\centering
\includegraphics[scale=.25]{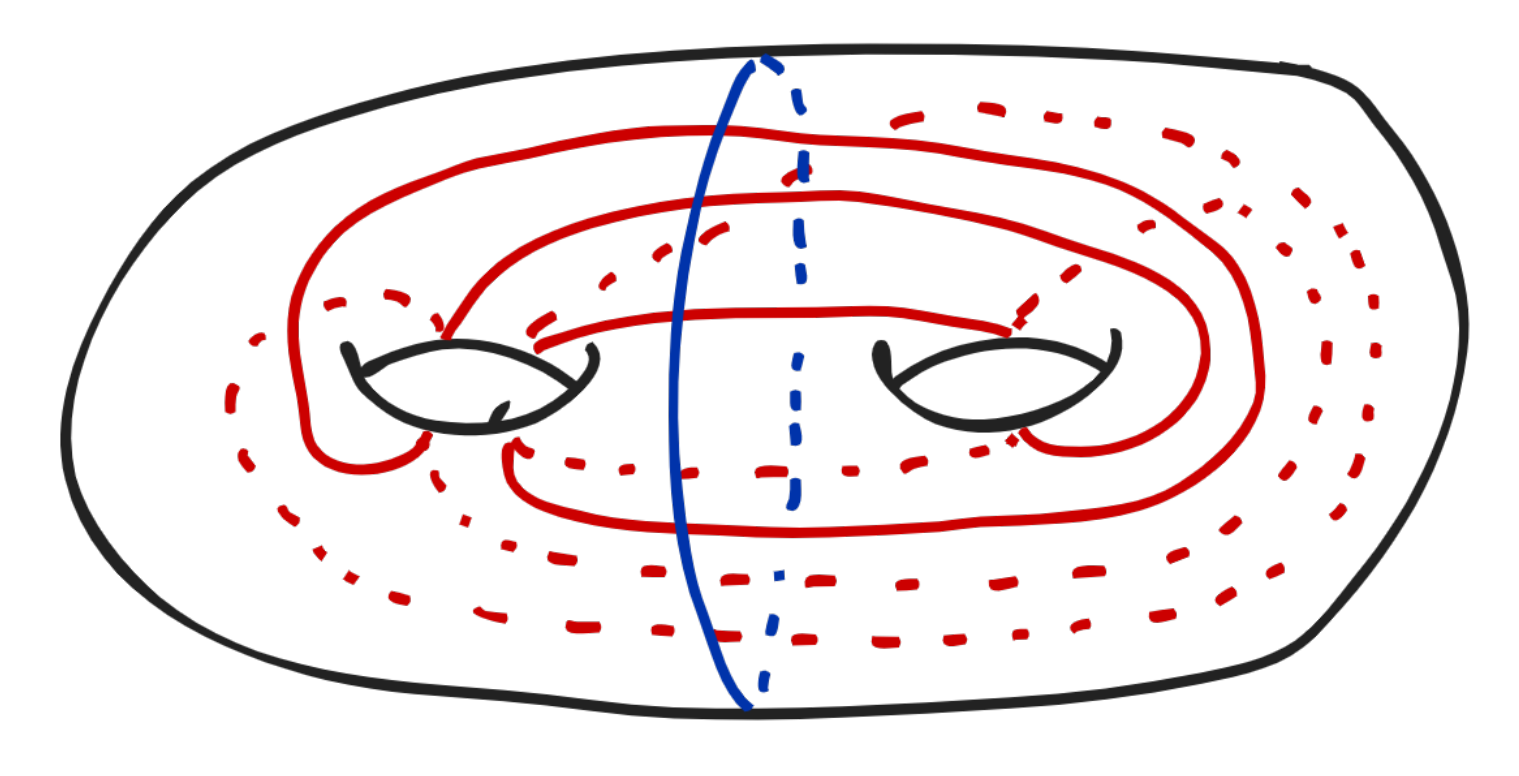}
\caption{The curves $c,d$ fill $S$, so by a result of Thurston, the mapping class $T_cT_d^{-1}$ is pseudo-Anosov. Since $c,d$ are reducing sphere curves, $T_c,T_d\in\bb G$. }
\label{fig:thurston-pseudo}
\end{figure}

As part of our analysis, we give a concrete description of the set of primitive disks that are vertical with respect to $V\cong \Sigma_{1,1}\times I$ (in each case), and we give a precise description of the image of the subsurface projection of $\ca P(V)$ to the horizontal boundary components (\S\ref{sec:fig8}). This is needed for the proof of the distance formula (Theorem \ref{thm:distance}). We also determine the subgroup of $\bb G$ that preserves the fibering (\ref{eqn:fibering}) when $\pa X$ is the figure-8 knot (\S\ref{sec:fig8-symmetry}). This is used for the proof of Theorem \ref{thm:pseudo}.

\subsection{The hole $X$ is an annulus}\label{sec:annuli}

\begin{thm}[Annular holes for $\ca P(V)$]\label{thm:annuli}
Let $X\subset S$ be an essential annulus. If $X$ is a hole, then the diameter of $X$ is at most $11$. 
\end{thm}

Theorem \ref{thm:annuli} is similar to \cite[Thm.\ 10.1]{masur-schleimer}, but there is one ingredient from their argument that doesn't work: the first Claim following Theorem 10.1, which asserts that if $X=n(c)$ is an annular hole for $\ca D(V)$, then $i(\pa D,c)\ge2$ for each $D\in\ca D(V)$. Figure \ref{fig:counterexample} shows that this statement is not true if $\ca D(V)$ is replaced by $\ca P(V)$. Although we cannot use this claim, we will use the  same overall strategy as the proof of \cite[Thm.\ 10.1]{masur-schleimer} to prove Theorem \ref{thm:annuli}. 

\begin{figure}[h!]
\labellist
\pinlabel $D$ at 145 305
\pinlabel $\pa F$ at 320 230
\pinlabel $c$ at 440 300
\endlabellist
\centering
    \includegraphics[scale=.4]{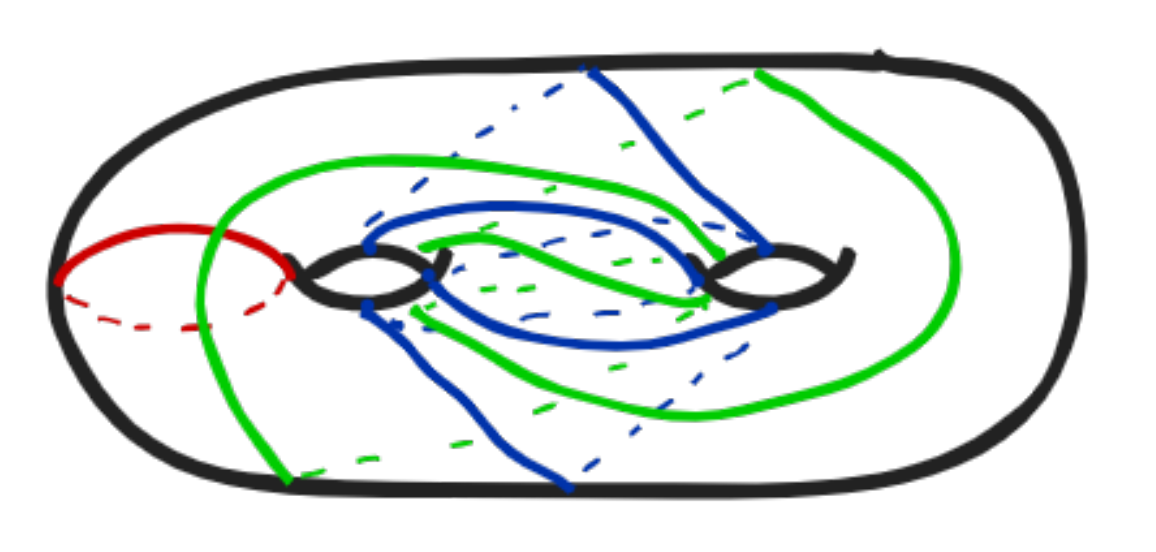}
\caption{The annulus $X=n(c)$ is a hole for $\ca P(V)$, and $D$ is a primitive disk with $i(\pa D,c)=1$. To see that $X$ is a hole, if $E$ is a primitive disk disjoint from $c$, then (by surgery) there exists primitive disk $E'$ that is disjoint from $D\cup c$. But there is a unique disk $F\sbs V$ with $\pa F$ disjoint from $\pa D\cup c$, and this disk is not primitive. }
\label{fig:counterexample}
\end{figure}

\begin{proof}[Proof of Theorem \ref{thm:annuli}]
Fix an essential annulus $X=n(c)\subset S$, and assume that $X$ is a hole for $\ca P(V)$. Choose primitive $D\sbs V$ minimizing $i(\pa D,c)$ among all primitive disks (equivalently, $D$ minimizes the number of components of $\pa D\cap X$). Assume for a contradiction that the diameter of $X$ is at least $12$. Then there exists primitive $E\sbs V$ so that $d_X(D,E)\ge 6$. 

Put $D,E,X$ in standard position (cf.\ \S\ref{rmk:standard-position}) with any triangular components of $S\setminus(\pa D\cup\pa E\cup\pa X)$ contained in $X$. Let $B\subset E$ be an outermost bigon cut by $D\cap E$. Denote $a:=B\cap D$ and $b:=B\cap \pa E$, and let $D', D''$ be the disks obtained by surgering $D$ along $B$; cf.\ Example \ref{ex:outermost}. 

By the Lemma \ref{lem:intersect} below, every component of $\pa D\cap X$ intersects every component of $\pa E\cap X$. 

\begin{lem}\label{lem:intersect} 
Fix vertices $v,w\in\ca C(S)$ that cut the annulus $X\sbs S$. Choose representatives so that $v,w,X$ are in standard position with triangular components of $S\setminus(v\cup w\cup \pa X)$ contained in $X$. If $d_X(v,w)\ge 6$, then every component of $v\cap X$ intersects every component of $w\cap X$. 
\end{lem}
Lemma \ref{lem:intersect} is the second claim after Theorem 10.1 in \cite{masur-schleimer}. The statement above is more general than the statement in \cite{masur-schleimer}, but their proof holds in this generality. We apply the lemma to with $(v,w)=(\pa D,\pa E)$. 

The proof proceeds by studying the relationship between the arc $b$ and the arcs $X\cap\pa E$ in $\pa E$. First note that no component of $X\cap\pa E$ can be strictly contained in $b$, since then this component would be disjoint from $\pa D\cap X$, contradicting Lemma \ref{lem:intersect}. Now we consider cases based on whether the endpoints of $b$ meet $X$, and how many components of $X$ they meet. These cases are pictured in \cite[Fig.\ 10.4]{masur-schleimer} and listed below. Ultimately, we will show that each of these cases is impossible, since otherwise either $D$ and $E$ are not in minimal position or at least one of $D',D''$ would intersect $X$ in fewer components than $D$ (contradicting our assumption). From this we deduce that there is no pair $D,E$ with $d_X(D,E)\ge6$, so the diameter of $X$ is at most $11$.  

The only case that is treated differently from \cite{masur-schleimer} is Case $4$ below. We include the other cases for completeness.

\paragraph{Case $1$:} the arc $b$ is disjoint from $X$. In this case, each component of $D\cap X$ contributes exactly one component to either $D'\cap X$ or $D''\cap X$ (not both). It follows then that one or both of $D',D''$ intersect $X$ in fewer components than $D$, a contradiction. This is illustrated in Figure \ref{fig:annulus-case1}. 

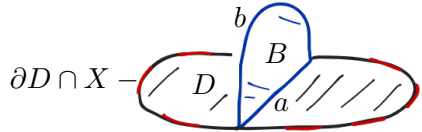
\begin{figure}[h!]
\labellist
\pinlabel $\pa D\cap X\>-$ at 5 855
\pinlabel $D$ at 130 850
\pinlabel $B$ at 200 880
\pinlabel $b$ at 165 915
\pinlabel $a$ at 205 830
\endlabellist
\centering
    \includegraphics[scale=.4]{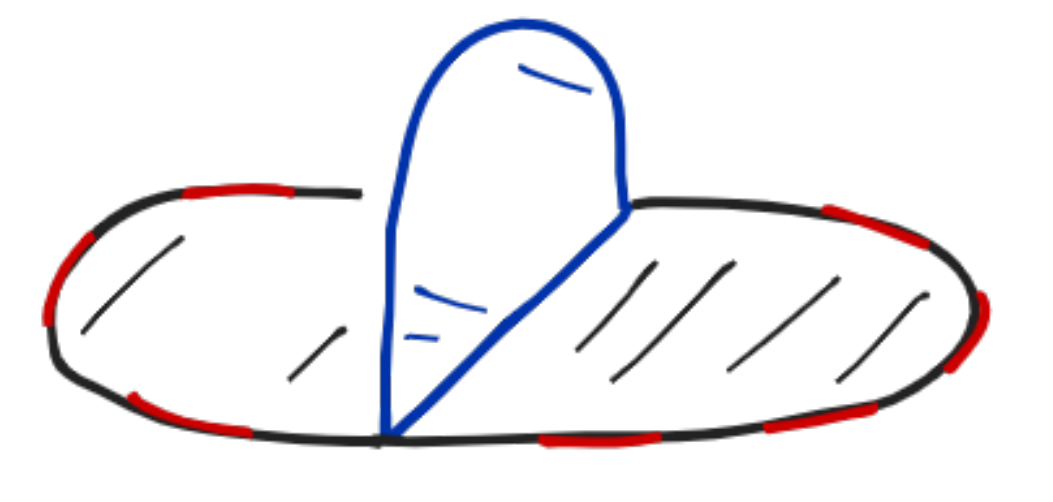}
\caption{The disk $D$, the bigon $B$, and the arcs of $\pa D\cap X$. In the figure, $b$ is disjoint from $X$, and one observes that $\pa D'$ and $\pa D''$ have fewer intersections with $X$ than $\pa D$ does.}
\label{fig:annulus-case1}
\end{figure}

\paragraph{Case $2$:} the arc $b$ is contained in $X$. In this case, either the endpoints of $b$ meet (i) the same or (ii) different components of $\pa D\cap X$. In case (i), let $d$ denote the component of $D\cap X$ meeting $b$. Then the complement of $d\cup b$ in $X$ contains a bigon, which contradicts the assumption that $D,E$ intersect minimally. In case (ii), observe that if $\pa D$ is oriented, then the two intersections of $b$ with $\pa D$ appear on the same side of the normal bundle of $\pa D$ in $S$ because $b$ is an arc of a surgery bigon. Then we conclude that there is a component of each of $D'\cap X$ and $D''\cap X$ whose endpoints lie on the same component of $\pa X$, so after an isotopy $D'$ and $D''$ intersects $X$ in fewer components than $D$, a contradiction. This case is illustrated in Figure \ref{fig:annulus-case2}.

\begin{figure}
\labellist
\pinlabel $b$ at 605 860
\pinlabel $b$ at 859 868
\pinlabel $d$ at 535 865
\pinlabel $\text{component of }\pa D'\cap X$ at 850 913
\pinlabel $\text{component of }\pa D''\cap X$ at 850 810
\pinlabel $\nwarrow$ at 800 825
\pinlabel $\searrow$ at 850 895
\endlabellist
\centering
\includegraphics[scale=.6]{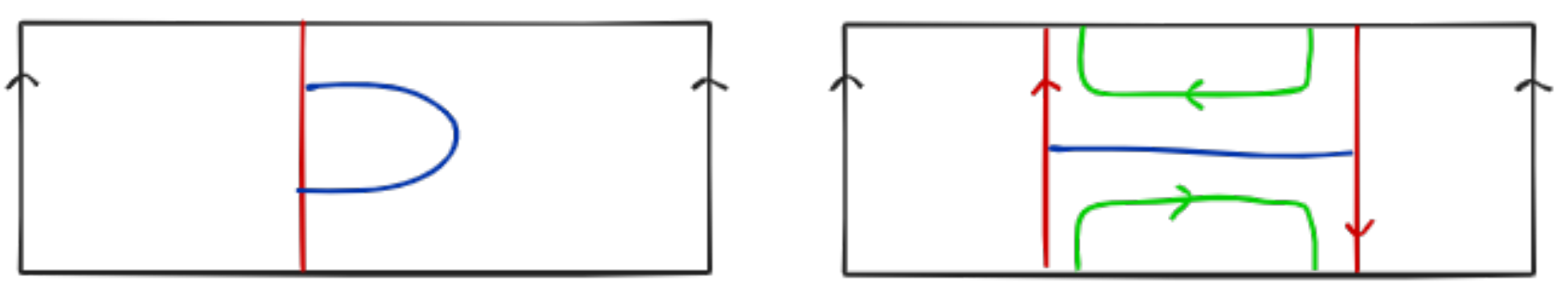}
\caption{The arc $b$ is contained in $X$ and connects either a component of $\pa D\cap X$ to itself (left) or two distinct components of $\pa D\cap X$ (right).}
\label{fig:annulus-case2}
\end{figure}
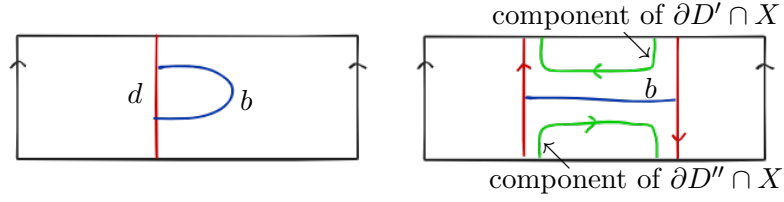

\paragraph{Case $3$:} exactly one endpoint of $b$ is contained in $X$. Let $d$ be the component of $D\cap X$ that meets $b$. The endpoints of $d$ meet different components of $X$ (otherwise $\pa D$ and $X$ do not intersect minimally). Then $b\cap\pa X$ (which is a single point) belongs to the same component as one endpoint of $d$. This implies that the surgered disk that contains this endpoint (either $D'$ or $D''$) can be isotoped to intersect $X$ in fewer components than $D$. See Figure \ref{fig:annulus-case3}. 

\begin{figure}
\labellist
\pinlabel $d$ at 420 250
\pinlabel $b$ at 460 255
\pinlabel $\text{component of }\pa D'\cap X$ at 410 305
\pinlabel $\text{component of }\pa D''\cap X$ at 555 260
\pinlabel $\nwarrow$ at 450 260
\pinlabel $\searrow$ at 455 293
\pinlabel $\nwarrow$ at 485 270
\endlabellist
\centering
\includegraphics[scale=.8]{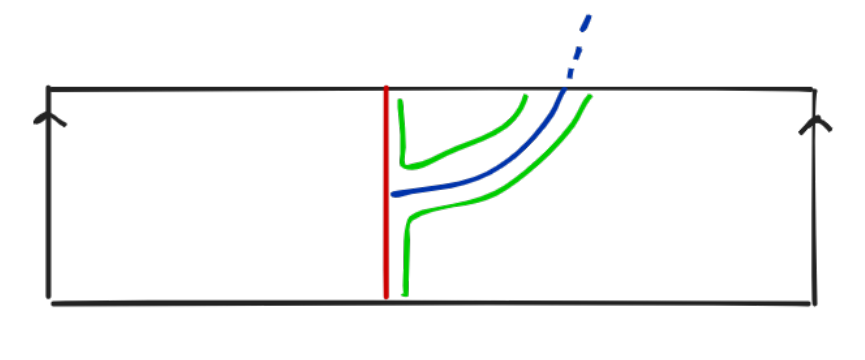}
\caption{The arc $b$ has one endpoint in $X$. }
\label{fig:annulus-case3}
\end{figure}
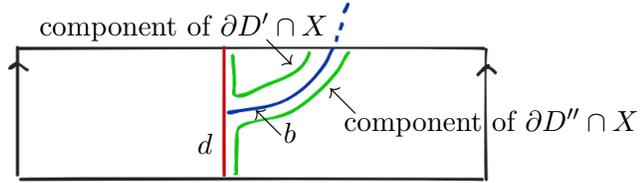

\paragraph{Case $4$:} both endpoints of $b$ are contained in $X$, but $b\not\subset X$. As in Case $2$, we divide into two additional subcases: either the endpoints of $b$ meet (i) the same or (ii) different components of $D\cap X$. 

In case (i), and let $d$ be the component of $D\cap X$ that meets $b$. Let $b_{\pm}$ and $d_{\pm}$ denote the endpoints of $b$ and $d$, labeled so that $d$ is a union of intervals $[d_-,b_-]\cup[b_-,b_+]\cup[b_+,d_+]$. We also denote $a_\pm\in b$ the point of $b\cap\pa X$ belonging to the same component of $b\cap X$ as $b_\pm$. As previously observed, $d_+$ and $d_-$ lie on different components of $\pa X$, and if we orient $d$, then the intersections of $b$ with $d$ both occur on the same side of $d$. See Figure \ref{fig:annulus-case4(i)}. 

Let $D''$ be the surgered disk that intersects the interval $[b_-,b_+]\subset\delta$. Consider further cases, depending on whether $a_+$ and $d_+$ belong to the same or different components of $\pa X$, and similarly for $a_-$ and $d_-$. If $a_+,d_+$ lie on the same component of $\pa X$ and the same for $a_-,d_-$, then $D'$ intersects $X$ in fewer components than $D$ (after an isotopy). If one pair $a_+,d_+$ or $a_-,d_-$ lie on the same component, and the other pair lies on different components, then $D''$ intersects $X$ in fewer components than $D$ (after an isotopy). These cases are pictured in Figure \ref{fig:annulus-case4(i)}. Finally, assume that $a_+,d_+$ lie on different components of $\pa X$ and the same for $a_-,d_-$. In this situation, the arcs $[b_-,a_-]$ and $[b_+,a_+]$ in the disk $X\setminus d$ have endpoints that link on the boundary. This forces these subarcs of $b$ to intersect, a contradiction. See Figure \ref{fig:annulus-case4(i)b}. This concludes case 4(i). 

\begin{figure}
\labellist
\pinlabel $d_-$ at 550 680
\pinlabel $d_+$ at 550 780
\pinlabel $D'$ at 580 680
\pinlabel $\uparrow$ at 580 694
\pinlabel $D'$ at 580 780
\pinlabel $\downarrow$ at 580 767
\pinlabel $D''$ at 620 720
\pinlabel $b_-$ at 530 715
\pinlabel $b_+$ at 530 738
\pinlabel $a_-$ at 630 680
\pinlabel $a_+$ at 610 780
\pinlabel $d_-$ at 800 680
\pinlabel $d_+$ at 800 780
\pinlabel $b_-$ at 780 715
\pinlabel $b_+$ at 780 738
\pinlabel $a_-$ at 855 680
\pinlabel $a_+$ at 885 680
\pinlabel $D'$ at 900 720
\pinlabel $D'$ at 830 680
\pinlabel $\uparrow$ at 830 694
\pinlabel $D''$ at 845 718
\endlabellist
\centering
\includegraphics[scale=.8]{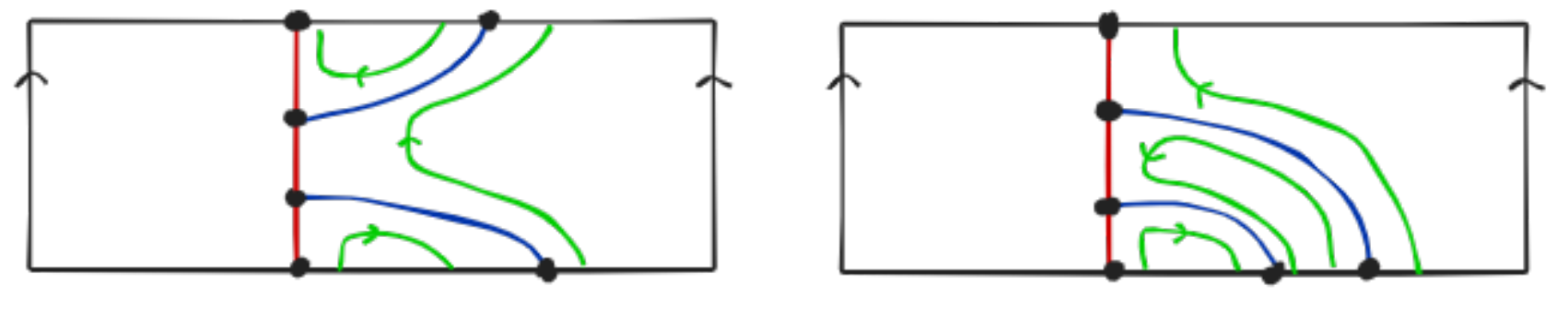}
\caption{Left: $a_+,d_+$ lie on the same component of $\pa X$, and the same holds for $a_-,d_-$. Right: $a_+,d_+$ lie on different components of $\pa X$, and $a_-,d_-$ lie on the same component of $\pa X$. }
\label{fig:annulus-case4(i)}
\end{figure}

\begin{figure}
\labellist
\endlabellist
\centering
\includegraphics[scale=.7]{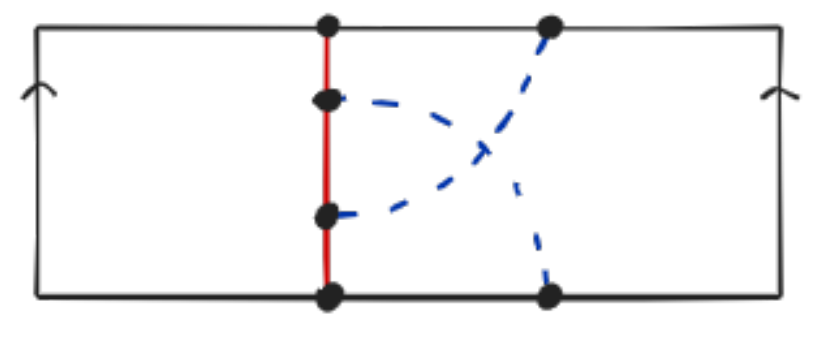}
\caption{$a_+,d_+$ lie on different components of $\pa X$ and the same holds for $a_-,d_-$.}
\label{fig:annulus-case4(i)b}
\end{figure}

Finally we consider case 4(ii). In this case, $D$ intersects $X$ in exactly 2 components, since otherwise $D'$ or $D''$ would have smaller intersection with $X$ than $D$. Let $d_{\pm}$ denote the two components of $D\cap X$ meeting $b$. Let $d_\pm'$ be the endpoint of $d_\pm$ that belongs to $D'$, and define $d_\pm ''$ similarly. Define $a_\pm, b_\pm$ as in case 4(i) with $[a_\pm,b_\pm]$ a subinterval of $b$ with $b_\pm\in d_\pm$ and $a_\pm\in\pa X$. Since the endpoints $d_+',d_+''$ of $d_+$ belong to different components of $\pa X$, exactly one endpoint lies in the same component as $a_+$. If $a_+$ (resp.\ $a_-$) lies on the same component of $\pa X$ as $d_+'$ (resp.\ $d_-'$), then $D'$ will be disjoint from $X$ (after isotopy); we conclude similarly when $a_+,a_-$ lie on the same component of $\pa X$ as $d_+'',d_-''$. If $a_+$ (resp.\ $a_-$) lies on the same component of $\pa X$ as $d_+'$ (resp.\ $d_-''$), then $|D'\cap X|=|D''\cap X|=1<|D\cap X|=2$, contradicting the fact that $D$ is minimal. The remaining case $(a_+,a_-)\leftrightarrow(d_+'',d_-')$ is similar. These different cases are pictured in Figure \ref{fig:annulus-case4(ii)}. 

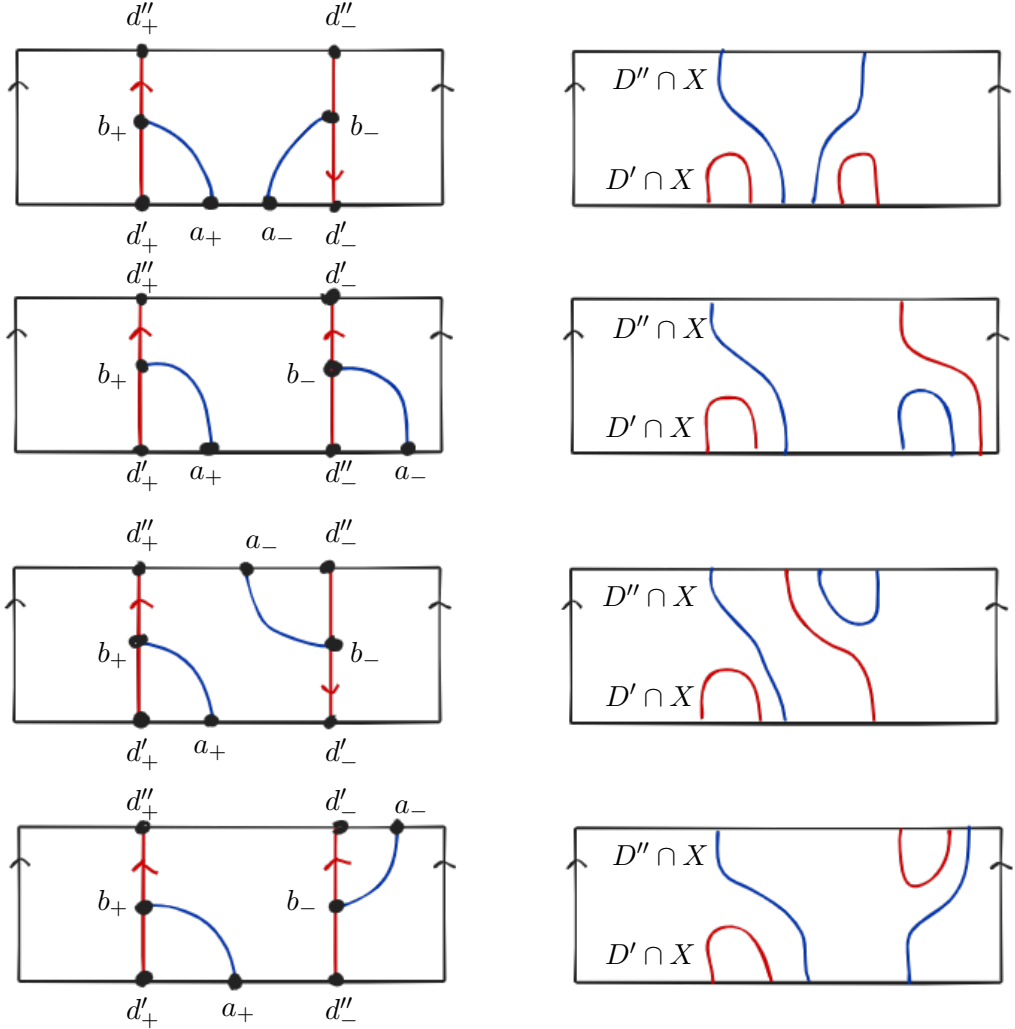
\begin{figure}
\labellist
\pinlabel $d_+''$ at 520 520
\pinlabel $d_+'$ at 520 410
\pinlabel $d_-''$ at 620 520
\pinlabel $d_-'$ at 620 410
\pinlabel $b_+$ at 505 465
\pinlabel $b_-$ at 632 465
\pinlabel $a_+$ at 552 410
\pinlabel $a_-$ at 588 410
\pinlabel $D''\cap X$ at 780 490
\pinlabel $D'\cap X$ at 775 440
\pinlabel $d_+'$ at 520 290
\pinlabel $d_+''$ at 520 390
\pinlabel $d_-''$ at 620 290
\pinlabel $d_-'$ at 620 390
\pinlabel $a_+$ at 552 290
\pinlabel $a_-$ at 655 290
\pinlabel $b_+$ at 505 340
\pinlabel $b_-$ at 600 340
\pinlabel $D''\cap X$ at 780 365
\pinlabel $D'\cap X$ at 775 315
\pinlabel $d_+'$ at 520 150
\pinlabel $d_+''$ at 520 260
\pinlabel $d_-''$ at 620 260
\pinlabel $d_-'$ at 620 150
\pinlabel $b_+$ at 505 200
\pinlabel $b_-$ at 632 200
\pinlabel $a_+$ at 554 152
\pinlabel $a_-$ at 580 255
\pinlabel $D''\cap X$ at 775 230
\pinlabel $D'\cap X$ at 775 180
\pinlabel $d_+'$ at 520 20
\pinlabel $d_+''$ at 520 125
\pinlabel $d_-'$ at 620 125
\pinlabel $d_-''$ at 620 20
\pinlabel $a_+$ at 568 20
\pinlabel $a_-$ at 655 122
\pinlabel $b_+$ at 505 75
\pinlabel $b_-$ at 600 75
\pinlabel $D'\cap X$ at 775 50
\pinlabel $D''\cap X$ at 780 100
\endlabellist
\centering
\includegraphics[scale=.75]{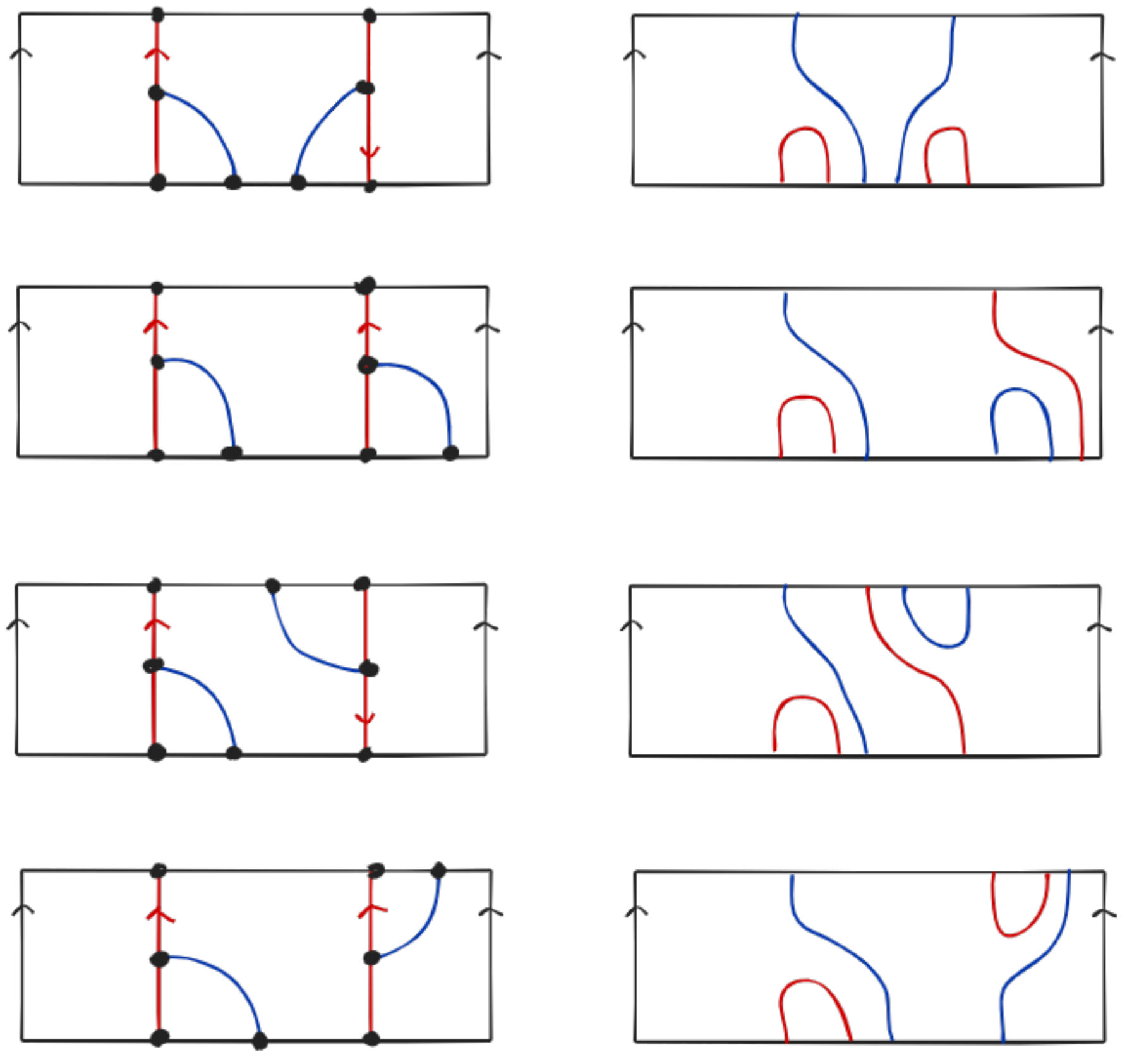}
\caption{Left: The entire intersection of $\pa D$ and $b$ with $X$ in the various cases, depending on whether $a_\pm$ lies on the same component as $d_{\pm}'$ or $d_{\pm}''$. Right: The entire intersection of $\pa D'$ and $\pa D''$ with $X$, which illustrates that (after an isotopy) $D',D''$ intersect $X$ in fewer components than $D$ does.}
\label{fig:annulus-case4(ii)}
\end{figure}

This completes the proof the proof of Theorem \ref{thm:annuli}. \end{proof}

\subsection{The hole $X$ is primitively compressible}\label{sec:compressible}

Recall from \S\ref{sec:background} that $X\subset S$ is primitively compressible if there exists a primitive disk $E\sbs V$ such that $\pa E\subset X$ is nonperipheral. 

\begin{prop}[Primitively compressible holes]\label{prop:compressible}
Let $X\subset S$ be a non-annular subsurface that is primitively compressible. If $X$ is a hole for $\ca P(V)$, then either $X=S$ or $X$ has diameter $\le 13$. 
\end{prop}

\begin{proof}[Proof of Proposition \ref{prop:compressible}]
Fix $X$ as in the statement, and suppose that there exist primitive disks $D_1,E_1$ with $d_X(D_1,E_1)\ge14$. We will show $X=S$, proceeding by contradiction. Since $X\neq \Sigma_{0,2}$ (by assumption) and $X\neq \Sigma_{0,3}$ (because $\ca C(\Sigma_{0,3})$ is empty), $X$ is homeomorphic to one of $\Sigma_{0,4}, \Sigma_{1,1}$, and $\Sigma_{1,2}$. 

By Lemma \ref{lem:surgery-sequence}, there exist disks $D,E$ supported in $X$ with $d_X(D_1,D),d_X(E_1,E)\le 6$. Then $d_X(D,E)\ge2$ by the triangle inequality. Since $\pa D,\pa E$ are vertices of $\ca C(X)$, this implies that $D$ and $E$ have nonzero intersection number. After an isotopy, we may assume that $D\cap E$ is circle free (c.f.\ \S\ref{rmk:standard-position}).

Consider an outermost bigon $(B,a,b)$ cut off from $E$ by $D\cap E$, and the surgered disks $D',D''$. The disks $D,D',D''$ are disjoint, and no two are isotopic (Example \ref{ex:outermost}). Furthermore, $\pa D,\pa D',\pa D''$ bound a subsurface $Y\cong \Sigma_{0,3}$ that is contained in $X$. Since $X$ is homeomorphic to one of $\Sigma_{0,4}$, $\Sigma_{1,1}$, or $\Sigma_{1,2}$, one of the boundary component of $Y$ is parallel to a boundary component of $X$. This means the corresponding disk ($D'$ or $D''$) can be isotoped to be disjoint from $X$, contradicting the assumption that $X$ is a hole. 
\end{proof}

\begin{rmk}\label{rmk:compressible}
By the proof of Proposition \ref{prop:compressible}, if $X\subset S$ is a hole and there exists primitive disks $D,E$ supported in $X$ such that $i(\pa D,\pa E)\neq0$, then there exist a triple $(D,D',D'')$ of disjoint disks supported in $X$, and this implies that $X=S$. We will use this observation later. 
\end{rmk}

\subsection{The hole $X$ is primitively incompressible} \label{sec:incompressible}

In this section we prove that if $X$ is primitively incompressible with sufficiently large diameter, then $X$ is a genus-1 Seifert surface for a fibered knot $K\sbs S\sbs S^3$. This is the content Theorem \ref{thm:incompressible}, although we formulate the theorem slightly differently. 

\begin{thm}[Primitively incompressible holes]\label{thm:incompressible}
If $X\sbs S$ is a primitively incompressible hole for $\ca P(V)$ with diameter $\ge61$, then $V$ and $\what V$ are homeomorphic to the $I$-bundle $\Sigma_{1,1}\times I$ with $X$ a component of the horizontal boundary of each. 
\end{thm}

Theorem \ref{thm:incompressible} is similar to \cite[Thm.\ 12.1]{masur-schleimer}, which we state below in a special case. 

\begin{thm}[Masur--Schleimer]\label{thm:MS-incompressible}
Suppose $X$ is an incompressible hole for $\ca D(V)$ with diameter $\ge57$. Then there is an $I$-bundle $T$ with an embedding $T\hra V$ so that $\pa_hT\sbs S$ and with $X$ a component of $\pa_hT$.  
\end{thm}

We cannot use Theorem \ref{thm:MS-incompressible} directly to prove Theorem \ref{thm:incompressible} since a primitively incompressible hole is not necessarily incompressible. However, we will use some arguments from \cite[Thm.\ 12.1]{masur-schleimer} in our proof. 

\begin{proof}[Proof of Theorem \ref{thm:incompressible}] Fix $X$ as in the statement of the Theorem. We divide the proof into several steps:
\begin{enumerate}
\item Assuming that $X$ has diameter $\ge 57$ with respect to $\ca P(V)$, we build an embedded $I$-bundle $T\hra V$ that contains $X$ as a component of its horizontal boundary. For this we follow the argument of \cite[\S12]{masur-schleimer} as far as possible. We show that either $T\hra V$ is (isotopic to) a homeomorphism or every component of $\pa X$ is nullhomologous in $V$. 
\item We show that if $X$ has diameter $\ge 61$, then $X$ is also a primitively incompressible hole for $\ca P(\what V)$ with diameter $\ge 57$. Our proof of this relies on details of the proof of Step 1. 
\item Combining Steps 1 and 2, we have embedded $I$-bundles $T\hra V$ and $\what T\hra\what V$, each with $X$ as a component of the horizontal boundary. We use the topology of $S^3=V\cup\what V$ to rule out all cases where either $T\hra V$ or $\what T\hra\what V$ is not a homeomorphism, and from this we conclude. 
\end{enumerate}

\paragraph{Step 1.} In this step we construct an embedded $I$-bundle $T\hra V$. The assumption needed in the construction is that $X$ is a primitively incompressible hole for $\ca P(V)$ with diameter $\ge57$. 
Part of the construction follows \cite[\S12]{masur-schleimer} and we refer the reader there for more details. We mostly use their notation. We also remark that a similar construction appeared in earlier work of Oertel; cf.\ \cite[Lem.\ 2.13]{oertel}. 

Fix primitive disks $D_0,E_0\sbs V$ so that $d_X(D_0,E_0)\ge57$. By Lemma \ref{lem:surgery-sequence}, we can find primitive disks $D,E$ that cannot be primitively boundary compressed into $S\setminus n(\pa X)$ and with $d_X(D,E)\ge45$. 

We put $D,E,X$ in standard position with triangular components of $S\setminus(\pa D\cup\pa E\cup\pa X)$ contained in $S\setminus X$ (c.f.\ \S\ref{rmk:standard-position}). 

Give $D$ the structure of a polygon with vertices $\pa D\cap\pa X$, and similarly for $E$. The edges of these polygons alternate between $X$ and $Y:=\ov{S\setminus X}$. Set $\Ga=D\cap E$. The assumption that $D,E$ cannot be primitively boundary compressed into $S\setminus n(\pa X)$ implies that every arc in $\Ga$ is a \emph{diagonal}, i.e.\ it connects a pair of distinct sides. (This argument is the same as \cite[first claim of \S12.5]{masur-schleimer}, although our assumption is slightly different.) 

\paragraph{A side with not many different diagonals.} A counting argument shows that there are sides $a\sbs\pa D\cap X$ and $b\sbs \pa E\cap X$ meeting at most $8$ different types of diagonals (two diagonals have the same type if they are parallel, i.e.\ their endpoints share the same sides). See \cite[Lem.\ 12.4]{masur-schleimer}. 

\paragraph{Arcs and rectangles.} We decompose $a$ into at most $8$ subarcs $\{a_i\}$, each containing one parallel collection of diagonals, and we choose rectangles $R_i\sbs D$ with one side of $R_i$ equal to $a_i$ and so that $R_i$ contains all the diagonals in $\Ga$ meeting $a_i$. We write $a_i'\sbs\pa D$ for the side of $R_i$ that's parallel to $a_i$. We do the same for $b$, getting subarcs $\{b_j\}$, rectangles $Q_j\sbs E$, and parallel subarcs $b_j'\sbs\pa E$. 

\paragraph{Large arcs.} If $|a_i\cap b_j|\ge 3$ for some $i,j$, then $a_i$ and $b_j$ are called \emph{large}. Similarly, say $a_i'$ and $b_j'$ are large if $|a_i'\cap b_j'|\ge3$. Note that $|a_i\cap b_j|=|a_i'\cap b_j'|$, so if $a_i$ and $b_j$ are large, then $a_i'$ and $b_j'$ are large (and vice versa). 

\paragraph{Graphs $\Theta,\Theta'\sbs S$.} Define $\Theta\sbs X$ the union of all the large $a_i$ and $b_j$, and define $\Theta'$ the union of the large $a_i'$ and $b_j'$. The set $\Theta$ (and $\Theta'$ similarly) has the structure of a graph where every vertex has degree either $1$ (coming from endpoints of $a_i$ and $b_j$) or $4$ (coming from intersections of $a_i$ and $b_j$). In addition, $\Theta$ and $\Theta'$ are isomorphic as graphs in an obvious way. 

Define $Z,Z'$ as a small regular neighborhood of $\Theta,\Theta'$ respectively. We have the following properties, which are proved in \cite{masur-schleimer}. We give some brief explanation of the proofs to help assure the reader that these proofs hold in our setting. 
\begin{itemize}
\item \cite[Claim 12.7]{masur-schleimer}: The graph $\Theta$ is nonempty. 

Proof sketch. If $\Theta=\vn$, then all the $a_i,b_j$ are small. From this one deduces that $|a\cap b|\le128$. Since $a,b$ are components of $\pa D\cap X$ and $\pa E\cap X$, this implies that $d_X(D,E)$ can't be large. In fact, $d_X(D,E)\le 24$, contradicting our assumption. 

\item \cite[Claim 12.8]{masur-schleimer}: No component of either $\Theta$ or $\Theta'$ is contained in a disk in $S$ or an annulus in $S$ that's peripheral in $X$. 

Proof sketch. If some component of $\Theta$ is contained in a disk or peripheral annulus, then one argues that either some $a_i,b_j$ cut a bigon from $S$, or $a_i,b_j,\pa X$ cut a triangle contained in $X$. Both of these contradict our assumption that $D,E,X$ are in standard position. The argument for $\Theta'$ is the same.

\item \cite[Claim 12.9]{masur-schleimer}: Let $Z_1$ be a component of $Z$ , and let $d$ be a component of $\pa Z_1$. Then $d$ is either inessential or peripheral in $X$. 

Proof sketch. Supposing that $d$ is both essential and non-peripheral, one can deduce that $|a\cap d|<256$ (counting separately intersections with $d$ coming from large and small $a_i$). Similarly, $|b\cap d|<256$, and this implies that $d_X(D,E)<45$, a contradiction. 

\item $\Theta$ (and hence also $\Theta'$) is connected, and $\Theta$ fills $X$. 

Proof sketch. This follows from the previous items: any component $Z_1\sbs Z$ fills $X$ by \cite[Claim 12.9]{masur-schleimer}. Then any other component must be contained in a disk or peripheral annulus in $X$, which contradicts \cite[Claim 12.8]{masur-schleimer}. Thus $Z$ is connected and fills $X$, and this implies the same for $\Theta$. 
\end{itemize} 

Let $\ca R=\{R_i\}$ and $\ca Q=\{Q_j\}$ be the large rectangles. Observe that $\ca R\cup\ca Q$ is an $I$-bundle whose horizontal boundary is $\Theta\cup\Theta'$. We can thicken $\ca R\cup\ca Q$ to an $I$-bundle $T_0$ with horizontal boundary $Z\cup Z'$.

\paragraph{Case 1:} $\Theta$ and $\Theta'$ intersect nontrivially. Then $Z\cup Z'$ is a connected subsurface of $X$. Since complementary components of $Z$ in $X$ are disks and peripheral annuli, the same is also true for complementary components of $Z\cup Z'$ in $X$. 

In this case, the horizontal boundary $\pa_hT_0$ is connected. Then since $T_0$ is orientable (because $V$ is), the base space of the $I$-bundle is a nonorientable surface. See \S\ref{sec:I-bundles}.

The vertical boundary $\pa_vT_0$ is a union of annuli, each connecting a pair of components of $\pa(Z\cup Z')$. If $A$ is such an annulus, and both components of $\pa A$ bound disks in $X$, then the union of $A$ and these disks is an embedded 2-sphere, which can be filled by a 3-ball, which we parameterize as $D^2\ti I$ to preserve the $I$-bundle structure. In this way we obtain an $I$-bundle $T$ with the following properties: 
\begin{enumerate}
\item[(1)] For every component $A$ of $\pa_vT$, at least one component of $\pa A$  is peripheral in $X$. In particular the number of components of $\pa_vT$ is at most the number of components of $\pa X$.
\item[(2)] Every component of $\pa X$ is parallel to a boundary component of some annulus $A\sbs\pa_vT$ (because $Z\cup Z'$ fills $X$, so every component of $\pa X$ is parallel to a component of $\pa(Z\cup Z')$). In particular the number of components of $\pa_vT$ is at least half the number of components of $\pa X$. 
\end{enumerate} 

Next we use $T$ to show that every component of $\pa X$ is null-homologous in $V$. We separate into cases, depending on whether $X$ is homeomorphic to $\Sigma_{1,1}$, $\Sigma_{1,2}$, or $\Sigma_{0,4}$ (we do not need to consider $\Sigma_{0,2}$, $\Sigma_{0,3}$, or $\Sigma_2$). 

\paragraph{Case 1(a):} $X\cong \Sigma_{1,1}$. By properties (1) and (2), $\pa_vT$ has a single component $A$. One component of $\pa A$ bounds a disk in $X$, and the other is parallel to $\pa X$. Thus $\pa X$ bounds a disk in $V$. 

\paragraph{Case 1(b):} $X\cong \Sigma_{1,2}$. Then $\pa_vT$ has either one or two components. 

If $\pa_vT$ has one component $A$, then the two components of $\pa A$ are parallel to the two components of $\pa X$. If $A$ is compressible, then each component of $\pa X$ bounds a disk in $V$. Otherwise, if $A$ is incompressible, then $A$ can be isotoped into $S$; the proof of this is contained in \cite[Claim 12.16]{masur-schleimer}. From this we conclude that $\pa X$ bounds an embedded $(\R P^2\#\R P^2)\setminus D^2$ in $V$ because in this case $T$ is a bundle over $(\R P^2\#\R P^2)\setminus D^2$, and $(\R P^2\#\R P^2)\setminus D^2$ embeds as a section in the $I$-bundle, and the boundary of this section is isotopic to $\pa X$.

If $\pa_vT$ has two components $A_1,A_2$, then each $A_i$ has one boundary component parallel to a component of $\pa X$ and another boundary component that is inessential in $X$. Then each component of $\pa X$ bounds a disk in $V$ (if one component of $\pa X$ bounds a disk then the other does automatically since the boundary components of $X$ are necessarily parallel in $S$).

Note that in every case, the components of $\pa X$ are nonseparating curves that are null-homologous in $V$. 

\paragraph{Case 1(c):} $X\cong \Sigma_{0,4}$. In this case $\pa_vT$ can have two, three, or four components. 

If $\pa_vT$ has four components, then every component of $\pa X$ bounds a disk in $V$. 

If $\pa_vT$ has three components, then two components of $\pa X$ bound disks in $V$. The other two components of $\pa X$ are joined by a component $A$ of $\pa_vT$. If $A$ is compressible, then every component of $\pa X$ bounds a disk. If $A$ is incompressible, then $A$ is isotopic into $S$, again by \cite[proof of Claim 12.16]{masur-schleimer}. In this case, $T$ is an $I$-bundle over $\R P^2\setminus(D^2\cup D^2\cup D^2)$ (because $\pa_h T\cong\Si_{0,6}$, c.f.\ \S\ref{sec:I-bundles}), and there is an embedding of $\R P^2\setminus D^2$ (the M\"obius band) with boundary isotopic to the components of $\pa X$ parallel to $\pa A$. 

Finally suppose that $\pa_vT$ has two components. In this case $\pa_hT\cong\Si_{0,4}$, so $T$ is an $I$-bundle over $\R P^2\setminus (D^2\cup D^2)$. 

First suppose that $X$ is incompressible (i.e.\ there is no disk $D\sbs V$ supported in $X$, primitive or not) and that $X$ is a hole for the disk complex $\ca D(V)$. Then since the curve complex of $B:=\R P^2\setminus (D^2\cup D^2)$ has diameter $\le 4$ \cite[\S2]{scharlemann-nonorientable} and the natural map $\ca C(B)\ra\ca C(X)$ is distance non-increasing \cite[\S6]{masur-schleimer}, this contradicts the fact that the diameter of $X$ with respect to $\ca P(V)$ is $\ge 57$ (any primitive disk $D$ can be surgered to a (not necessarily primitive) vertical disk $D'$ with $d_X(D,D')\le6$ \cite[Lems.\ 8.12 and 11.7]{masur-schleimer}). See also the argument of \cite[\S12.18]{masur-schleimer}.

If $X$ is incompressible, but is not a hole for $\ca D(V)$, then two components of $\pa X$ bound disks, and this implies that at least one component of $\pa_vT$ is compressible. Suppose the other component $A$ is incompressible (otherwise every component of $\pa X$ bounds a disk), then similar to a previous case, we find that there are two components of $\pa X$ that each bound an embedded $\R P^2\setminus D^2$ in $V$. 

Now suppose that $X$ is compressible, i.e.\ there exists a (not necessarily primitive) disk $D\sbs V$ with $\pa D\sbs X$. If each component of $\pa_vT$ is incompressible, then $D$ can be isotoped to be contained in $T$, which contradicts the fact that the horizontal boundary component of an $I$-bundle is incompressible \cite[Ex.\ 5.5]{masur-schleimer}. Therefore, at least one component of $\pa_vT$ is compressible, and we conclude as in the previous paragraph.



Overall, we find that in every case, every component of $\pa X$ is null-homologous in $V$. Note that the only case when $\pa X$ has a component that is separating in $S$ is when $X\cong\Sigma_{1,1}$. 

\paragraph{Case 2:} $\Theta$ and $\Theta'$ are disjoint. In this case $T_0\cong Z\ti I$ with horizontal boundary $\pa_hT_0=Z\sqcup Z'$ and $Z\sbs X$ and $Z'\sbs Y$. 

{\it Claim.} $\xi(Y)\ge\xi(X)$, where $\xi$ denotes the complexity $\xi(\Sigma_{g,b})=3g-3+b$. 

Since $X\sbs S$ is homeomorphic to one of $\Sigma_{1,1}, \Sigma_{1,2}, \Sigma_{0,4}\sbs \Sigma_2$, the claim implies that $X$ and $Y$ are both homeomorphic to $\Sigma_{1,1}$. 

{\it Proof of Claim.} Since $S=X\cup Y$ (disjoint subsurfaces glued along boundary), $X$ and $Y$ have the same number of boundary components. Thus to prove the claim, it suffices to show that $g(Y)\ge g(X)$. Since $Z$ fills $X$, $g(X)=g(Z)$. Since $Z\cong Z'$ embeds in $Y$, $g(Y)\ge g(Z)$. These combine to give $g(Y)\ge g(X)$. This proves the claim. 

As a consequence, we also conclude that $Z'$ fills $Y$, since if $b\ge1$, then any embedded subsurface $\Sigma_{g,b}\hra \Sigma_{g,1}$ fills. 

Every component $\pa_vT_0$ is an annulus $A$ connecting a component of $\pa Z$ to a component of $\pa Z'$. If both components of $\pa A$ are inessential, then we obtain a 2-sphere that can be filled with a 3-ball ($\cong D^2\ti I$). In this way, we enlarge $T_0$ to an $I$-bundle $T$ so that every component of $\pa_vT$ is an annulus $A$ so that at least one component of $\pa A$ is isotopic to $\pa X\simeq\pa Y$. 

Consequently, the vertical boundary $\pa_vT$ has either one or two components. 

If $\pa_vT$ has two components, then it follows that $\pa X$ bounds a disk in $V$. 

Suppose $\pa_vT$ has one component $A$. Then $T\cong \Sigma_{1,1}\ti I$. The two components of $\pa A$ are parallel to $\pa X$ and $\pa Y$. If $A$ is compressible, then $\pa X$ bounds a disk in $V$ (and same for $\pa Y$). If $A$ is incompressible, then $A$ is isotopic into $S$ by \cite[proof of Claim 12.16]{masur-schleimer}. This implies that $T\hra V$ is isotopic to a homeomorphism. 

\paragraph{Step 2.} In this step we show that if $X$ has diameter $\ge 61$, then we can also construct an embedded $I$-bundle $\what T$ in $\what V$. 

\begin{lem}Assume that $X$ is a primitively incompressible hole for $\ca P(V)$ with diameter $\ge 57$. Then $X$ is also a hole for $\ca P(\what V)$. 
\end{lem}

\begin{proof}
Suppose for a contradiction that $X$ is not a hole for $\ca P(\what V)$. This means that there is a primitive disk $\what D\sbs\what V$ that is supported in $S\setminus X$. 

If $X\cong\Sigma_{1,2}$ or $X\cong\Sigma_{0,4}$, then by Step 1, every component of $\pa X$ is nonseparating and nullhomologous in $V$. This implies that no component of $\pa X$ is isotopic to $\pa\what D$, since otherwise this would imply that $H_2(S^3)\neq0$ by Lemma \ref{lem:mayer-vietoris}. 

Assume now that $X$ is homeomorphic to $\Sigma_{1,1}$. By Step 1, $\pa X$ bounds a disk in $V$. 

{\it Claim.} $\pa X$ also bounds a disk in $\what V$.

The claim implies that $\pa X$ is a reducing sphere, which implies that there is a (unique) primitive disk $D\sbs V$ with $\pa D\sbs X$. This contradicts the fact that $X$ is primitively incompressible. Thus to prove the lemma it suffices to prove the claim. 

{\it Proof of Claim.} We want to show that $\pa X$ is homotopically trivial in $\what V$. Let $\what E\sbs\what V$ be a primitive disk that is disjoint from $\what D$. An element of $\pi_1(\what V)$ is determined by its intersection with $\what D$ and $\what E$ (Remark \ref{rmk:primitive-test}). Denoting $\de,\ep\in\pi_1(\what V)$  the free generators dual to $\what D$ and $\what E$, then the word $w\in\pair{\de,\ep}$ determined by $\pa X$ is a power of $\epsilon$ since $\pa X$ is disjoint from $\what D$. In addition $\pa X$ is in the commutator subgroup since $\pa X$ bounds $X\cong\Sigma_{1,1}$. This implies $w$ is trivial. 
\end{proof}

Next we show that the diameter of $X$ with respect to $\ca P(\what V)$ is $\ge 57$. By assumption, there exist primitive disks $D,E\sbs V$ with $d_X(D,E)\ge 61$. Let $\what D,\what E\sbs \what V$ be the (unique) dual pair of primitive disks. In particular, $D\cap\what E=\vn=\what D\cap E$, which implies $d_X(D,\what E),d_X(\what D,E)\le 2$ \cite[Lem.\ 2.3]{masur-minsky2}. Now the triangle inequality gives $d_X(\what D,\what E)\ge 57$ and hence the diameter bound.  

Finally, $X$ is primitively incompressible in $\what V$ because primitively compressible holes for $\ca P(\what V)$ have diameter $\le 13$ by Proposition \ref{prop:compressible}.

\paragraph{Step 3.} By Steps 1 and 2, there are embedded $I$ bundles $T\hra V$ and $\what T\hra \what V$. Both $T$ and $\what T$ have $X$ as a component of their horizontal boundary. We will show that most of the cases enumerated in Step 1 are impossible, for a variety of reasons. 

First suppose that $X$ is homeomorphic to $\Sigma_{1,2}$ or $\Sigma_{0,4}$. By Step 1, every component of $\pa X$ is null-homologous in both $V$ and $\what V$. Since the components of $\pa X$ are non-separating in $S$, this implies that $H_2(S^3)\neq0$ (Lemma \ref{lem:mayer-vietoris}), which is a contradiction. 

Next suppose that $X\cong \Sigma_{1,1}$. Either $\pa X$ bounds a disk in both, exactly one, or neither of $V$ and $\what V$. We consider these cases separately. Recall from Step 1 that if $\pa X$ does not bound a disk in $V$ then the inclusion $T\hra V$ is (isotopic to) a homeomorphism (and similarly for $\what V$). 

If $\pa X$ bounds a disk in both, then together these disks form a reducing sphere for the Heegaard splitting. In particular, on either side of this reducing sphere we can find primitive disks in $V$, but this contradicts the fact that $X$ is a hole for $\ca P(V)$. 

Suppose that $\pa X$ bounds a disk in $\what D\sbs \what V$ and that $T\hra V$ is a homeomorphism. The disk $\what D$ separates $\what V$ into two genus-1 handlebodies, one containing $X$, and each containing a unique disk $\what D',\what D''$. These disks are primitive, since we can easily find a vertical disk in $T\cong V$ that intersects $\what D'$ or $\what D''$ once. But this implies that $X$ is not a hole for $\ca P(\what V)$, contradicting Step 2.

The only remaining possibility is that both $T\hra V$ and $\what T\hra \what V$ are homeomorphisms, which is the desired conclusion of the Theorem. 
\end{proof}

\subsection{Subsurface projection for the trefoil knot}\label{sec:trefoil}

In this section we prove the last part of Theorem \ref{thm:holes}, showing if $X$ is a large-diameter hole for $\ca P(V)$, then there exists $g\in\bb G$ that preserves $X$ and so that $\rest{g}{X}$ is pseudo-Anosov. Given the discussion in the introduction to \S\ref{sec:holes}, it remains to show that if $X$ is a genus-1 Seifert surface for the trefoil knot, then $X$ has bounded diameter. 

In this section we write $S^3=V\cup\what V$ as $(F\ti I)\cup_\phi(\what F\ti I)$ with $F=\what F=\Sigma_{1,1}$ and $\phi=T_aT_b$. We also denote $\Sigma=\pa(F\ti I)$. For more on this notation, see \S\ref{sec:fibered-link}. 

\begin{thm}\label{thm:finite-diameter}
Fix $F=\what F=\Sigma_{1,1}$ and fix $\phi=T_aT_b\in\Mod(F)$. Write $S^3=(F\times I)\cup_\phi(\what F\times I)$, and identify this with the standard Heegaard splitting $S^3=V\cup\what V$. Then the subsurfaces $X=F\ti\{1\}$ and $Y=F\ti\{0\}$ have finite diameter with respect to $\ca P(V)$. 
\end{thm}

\begin{proof}[Proof of Theorem \ref{thm:finite-diameter}]
The proof for $X$ and $Y$ is the same. For concreteness we work with $X$. Let $\ca V\sbs\ca P(V)$ denote the set of vertical disks $D\sbs F\ti I\cong V$. First  we show that $\ca V$ is finite. Then we use this to show $X$ has finite diameter with respect to $\ca P(V)$.


\paragraph{Finitely many vertical primitive disks.} 
First we specify the coordinates we will use to compute. Consider the curves and arcs in Figure \ref{fig:S11-cut}; as pictured, we cut $F=\Sigma_{1,1}$ along the arcs $a',b'$ to get an octagon, which we use to draw pictures. 

\begin{figure}[h!]
\labellist
\pinlabel $a$ at 230 880
\pinlabel $b$ at 380 890
\pinlabel $a'$ at 560 830
\pinlabel $b'$ at 400 720
\pinlabel $a'$ at 810 730
\pinlabel $b$ at 820 890
\pinlabel $a$ at 860 830
\pinlabel $b'$ at 930 850
\endlabellist
\centering
\includegraphics[scale=.3]{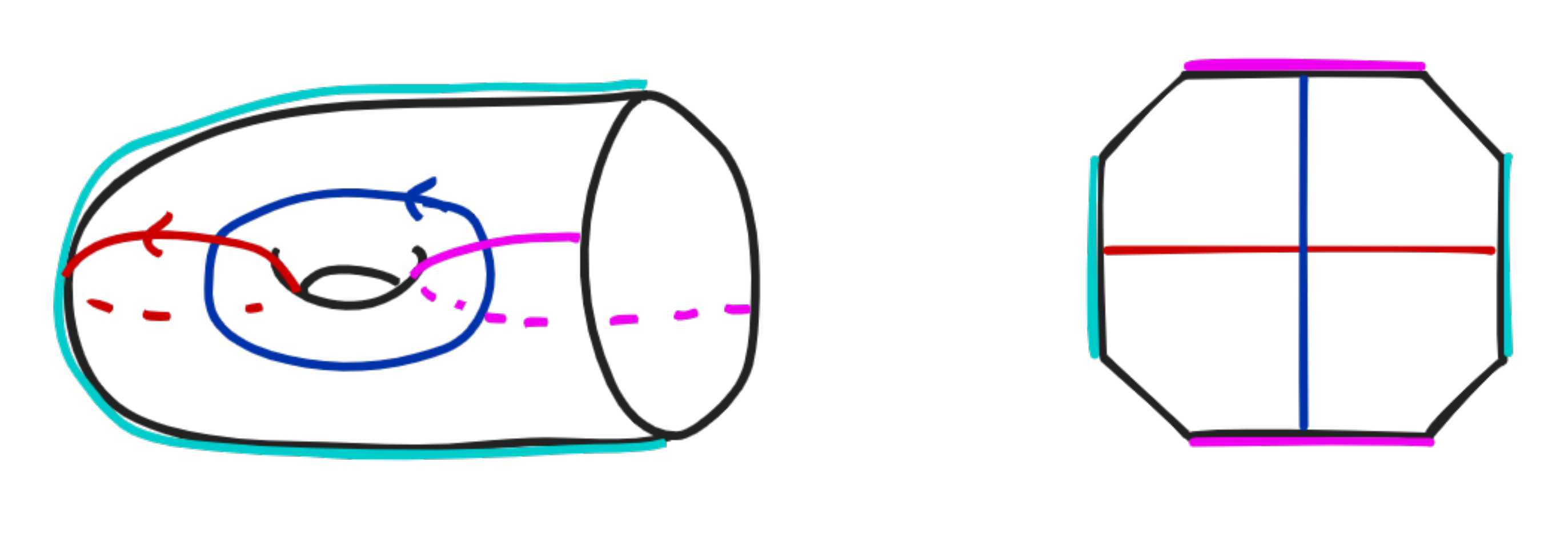}
\caption{Coordinates for the computation.} 
\label{fig:S11-cut}
\end{figure}
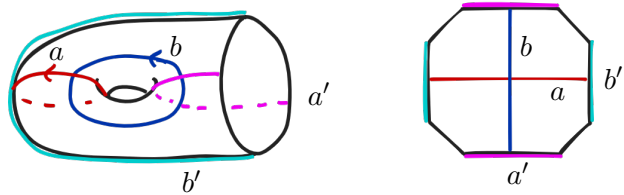

Figure \ref{fig:dual-primitive-ab} shows the intersection of two vertical primitive disks in $\what F\ti I$ with $F\ti0$ and $F\ti 1$.

\begin{figure}[h!]
\labellist
\pinlabel $\pa\what E_1$ at 100 90
\pinlabel $\pa\what E_2$ at 180 160
\pinlabel $\pa\what E_1$ at 415 80
\pinlabel $\pa\what E_2$ at 400 160
\endlabellist
\centering
\includegraphics[scale=.5]{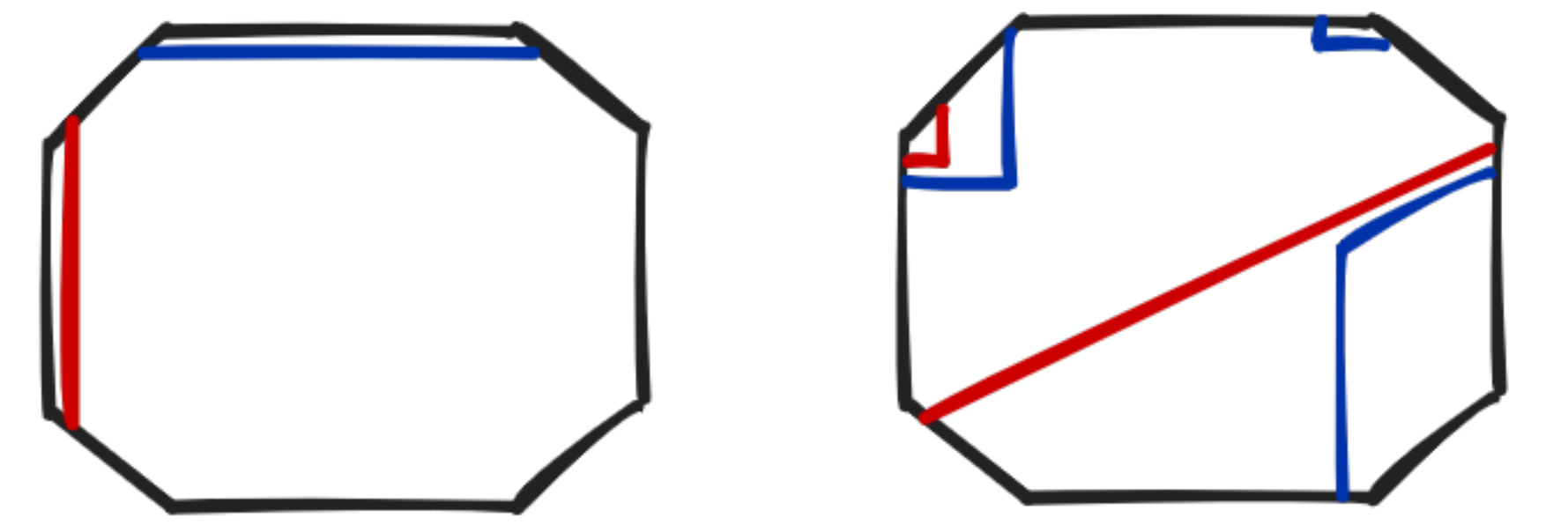}
\caption{Dual primitive disks.}
\label{fig:dual-primitive-ab}
\end{figure}
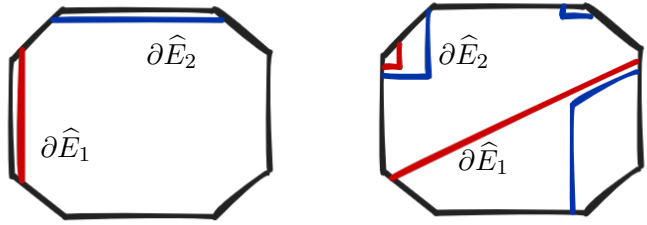

Given an arc $\lambda\subset \Sigma_{1,1}$, we define the slope as $\fr{\lambda\cdot a}{\lambda\cdot b}\in\Q\cup\{\infty\}$, where the dot product denotes the algebraic intersection number. This depends on our fixed orientation of $a,b$, but doesn't depend on how we orient $\lambda$. For any slope $\frac{p}{q}\in\Q\cup\{\infty\}$, there is a unique arc with that slope; see Figure \ref{fig:slopes-pos-neg}.

\begin{figure}[h!]
\labellist
\pinlabel $p-1$ at 350 660
\pinlabel $q-1$ at 160 520
\pinlabel $p-1$ at 655 660
\pinlabel $q-1$ at 850 520
\endlabellist
\centering
\includegraphics[scale=.4]{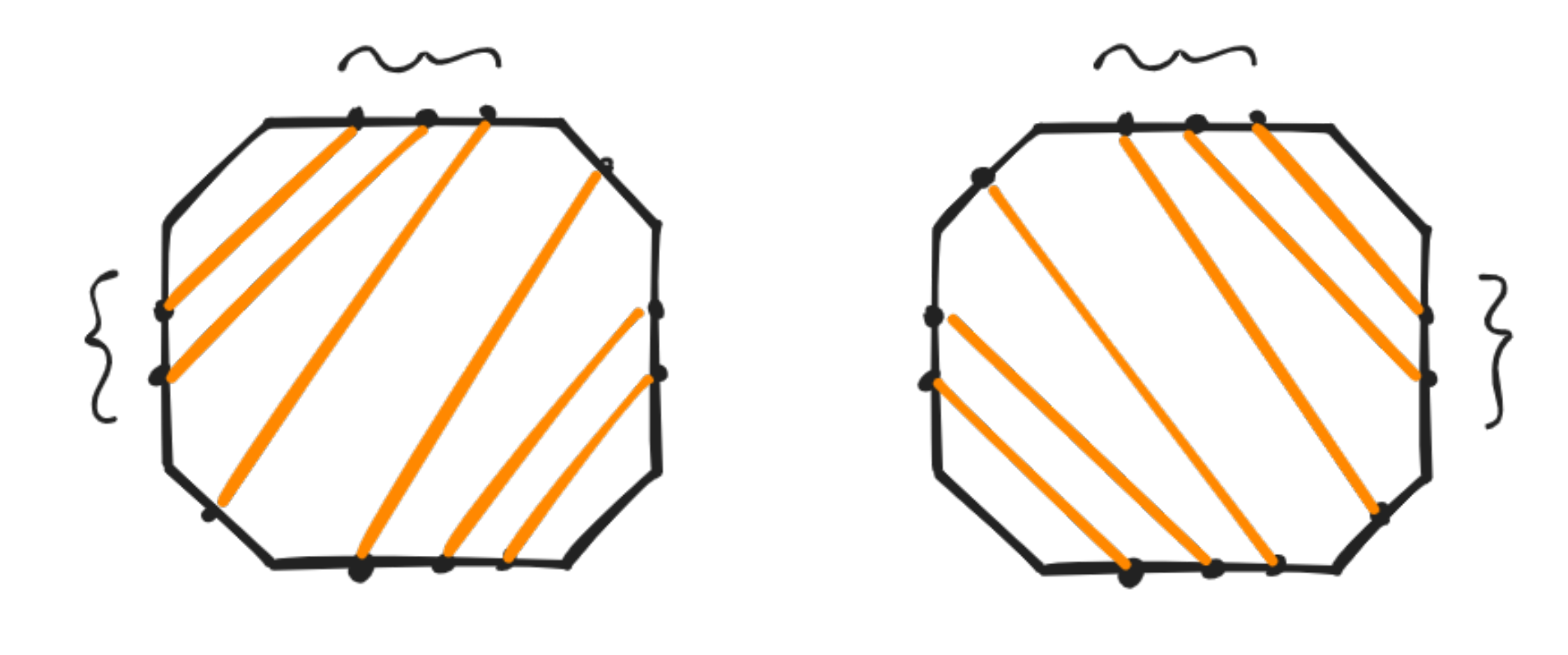}
\caption{Arcs on $\Sigma_{1,1}$ of slope $4/3$ (left) and $-4/3$ (right).}
\label{fig:slopes-pos-neg}
\end{figure}
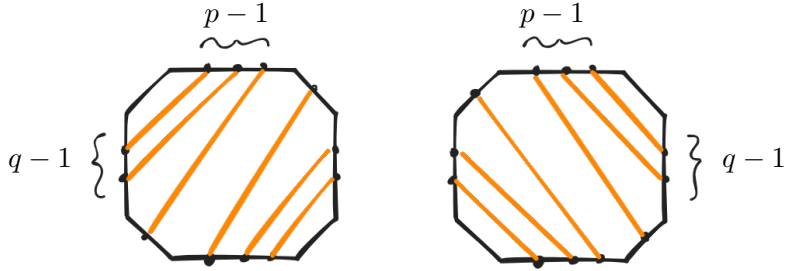

It is easy to check that the vertical disks corresponding to the arcs of slope 
\[\pm1,0, \infty, 1/2, 2\] are all primitive. These fall into two orbits of $\phi=T_aT_b$, which has order 3 in $\PSL_2(\Z)$. We show that no other vertical disk is primitive. Let $D=D_{p/q}$ be the vertical disk corresponding to an arc of slope $\fr{p}{q}$. If $\fr{p}{q}\notin\{0,1,\infty\}$, then up to the action of $\phi$, we can assume that $0<\frac{p}{q}<1$. The word $w$ in $r,b$ corresponding to $\partial D$ contains $(q-1)$ occurrences of $r$ and $b$, $(q-1-p)$ occurrences of $r^{-1}$, and $(p-1)$ occurrences of $b^{-1}$. See Figure \ref{fig:TaTb-slope}. Here it's important to remember that the orientations on $F\times 0$ and $F\times 1$ are induced from $F\times I$, and the obvious homeomorphism $F\times0\cong F\times 1$ is orientation-reversing. 

It's easy to check that $w$ is cyclically reduced (the subword corresponding to $F\ti0$ starts and ends with $r$, and the subword corresponding to $F\ti1$ starts and ends with $b$, so there is no cancellation). A cyclically-reduced primitive word doesn't contain both $r$ and $r^{-1}$, and also doesn't contain both $b$ and $b^{-1}$. Then at least one of $q-1$, $p-1$, or $q-1-p$ is equal to $0$. The case $q-1=0$ is impossible since $0<\frac{p}{q}<1$. If $p-1=0$, then $w=r^q(br^{-1})^{q-1}b$, which implies that $q=1$, which again is impossible. Finally, if $q-1-p=0$, then $w=r(b^{-1}r)^{p-1}b^p$, which implies that $p=1$, so $\frac{p}{q}=\frac{1}{2}$. 

\begin{figure}[h!]
\labellist
\endlabellist
\centering
\includegraphics[scale=.5]{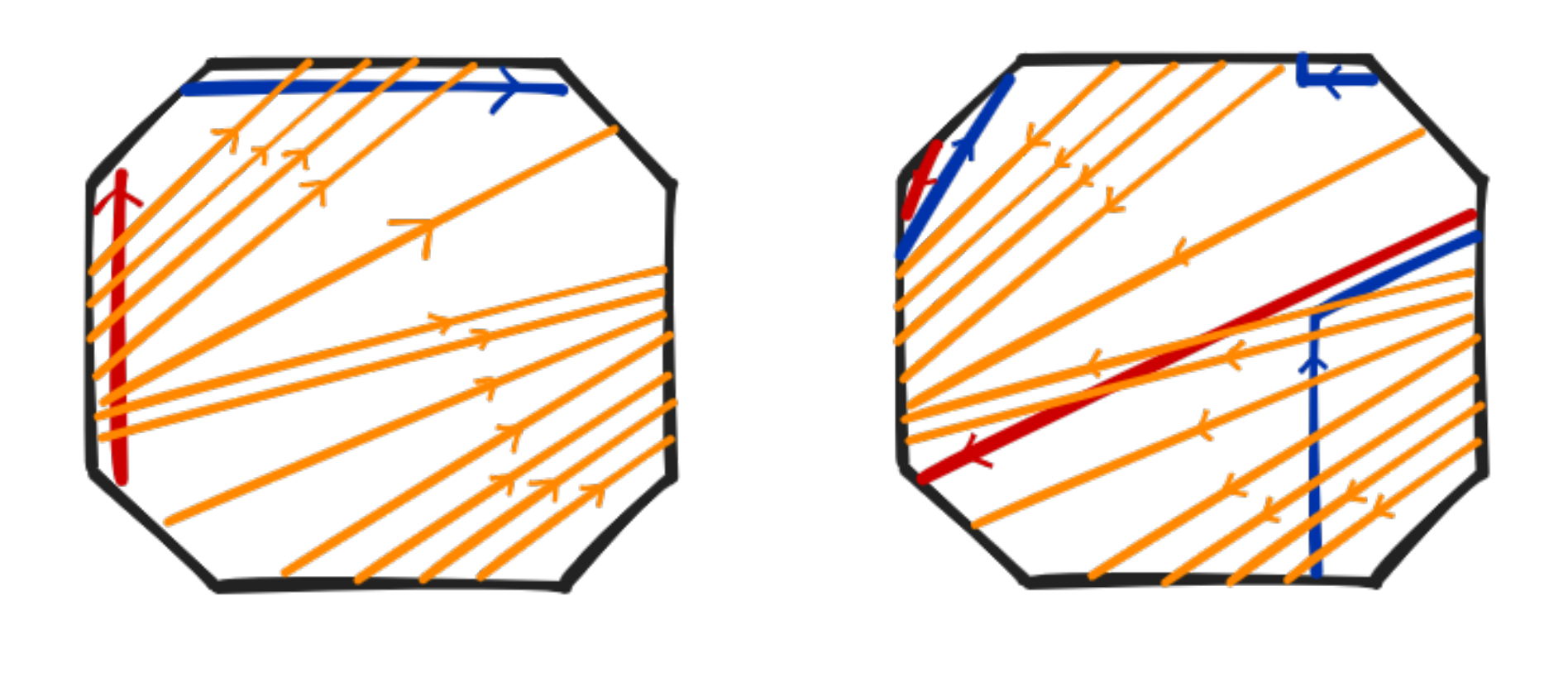}
\caption{Arc slope $0<\fr{p}{q}<1$ pictured on $F\ti 0$ and $F\ti 1$.}
\label{fig:TaTb-slope}
\end{figure}

This proves that the only vertical primitive disks are the ones corresponding to arcs of slope $\fr{p}{q}\in\{\pm1,0, \infty, 1/2, 2\}$.

\paragraph{$X$ has finite diameter.} 
Fix a primitive disk $D_0\sbs V$. By Lemma \ref{lem:surgery-sequence}, there exists primitive $D\sbs V$ such that $D$ does not admit a primitive boundary compression into $S\setminus n(\pa X)$ and $d_{X}(D_0,D)\le 6$. To prove that $X$ has finite diameter, it suffices to show that $D$ is vertical, since we showed above that there are only finitely many vertical disks. The argument below is similar to \cite[Lem.\ 8.12]{masur-schleimer} but is more subtle.

Let $D_0,D_1,D_\infty\subset F\ti I$ be the vertical primitive disks corresponding to slopes $0,1,\infty$. These disks split $F\ti I$ into two hexagonal prisms. Let $D\subset F\ti I$ be a primitive disk, and assume that $D$ cannot be primitively compressed into $S\setminus n(\pa X)$. We show that $D$ can be isotoped to be vertical. 

First we can isotope $D$ to be be in minimal position with $\pa X$, and by a further isotopy we can assume that $D$ is vertical on $\pa F\ti I$. Next we can isotope $D$ to be vertical on $D_0\cup D_1\cup D_\infty$, since otherwise there would exist a primitive compression of $D$ into $S\setminus n(\pa X)$. Compare with \cite[Proof of Lem.\ 8.12]{masur-schleimer}.

It remains to show we can isotope $D$ to be vertical in the two hexagonal prisms. Fix one prism $H\ti I$, and let $D'$ be the intersection of $D$ with $H\times I$. The surface $D'$ can be built from $D'\cap H\ti [1-\ep,1]$ by a sequence of handle attachments. The first handle attachment is dual to boundary compression $(B,a,b)$ of $D'$ into $X=F\ti 1$; here $b\subset X$ is an arc connecting two components $\delta_1,\delta_2$ of $D'\cap X$, and $a$ is the core of the handle being attached. One can deduce that the bigon $B$ is primitive (defined in \S\ref{sec:surgery}) by showing that $B$ is an outermost bigon cut from a vertical primitive disk $D''$ by $D'$. To see this, observe that either the arc of slope $1/2$ or $-1$ on $X$ can be isotoped to intersect both $\delta_1$ and $\delta_2$ exactly once. See Figure \ref{fig:primitive-compression} for an example. It follows that $D'$ cuts a bigon isotopic to $B$ from either $D''=D_{1/2}$ or $D''=D_{-1}$. This contradicts the fact that $D$ is primitively incompressible, and from this we deduce that we can isotope $D$ to be vertical in each hexagonal prism, as desired.
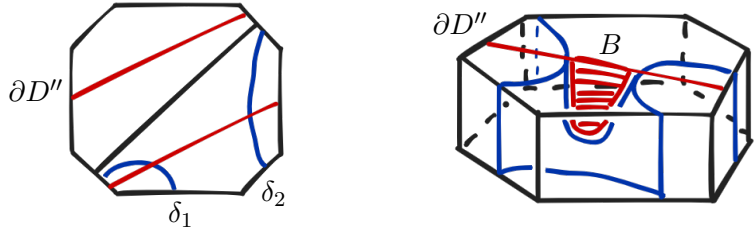
\begin{figure}[h!]
\labellist
\pinlabel $\de_1$ at 340 460
\pinlabel $\de_2$ at 430 480
\pinlabel $\pa D''$ at 200 580
\pinlabel $\pa D''$ at 610 645
\pinlabel $B$ at 755 625
\endlabellist
\centering
\includegraphics[scale=.4]{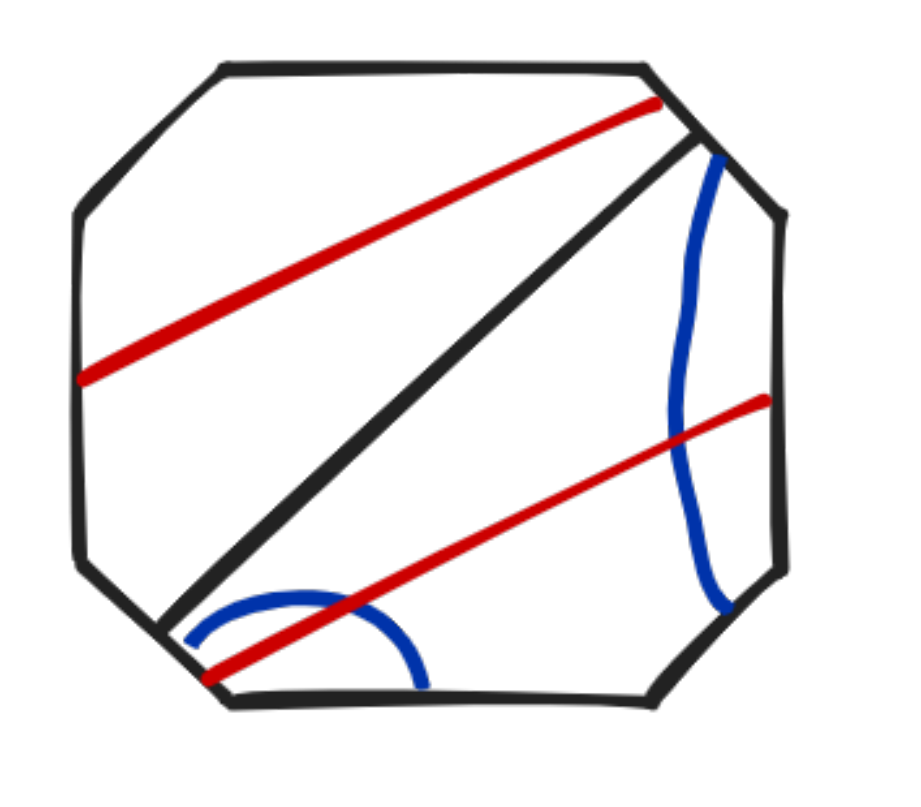}\hspace{.5in}
\includegraphics[scale=.4]{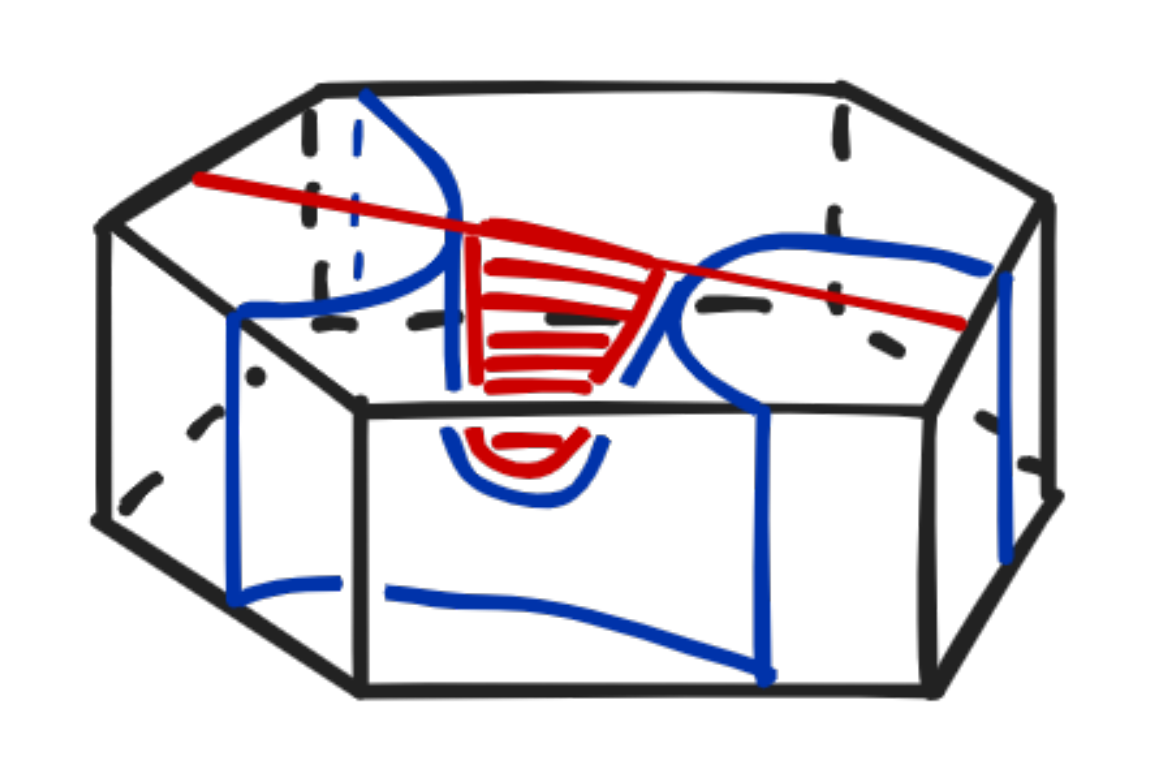}
\caption{Left: $F\times 1$. Right: hexagonal prism $H\times [1-\ep,1]$ and a bigon $B$ cut from the vertical disk $D''$ by $D'$.}
\label{fig:primitive-compression}
\end{figure}
\end{proof} 

This completes the proof of Theorem \ref{thm:holes}.  

\subsection{Subsurface projection for the figure-8 knot}\label{sec:fig8}

In this section, similar to \S\ref{sec:trefoil}, we write $S^3=V\cup\what V$ as $(F\ti I)\cup_\phi(\what F\ti I)$ with $F=\what F=\Sigma_{1,1}$, but now with $\phi=T_aT_b^{-1}$. 

Here we compute the set $\ca V\sbs \ca P(V)$ of primitive disks that are vertical in $F\ti I\cong V$, and prove that the subcomplex $\ca P(V;X)\subset\ca P(V)$ spanned by $\ca V$ is quasi-isometric to a line. This computation is necessary for proving Theorem \ref{thm:distance} (distance formula). 




\begin{figure}[h!]
\labellist
\pinlabel $\pa\what E_1$ at 70 700
\pinlabel $\pa\what E_2$ at 120 750
\pinlabel $\pa\what E_1$ at 495 680
\pinlabel $\pa\what E_2$ at 420 730
\endlabellist
\centering
\includegraphics[scale=.5]{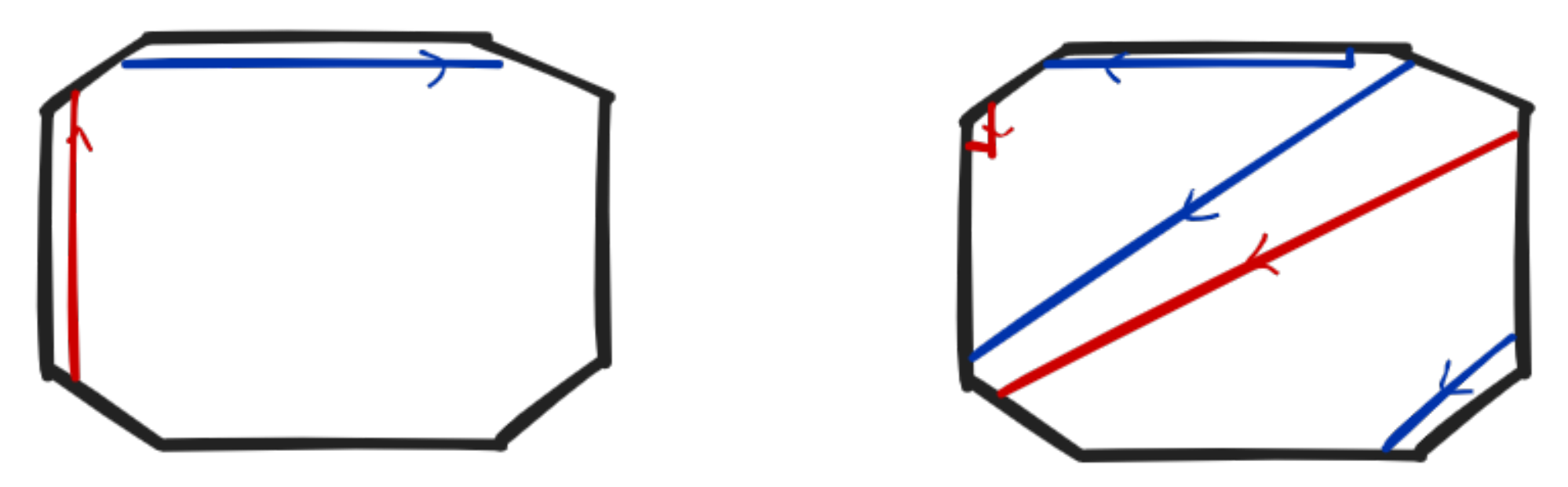}
\caption{Dual primitive disks.}
\label{fig:dual-primitive-ab-inv}
\end{figure}
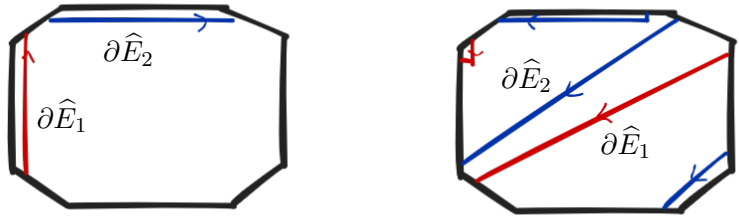

We begin by identifying $\ca V$. As in the previous section, let $D_{p/q}$ be the vertical disk whose intersection with $X$ has slope $p/q$. Figure \ref{fig:dual-primitive-ab-inv} shows the intersection of two (vertical) primitive disks in $\what F\ti I$ with $F\ti0$ and $F\ti 1$. Using these, it is easy to see that $D_{0}$ and $D_\infty$ are primitive. Then the same is true for the orbits of $\left(\begin{array}{c}1\\0\end{array}\right)$ and $\left(\begin{array}{c}0\\1\end{array}\right)$ under $\phi=T_aT_b^{-1}=\left(\begin{array}{cc}2&1\\1&1\end{array}\right)$. See Figure \ref{fig:farey}. 

\begin{figure}[h!]
\labellist
\small
\pinlabel $-\fr{21}{34}$ at 270 440
\pinlabel $-\fr{8}{5}$ at 340 350
\pinlabel $-\fr{3}{2}$ at 400 300
\pinlabel $-\fr{1}{1}$ at 500 230
\pinlabel $\fr{0}{1}$ at 645 200
\pinlabel $\fr{1}{2}$ at 780 225
\pinlabel $\fr{3}{5}$ at 880 280
\pinlabel $\fr{8}{13}$ at 950 340
\pinlabel $\fr{21}{34}$ at 1005 430
\pinlabel $-\fr{34}{21}$ at 270 760
\pinlabel $-\fr{13}{8}$ at 330 860
\pinlabel $-\fr{5}{3}$ at 400 930
\pinlabel $-\fr{2}{1}$ at 500 980
\pinlabel $\fr{1}{0}$ at 640 1000
\pinlabel $\fr{1}{1}$ at 780 975
\pinlabel $\fr{2}{3}$ at 880 920
\pinlabel $\fr{5}{8}$ at 960 830
\pinlabel $\fr{13}{21}$ at 1015 740
\endlabellist
\centering
\includegraphics[scale=.2]{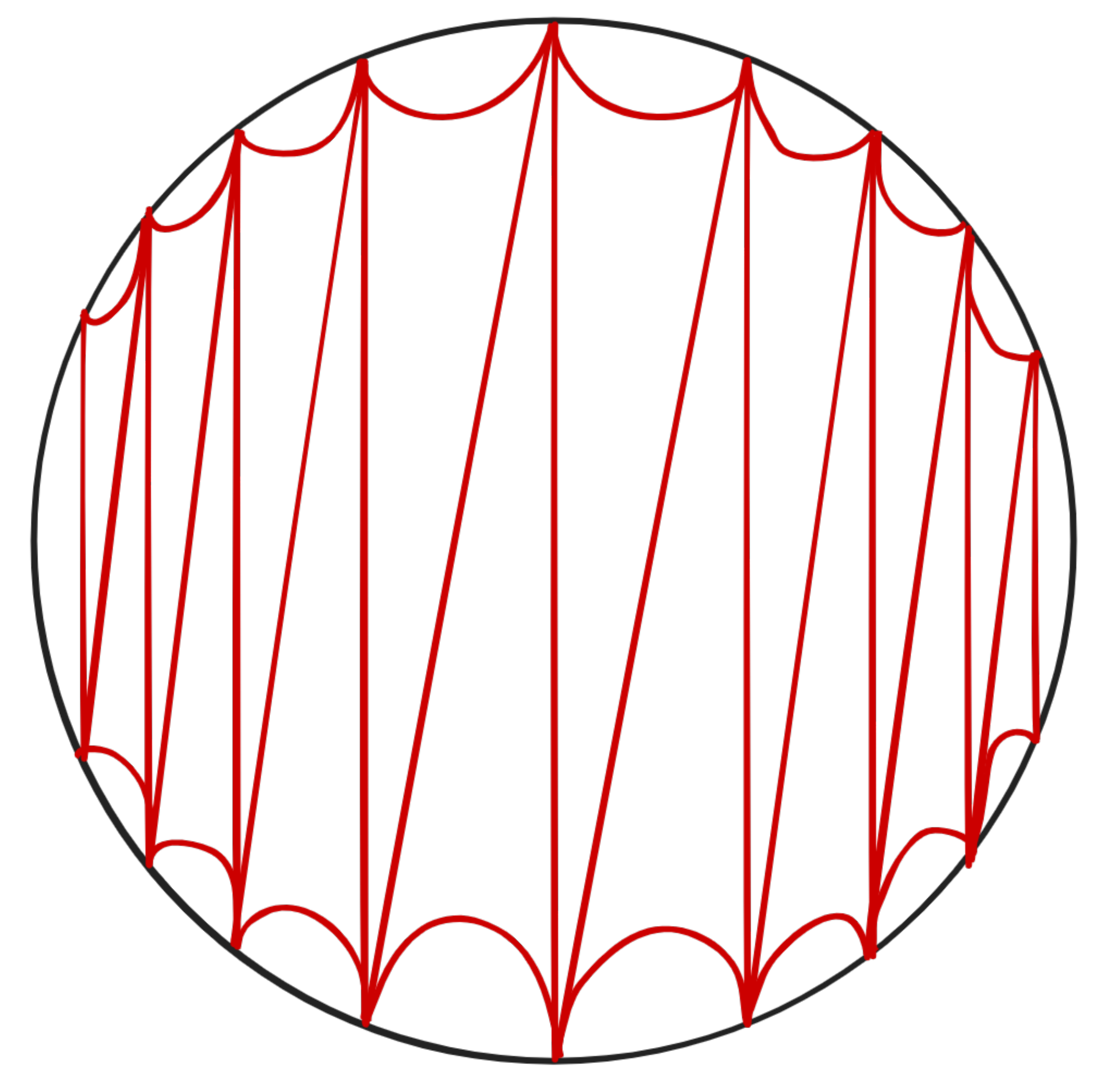}
\caption{The complex $\ca P(V;X)$ embeds in the Farey graph.}
\label{fig:farey}
\end{figure}

We show that any vertical primitive disk is in one of these two orbits. Fix $D_{p/q}$ such that $\fr{p}{q}$ is not in one of the above two orbits. Up to the action of $\phi$, we can assume that $\fr{p}{q}$ belongs to one of the two intervals $(-1,0)$ or $(1,\infty)$ in $\Q$. 

First suppose that $\fr{p}{q}\in(1,\infty)$. The word $w\in\pair{r,b}$ corresponding to $D_{p/q}$ has $(2p-q-1)$ occurrences of $b$ and $(p-1)$ occurrences of $b^{-1}$. See Figure \ref{fig:fig8-slopes}. It's easy to check that $w$ is cyclically reduced (the subword corresponding to $F\times 0$ starts and ends with $b^{-1}$, and the subword corresponding to $F\times 1$ starts and ends with $r$, so there is no cancellation). Then either $2p-q-1=0$ or $p-1=0$. Both of these are incompatible with the assumption that $\fr{p}{q}\in(1,\infty)$. This shows $D_{p/q}$ is not primitive when $\fr{p}{q}\in(1,\infty)$. 

\begin{figure}[h!]
\labellist
\endlabellist
\centering
\includegraphics[scale=.5]{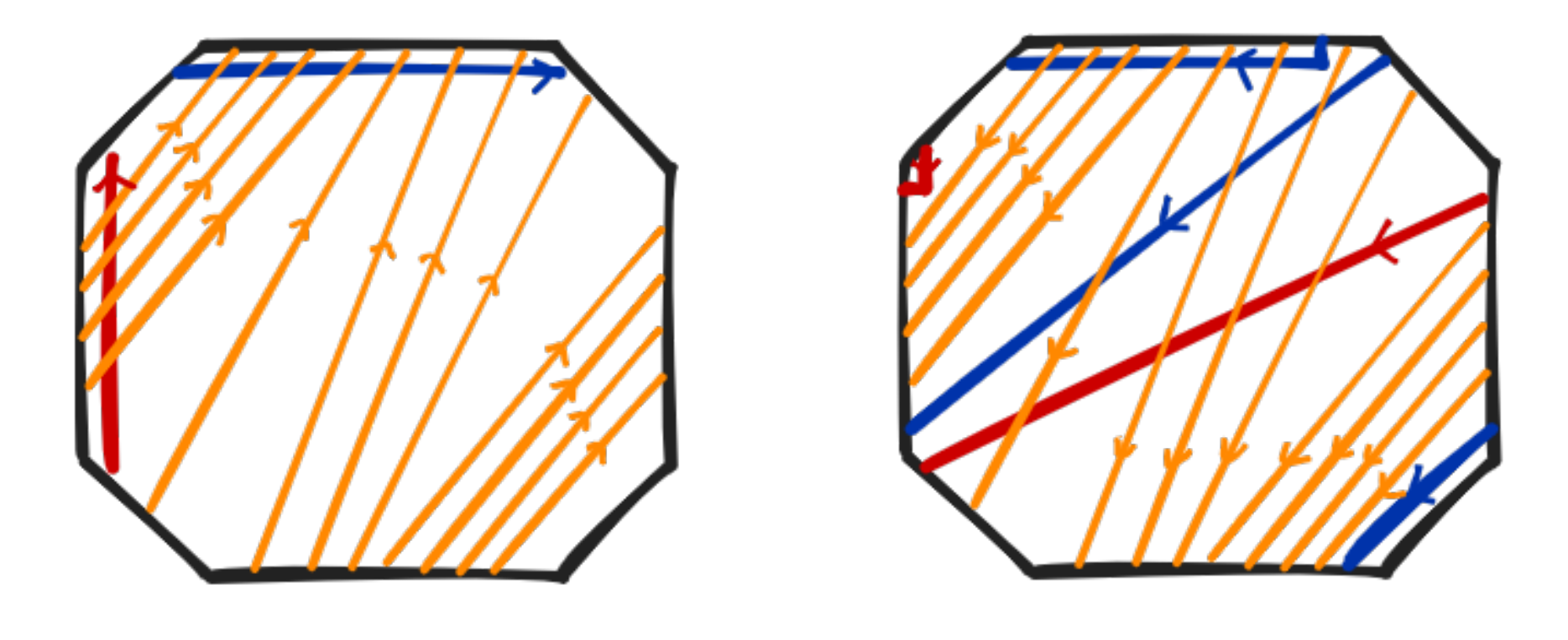}
\includegraphics[scale=.5]{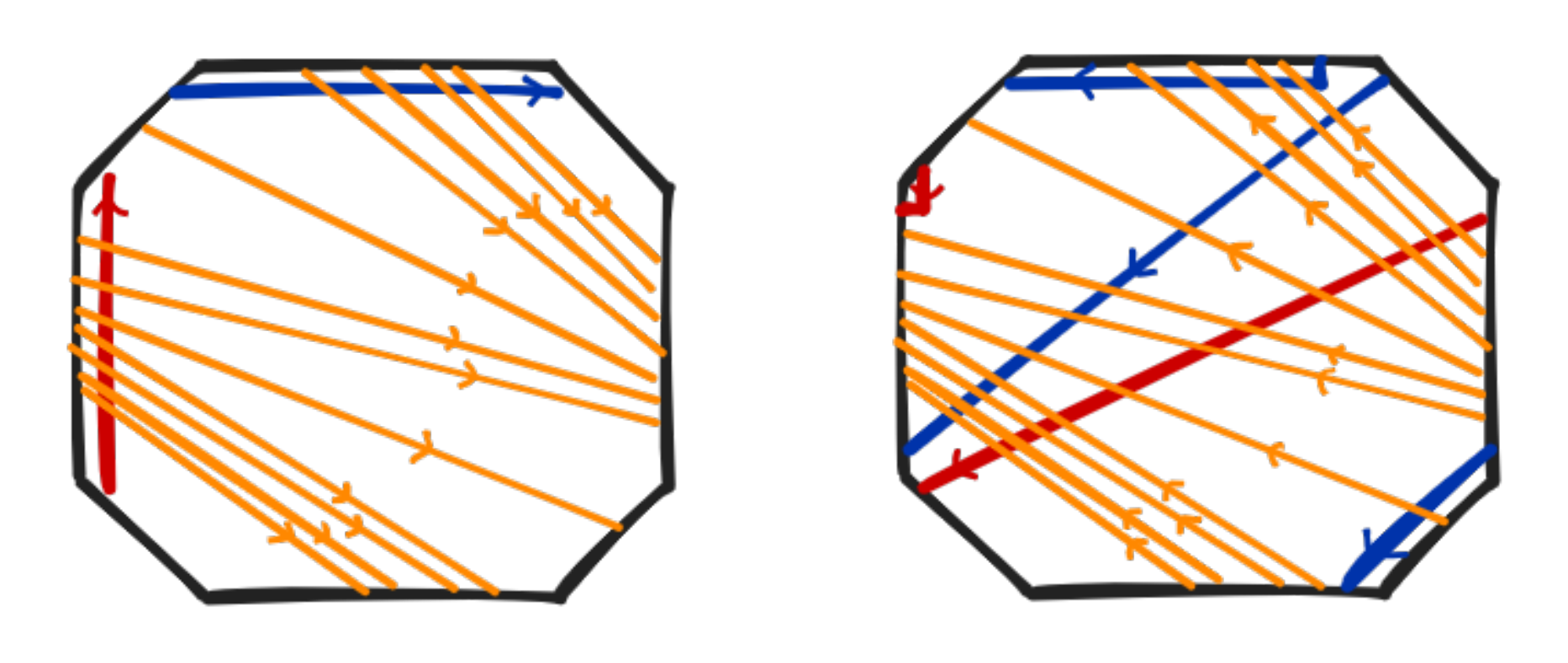}
\caption{An arc with slope $1<\fr{p}{q}<\infty$ (top row) or $-1<\fr{p}{q}<0$ (bottom row) pictured on $F\ti 0$ and $F\ti 1$.}
\label{fig:fig8-slopes}
\end{figure}

Next suppose that $\fr{p}{q}\in(-1,0)$. The word $w\in\pair{r,b}$ corresponding to $D_{p/q}$ has $(|q|-1)$ occurrences of $r$ and $(|p|+|q|-1)$ occurrences of $r^{-1}$. See Figure \ref{fig:fig8-slopes}. The word $w$ is cyclically reduced because the subword corresponding to $F\times 0$ starts and ends with $r$, the subword corresponding to $F\times 1$ starts and ends with $b^{-1}$, which implies there is no cancellation. Either of the condition $|q|-1=0$ or $|p|+|q|-1=0$ is incompatible with the assumption that $\fr{p}{q}\in(-1,0)$. This shows that $D_{p/q}$ is not primitive. 


Now we claim that for every vertex $D\in\ca P(V)$ there is a vertical disk $D'$ so that $d_X(D,D')\le 6$. This follows by the exact same argument as in \S\ref{sec:trefoil}; the special fact we needed for that argument is that the vertical disks with slopes $\pm 1, 0, \infty, 1/2$ are primitive; this also holds in the current case. 

It follows that the projection of $\ca P(V)$ to $\ca C(X)$ is quasi-isometric to the subcomplex pictured in Figure \ref{fig:farey}, which is quasi-isometric to a line. 

\begin{rmk}\label{rmk:incompressible-distance}
The curve and arc complexes of $X\cong \Sigma_{1,1}$ are isometric via the map $\ca A(X)\ra\ca C(X)$ used to define the subsurface projection (\S\ref{sec:subsurface-projection}). Consequently the subsurface projection of $\ca P(V;X)$ to $\ca C(X)$ is an embedding, which is in fact a quasi-isometric embedding. In particular, the distance in $\ca P(V;X)$ between vertices $D,E\in \ca P(V;X)$ is (coarsely) bounded above by $d_X(D,E)$. 
\end{rmk}

\begin{rmk}\label{rmk:paired} 
In a similar way, we arrive at the same description for $\ca P(V;Y)$, where $Y=F\ti0$ is the other horizontal boundary component. 
Since every primitive disk can be surgered to a vertical disk, while only changing its subsurface projection by a uniformly bounded amount, it follows that the infinite-diameter holes $X$ and $Y$ are \emph{paired} in the sense of \cite[Defn.\ 5.6]{masur-schleimer}. This will be used in the discussion of the distance formula. 
\end{rmk}

\subsection{Goeritz symmetries preserving the figure-8 knot}\label{sec:fig8-symmetry}

The following result is needed for Theorem \ref{thm:pseudo}. 

\begin{prop}[figure-8 knot stabilizer]\label{prop:fig8-stabilizer}
Let $K$ be the embedding of the figure-$8$ knot on $S\subset S^3$ in Figure \ref{fig:fig8}. The subgroup of $\bb G$ that fixes $K$ is generated by $\al$, $\beta\delta\beta^{-1}\delta$, and $\ga\de$. 
\end{prop}

\begin{proof}
First we apply Proposition \ref{prop:fiber-preserving} and Remark \ref{rmk:relation}. If $g\in\bb G$ preserves $K$ and preserves each component of $S\setminus K$, then $g$ is represented by a product homeomorphism $g=\psi\ti\id$ on $F\ti I$ and $\what F\ti I$, where $\psi$ commutes with $\phi=T_aT_b^{-1}$. The centralizer of 
\[A:=\left(\begin{array}{cc}2&1\\1&1\end{array}\right)\leftrightarrow\phi\>\>\>\text{ in }\>\>\>\SL_2(\Z)\cong\Mod(F)\]
is generated by $A$ and $-I$. 

One can show that $\be\de\be^{-1}\de$ fixes $K$, preserves each of $F\ti 1$ and $F\ti 0$, and acts on the homology of both by the matrix $A$ (we found this element by trail and error). The hyperelliptic $\alpha$ (of course) also fixes $K$ and acts on homology by $-I$. Therefore the subgroup of $\bb G$ that both fixes $K$ and preserves each component of $S\setminus K$ is generated by $\alpha$ and $\beta\delta\beta^{-1}\delta$. This subgroup has index (at most) 2 in the stabilizer of $K$. Finally, observe that $\gamma\delta$ conjugates $\beta\delta\beta^{-1}\delta$ to its inverse. This element preserves $K$ and swaps the two components of $S\setminus K$ (c.f.\ Remark \ref{rmk:relation}). Therefore, the stabilizer of $K$ is generated by $\alpha,\beta\delta\beta^{-1}\delta$, and $\gamma\delta$. 
\end{proof}


\section{Distance formula for $\ca P(V)$}\label{sec:distance}

Our proof of the distance formula (Theorem \ref{thm:distance}) follows closely the approach of \cite{masur-schleimer}. There are two parts: an upper and a lower bound on $d_{\ca P(V)}(D,E)$. The lower bound follows directly from \cite[Thms.\ 6.10, 6.12]{masur-minsky2}, as explained in \cite[Thm.\ 5.14]{masur-schleimer}. To prove the upper bound requires additional work, and for this we use the axiomatic approach of \cite[\S13]{masur-schleimer}.

The general axiom-based proof of the upper bound \cite[Thm.\ 13.1]{masur-schleimer} is an inductive argument, which, in our case, necessitates the definition of relative versions of $\ca P(V)$. If $X\subset S$ is a hole for $\ca P(V)$, then we define $\ca P(V;X)\subset\ca P(V)$ as the subcomplex spanned by primitive disks supported in $X$. Note that we do not necessarily assume that $X$ has infinite diameter. If $X$ is primitively incompressible, then $\ca P(V;X)$ is empty, which is undesirable, so when $X$ also has infinite diameter, we define $\ca P(V;X)$ as the complex spanned by vertical primitive disks in the $I$-bundle constructed in the proof of Theorem \ref{thm:holes}. By the nature of the induction, we will not need to deal with the case $X$ is a primitively incompressible hole with bounded diameter, so we keep $\ca P(V;X)$ empty in this case. 

For the induction, we need to prove the following more general version of the distance formula upper bound. 

\begin{thm}\label{thm:distance-induction}
Let $X\subset S$ be a hole for $\ca P(V)$. Given $c\ge0$, there is a constant $A$ depending only on $c$ and the topological type of $X$, such that for any two vertices $D,E\in\ca P(V;X)$, we have
\[d_{\ca P(V;X)}(D,E)\le A\sum \{d_Y(D,E)\}_c+A,\]
where the sum is over all holes $Y\subset X$ for $\ca P(V)$. 
\end{thm}

Given our previous work, the most interesting case of Theorem \ref{thm:distance-induction} is when $X=S$, as we now explain. If $X\subsetneq S$ is an infinite-diameter hole, then the distance formula holds by our classification of such $X$ in Theorem \ref{thm:holes} and by the computation of $\ca P(V;X)$ given in \S\ref{sec:fig8}; see Remark \ref{rmk:incompressible-distance}. In the remaining case, when $X$ is a primitively compressible hole, Theorem \ref{thm:distance-induction} holds trivially by the following proposition. 

\begin{prop}\label{prop:compressible-holes}
Let $X\subsetneq S$ be a primitively compressible hole for $\ca P(V)$. Then $\ca P(V;X)$ is a connected, graph with at most two vertices. 
\end{prop}

\begin{proof}
As noted in Remark \ref{rmk:compressible}, if a hole $X\subset S$ supports two primitive disks with nonzero intersection number, then $X=S$. Therefore, our assumption $X\neq S$ implies that $\ca P(V;X)$ is either a single vertex or is two vertices connected by an edge. 
\end{proof}

It remains to prove Theorem \ref{thm:distance-induction} for $X=S$. It is here that we use the axiomatic approach of \cite{masur-schleimer}. 

\subsection{Axiomatic approach to the distance formula for $\ca P(V)$} \label{sec:upperbound}

This section explains a proof of Theorem \ref{thm:distance-induction} in the main case of interest $X=S$. 

Our approach is to use \cite[Thm.\ 13.1]{masur-schleimer}, which states that there are seven axioms that together imply the desired upper bound. Instead of writing the full statement of the axioms with all the necessary definitions (this takes multiple pages in \cite{masur-schleimer}), we give precise references to \cite{masur-schleimer} and follow their notation as closely as possible. Our goal is to explain why the axioms hold for $\ca P(V)$, focusing on the parts that are the least straight-forward. Ultimately, the argument is very similar to the proof in \cite[\S18-19]{masur-schleimer} that the axioms are satisfied for the disk complex $\ca D(V)$, using our analogous results. In fact, there are even some simplifications because we only work in genus 2 and because the classification of large holes (Theorem \ref{thm:holes}) is simpler than the corresponding result for $\ca D(V)$, c.f.\ \cite[Thm.\ 1.1]{masur-schleimer}. 

Below we list the axioms and explain why they hold for $\ca P(V)$. See \cite[\S13]{masur-schleimer} for statements of the axioms, and see \cite[\S18-19]{masur-schleimer} for additional information about why the axioms hold for $\ca D(V)$. 

\paragraph{Large Holes Axiom.} This is stated in \cite[Axiom 13.3]{masur-schleimer}. Here we may take ``large" to mean diameter $\ge 61$, which implies the hole has infinite diameter (Theorem \ref{thm:holes}). This axiom states that any two large holes either overlap \cite[Defn.\ 2.8]{masur-schleimer} or are paired \cite[Defn.\ 5.6]{masur-schleimer}. This holds trivially for $\ca P(V)$ because all proper holes are homeomorphic to $\Sigma_{1,1}$, so the only way two large holes can be disjoint is if they are paired, cf.\ Remark \ref{rmk:paired}.

\paragraph{Marking Path, Accessibility Intervals, Combinatorial Sequence Axioms.} 
These three axioms concern properties of the following data associated to a pair of primitive disks $D,E\subset V$. These axioms hold for the exact same reason as they hold for $\ca D(V)$. Start with a surgery path $D=D_0,\ldots,D_N=E$ as defined in \S\ref{sec:surgery}. 

\begin{itemize}
\item The \emph{marking path} is defined as follows. The sequence $\{D_n\}$ induces a nested sequence of train tracks $\tau_0\succ \cd\succ\tau_N$ by the construction explained in \cite[\S4]{masur-minsky-qc}. A train track $\tau$ has an associated a vertex cycle $\vertex(\tau)\subset\ca C(S)$ (see e.g.\ \cite[\S3]{masur-minsky-qc}), which is a marking in the sense of \cite[\S2.9]{masur-schleimer}. The marking path is defined as $\mu_n=\vertex(\tau_n)$. 
\item Define the \emph{combinatorial sequence} as follows. By \cite[Thm.\ 1.3]{masur-minsky-qc}, the subset $\bigcup_{n=0}^N\vertex(\tau_n)\subset\ca C(S)$ is an unparameterized quasi-geodesic, which means that there is a subsequence $\tau_0\succ\cdots\succ\tau_K$ such that for any choice of $v_k\in\vertex(\tau_k)$, the sequence $v_0,\ldots,v_K$ is a quasi-geodesic in $\ca C(S)$. The combinatorial sequence, is the associated sequence $D=D_0,\ldots,D_K=E$ of primitive disks. 
\item By construction of the combinatorial sequence, there is an increasing \emph{re-indexing function} $r:[0,K]\rightarrow[0,N]$. 
\item Given the sequence $\{\tau_n\}_{n=0}^N$, for each subsurface $X\subset S$, an \emph{accessibility interval} $J_X\subset[0,N]$ is defined in \cite[\S5.2]{MMS}. 
\end{itemize}

The Marking Path Axiom \cite[Axiom 13.4]{masur-schleimer} states that the $\mu_n$ have nested supports and that the projection $\pi_X(\mu_n)$ to any subsurface $X\sbs S$ is an unparameterized quasi-geodesic in $\ca C(X)$ with constant depending only on $\ca P(V)$. The former property holds by construction, and the latter property holds as a consequence of a general result proved by Masur--Mosher--Schleimer \cite[Thm.\ 5.5]{MMS}.

The Accessibility Intervals Axiom (see \cite[Axiom 13.5]{masur-schleimer}) is verified in \cite[Thm.\ 5.3]{MMS}, just as for $\ca D(V)$; see \cite[\S18]{masur-schleimer}. 

The Combinatorial Sequence Axiom (see \cite[Axiom 13.6]{masur-schleimer}) follows from the construction of $\{D_k\}_{k=0}^K$ and $\{\mu_n\}_{n=0}^N$, just as for $\ca D(V)$. Details are in \cite[\S3-4]{masur-minsky-qc}; see also \cite[Thm.\ 19.3]{masur-schleimer}. 

%

\paragraph{Replacement Axiom.} This is stated in \cite[Axiom 13.7]{masur-schleimer}. The proof of this axiom for $\ca P(V)$ follows quickly from the following claim (cf.\ \cite[\S19.5]{masur-schleimer}). 

{\it Claim.} Let $Y\subsetneq S$ be a large incompressible hole for $\ca P(V)$ with corresponding $I$-bundle $T$, and fix a primitive disk $D$ such that $i(\pa D,\pa Y)<K$ for some constant $K$, then there exists a primitive disk $D'$ that is vertical in $T$ and such that $d_{\ca P(V)}(D,D')\le\log_2(K)$. 

{\it Proof of Claim.} The proof is similar to \cite[\S19.5]{masur-schleimer}. Consider a maximal sequence of primitive compressions of $D$ into $S\setminus\pa Y$. Each compression replaces $D$ with two disks, and we choose the one with at most half as many intersections with $\pa Y$. This process ends after $\le \log_2(K)$ steps with a disk $D'$ with no primitive compression into $S\setminus\pa Y$, and by the argument in \S\ref{sec:fig8}, $D'$ can be isotoped to be vertical. This proves the claim. 

\paragraph{Straight Intervals Axiom.} See \cite[Axiom 13.13]{masur-schleimer}. This axiom holds for for $\ca P(V)$ by the exact same argument as \cite[\S19.7]{masur-schleimer}.

\paragraph{Shortcut Intervals Axiom.} See \cite[Axiom 13.14]{masur-schleimer}. The proof of this axiom ultimately boils down to the following claim (cf.\ \cite[\S19.8]{masur-schleimer}). 

{\it Claim.} Let $Z$ be a non-hole for $\ca P(V)$, and fix disks $D,E$ such that $i(\pa D,\pa Z)<K$ and $i(\pa E,\pa Z)<K$. Then $d_{\ca P(V)}(D,E)<K'$ for some constant $K'$ depending only on $K$. 

{\it Proof of Claim.} Let $Y\sbs S$ be the (closure of the) complement of $Z$. Note that $Y$ is primitively compressible because $Z$ is a non-hole for $\ca P(V)$ and $Y\cup Z=S$. As in the argument for the Replacement Axiom, we choose maximal compression sequences $D\leadsto D'$ and $E\leadsto E'$ of $D$ and $E$ into $S\setminus\pa Z$ with each sequence of length $\le\log_2(K)$. It suffices to give a uniform bound on $d_{\ca P(V)}(D',E')$. If $D'$ and $E'$ are disjoint, then we are done, so we assume $D',E'$ intersect nontrivially. Then $D',E'$ lie in the same component of $S\setminus\pa Z$. If this component is $Z$, then $d_{\ca P(V)}(D',E')\le 2$ since $Y$ is primitively compressible. Next suppose $D',E'$ are contained in a component $Y'$ of $Y$. Observe that $Y'$ is not a hole for $\ca P(V)$ since if it were, this would imply $Y=S$ by Remark \ref{rmk:compressible}. Then $Y$ is not a hole, so again we conclude that $d_{\ca P(V)}(D',E')\le 2$.

\subsection{Consequences of the distance formula} 

In this section we explain several consequences of the distance formula that we use in Section \ref{sec:ccc} to prove Theorem \ref{thm:ccc}. The results in this section appear in a similar form in \cite{BBKL}, although we need to formulate them differently, replacing the assumption that $G$ is quasi-isometrically embedded in $\Mod(S)$ with the assumption that $G<\bb G$ and the orbit map $G\ra\ca P(V)$ is quasi-isometrically embedded. In fact, the results we prove work more generally: 

In this section (and this section only), we let $S$ denote any closed oriented surface of genus $\ge2$, and we fix a subgroup $\bb G<\Mod(S)$ that acts on a subcomplex $\ca P\sbs\ca C(S)$. We assume that $\ca P$ has a distance formula. Similar to \cite[\S2]{BBKL}, we express the distance formula in the following form, which is equivalent to (\ref{eqn:distance}): given $\be>0$ (sufficiently large), there is $\la>0$ so that 
\begin{equation}\label{eqn:distance-general}
\fr{1}{\la} d_{\ca P}(u,u')\le \sum_{\substack{X\sbs S\\\text{ hole for $\ca P$}}}\{d_X(u,u')\}_\be\le \la\>d_{\ca P}(u,u')
\end{equation}
for any vertices $u,u'$ of $\ca P$ satisfying either $\sum_X\{d_X(u,u')\}_\be\neq0$ or $d_{\ca P}(u,u')\ge \la$. 

We also assume that there is complex $\ca M$ whose vertices are markings (i.e.\ vertices of $\ca M(S)$; see \S\ref{sec:complexes}) such that (1) $\ca M$ is quasi-isometric to $\bb G$, and (2) there is a uniform bound on the distance in $\ca M(S)$ between adjacent markings in $\ca M$. 

All of these assumptions hold for $\bb G$ the genus-2 Goeritz group, $\ca P=\ca P(V)$ the primitive disk complex, and $\ca M=\ca M(V,\what V)$ the Heegaard marking complex. See Theorem \ref{thm:distance} and Lemmas \ref{lem:marking-comparison} and \ref{lem:marking-qi-G}. 

The main result needed for the proof of Theorem \ref{thm:ccc} is Proposition \ref{prop:contains-reducible}. This result and its proof are similar to the main theorem of \cite{BBKL}. 

\begin{prop}\label{prop:contains-reducible}
Fix a finitely-generated subgroup $G<\bb G$. If the orbit map $G\ra\ca P$ is a q.i.\ embedding, but the orbit map $G\ra\ca C(S)$ is not a q.i.\ embedding, then $G$ contains a reducible element. 
\end{prop}



To prove Proposition \ref{prop:contains-reducible}, we use two additional results. The first, Proposition \ref{prop:reducibility}, is proved in \cite[Prop.\ 3.1]{BBKL}. The second, Proposition \ref{prop:linear-factors}, is similar to  \cite[Prop.\ 4.1]{BBKL} but with a slightly different assumption. We sketch the proof below. 


For a subsurface $Z\sbs S$, we write $\ca M(Z)$ for the marking complex; see \cite[\S2]{BBKL}. Here ``marking" refers to the clean, complete markings defined in \cite[\S2.5]{masur-minsky2}; this differs from the more flexible notion of markings \cite[\S2.9]{masur-schleimer} that is used in \S\ref{sec:upperbound}.

\begin{prop}[Reducibility criterion]\label{prop:reducibility} 
Let $G<\Mod(S)$ be finitely generated, and let $|g|$ denote the word length of $g\in G$ with respect to a finite generating set. Fix a marking $\mu$. For any $c>0$ there exists $R=R(c)>0$ so that if $|g|>R$ and there exists a proper subsurface $Z\subset S$ with $d_{\ca M(Z)}(\mu,g\mu)\ge c|g|$, then $G$ contains a reducible element. 
\end{prop}

\begin{prop}[Linearly summing projections]\label{prop:linear-factors}
Fix $G<\bb G$ with finite generating set $\Omega$, and fix a marking $\mu\in\ca M$. Assume that the orbit map $G\ra\ca P$ is a quasi-isometric embedding. Then there exists $ K,C>0$ with the following property. For $g\in G$, if $|g|>C$ and $d_S(\mu,g\mu)<\frac{|g|}{K}$, then we can find subsurfaces $Z_1,\ldots,Z_k\sbs S$ and write $g=g_1\cdots g_k$ geodesically in $G$ with $k\le d_S(\mu,g\mu)$ so that 
\[|g|\le K\sum_{j=1}^kd_{\ca M(Z_j)}(\mu,g_j\mu).\]
Furthermore, by increasing $K$ we can also arrange that $d_{\ca M(Z_j)}(\mu,g_j\mu)\le K|g_j|$ for each $j$. 
\end{prop}

\begin{proof}[Proof of Proposition \ref{prop:contains-reducible}]
Fix a marking $\mu\in\ca M$. Let $K,C$ be the constants from Proposition \ref{prop:linear-factors}, and fix $M>K$ (to be chosen more precisely later). Since $G\ra\ca C(S)$ is not a q.i.\ embedding, we can find group elements $g\in G$ with $|g|$ arbitrarily large and $d_S(\mu,g\mu)\le\fr{|g|}{M}$. Then we can apply Proposition \ref{prop:linear-factors}: write $g=g_1\cdots g_k$ with $k\le d_S(\mu,g\mu)$ and take $Z_1,\ldots,Z_k\sbs S$ so that $|g|\le K\sum_{j=1}^kd_{\ca M(Z_j)}(\mu,g_j\mu)$. 

For $t>0$, let $R(\fr{1}{t})$ be the constant from Proposition \ref{prop:reducibility}, and consider the partition of $\{1,\ldots,k\}$ into 
\[J_{\le}(t)=\{j: d_{\ca M(Z_j)}(\mu,g_j\mu)\le\fr{|g_j|}{t}\}\>\>\>\text{ and }\>\>\>J_{>}(t)=\{j: d_{\ca M(Z_j)}(\mu,g_j\mu)>\fr{|g_j|}{t}\}.\]
If $|g_j|\le R(\fr{1}{t})$ for each $j\in J_>(t)$, then using the conclusions of Proposition \ref{prop:linear-factors} and the assumption $d_S(\mu,g\mu)\le\fr{|g|}{M}$, we obtain the following estimate (cf.\ \cite[\S5]{BBKL}): 
\[\begin{array}{rcl}
|g|&\le&K\sum_{j=1}^kd_{\ca M(Z_j)}(\mu,g_j\mu)\\[2mm]
&=&K\left[\sum_{j\in J_{\le}(t)}d_{\ca M(Z_j)}(\mu,g_j\mu)+\sum_{j\in J_>(t)}d_{\ca M(Z_j)}(\mu,g_j\mu)\right]\\[2mm]
&\le& K\left[\fr{|g|}{t}+\sum_{j\in J_>(t)} K|g_j|\right]\\[2mm]
&\le& K\left[\fr{|g|}{t}+k\cdot K\cdot R(\fr{1}{t})\right]\\[2mm]
&\le& K\left[\fr{|g|}{t}+d_S(\mu,g\mu)\cdot K\cdot R(\fr{1}{t})\right]\\[2mm]
&\le& K\left[\fr{|g|}{t}+\fr{|g|}{M}\cdot K\cdot R(\fr{1}{t})\right]=\left(\fr{K}{t}+\fr{K\cdot R(\fr{1}{t})}{M}\right)\cdot |g|\\[2mm]
\end{array}\]
Now choose $t>2K$ and any $M>2K\cdot R(\fr{1}{t})$, so that $\left(\fr{K}{t}+\fr{K\cdot R(\fr{1}{t})}{M}\right)<1$. By the computation above, there must be some $j\in J_>(t)$ with $|g_j|>R(\fr{1}{t})$. Then we may apply Proposition \ref{prop:reducibility} to conclude that $G$ contains a reducible element. 
\end{proof}

It remains to explain the proof of Proposition \ref{prop:linear-factors}. Before this, we recall the distance formula for the marking complex due to \cite{masur-minsky2}; see also \cite[Thm.\ 2.7]{BBKL} for the form that we state it. 
Let $Z\sbs S$ be a subsurface. Given $\be>0$ (sufficiently large), there is $\kappa>0$ so that 
\begin{equation}\label{eqn:marking-distance}
\fr{1}{\ka} d_{\ca M(Z)}(\mu,\mu')\le \sum_{Y\sbs Z}\{d_Y(\mu,\mu')\}_\be\le \ka\>d_{\ca M(Z)}(\mu,\mu')
\end{equation}
for every $\mu,\mu'\in\ca M(S)$ so that either the middle term is nonzero or $d_{\ca M(Z)}(\mu,\mu')\ge\ka$.


\begin{proof}[Proof of Proposition \ref{prop:linear-factors}]
We explain the last statement first (this is one difference between our argument and the argument in \cite[\S4]{BBKL}). Since the orbit map $G\ra\ca P$ is a quasi-isometric embedding, so too is the orbit map $G\ra\ca M$ (see Lemma \ref{lem:contraction}). Then by Lemma \ref{lem:marking-comparison}, $d_{\ca M(S)}(\mu,g_j\mu)$ is coarsely bounded above by $|g_j|$. Combining this with the fact that subsurface projection of markings is coarsely Lipschitz (see \cite[Prop.\ 2.3]{BBKL}), one deduces the last sentence of the proposition. 

The rest of the proof is very similar to \cite[\S4]{BBKL}. The only difference is that we substitute, in the appropriate place, the distance formula for $\ca P$ (\ref{eqn:distance-general}) for the Masur--Minsky distance formula (\ref{eqn:marking-distance}). (Note, however, that we still use the Masur--Minsky distance formula in a different part of the argument.) Since precise details are contained in \cite[\S4]{BBKL}, we will only sketch the argument. In particular, we will be imprecise with some of the constants (e.g.\ we will not carefully specify $C$). 

Fix $\be\gg0$ sufficiently large (precisely how large will be evident at the end of the argument). Since $G\ra\ca P$ is a q.i.\ embedding, by the distance formula (\ref{eqn:distance-general}), there is a constant $K_1$ so that if $|g|\gg0$, then 
\[|g|\le K_1\sum_{\substack{X\sbs S\\\text{ hole for $\ca P$}}}\big\{d_X(\mu,g\mu)\big\}_\be\]

\paragraph{BBF factors and subsurface order.} By \cite{BBF}, it is possible to decompose the set of essential subsurfaces of $S$ into finitely many equivalence classes (called BBF factors), one of which is $\{S\}$, so that surfaces in the same BBF factor overlap (i.e.\ are neither disjoint nor nested). Combining this fact with the preceding inequality, there is a constant $K_2\ge K_1$ and a BBF factor $\bb Y$ so that 
\begin{equation}\label{eqn:BBF}
|g|\le K_2\sum_{\substack{Y\in\bb Y\\ \text{hole}}}\{d_Y(\mu,g\mu)\}_\be.\end{equation}

Ultimately, we will choose our constant $K$ to be larger than $K_2$, for then the assumption $d_S(\mu,g\mu)<\fr{|g|}{K}$ implies that $\bb Y\neq\{S\}$ in (\ref{eqn:BBF}). Assuming this, let $Y_1,\ldots,Y_n\in\bb Y$ be the holes with $d_Y(\mu,g\mu)\ge\be$. If $\be$ is large enough, it's possible to put the $Y_i$ in increasing order with respect to the subsurface order, which is defined using the Behrstock inequality, c.f.\ \cite[Prop.\ 2.6]{BBKL}. 

\paragraph{Maximal subsurfaces.} 
Define $0=i_0<i_1<\cdots<i_k=n$ inductively so that $i_j$ is the largest index so that 
\[Y_{i_{j-1}+1}\cup\cdots \cup Y_{i_j}\]
is a proper subsurface. Denote this subsurface by $Z_j'$. If $\be$ is sufficiently large, then the number $k$ of these maximal subsurfaces is bounded above by $d_S(\mu,g\mu)$ by \cite[Lem.\ 4.4]{BBKL}. 

We use this to decompose the sum in (\ref{eqn:BBF})
\[\sum_{i=1}^n d_{Y_i}(\mu,g\mu)=\sum_{j=1}^k\>\sum_{\ell=i_{j-1}+1}^{i_j}d_{Y_\ell}(\mu,g\mu)\]

\paragraph{Prefixes of $g$.} Next we bound the preceding sum. Write $g=s_1\cdots s_m$ in the generating set $\Omega$. Let $b>0$ be an upper bound on $d_{\ca M(Z)}(\mu,s\mu)$ and $d_{Z}(\mu,s\mu)$ for $s\in\Omega$ and any subsurface $Z\subset S$. For $1\le j<k$ define $g_j'$ as the largest prefix of $g=s_1\cdots s_m$ so that $d_{Y_{i_j}}(g_j'\mu,g\mu)\ge 2b$. Set also $g_{0}=\id$ and $g_{k}=g$. By \cite[Lem.\ 4.3]{BBKL}, $g_{j-1}'$ is a prefix of $g_{j}'$, and 
\begin{equation}\label{eqn:BBKL}d_{Y_\ell}(\mu,g\mu)-5b\le d_{Y_\ell}(g_{j-1}'\mu,g_j'\mu)\end{equation}
for $i_{j-1}<\ell\le i_j$. (Note: here our notation differs from \cite{BBKL}; it would agree if we wrote $g_{i_j}'$ instead of $g_j'$.) Since $d_{Y_\ell}(\mu,g\mu)\ge\be$ by construction, if we choose $\be$ sufficiently large, we can assume that $d_{Y_\ell}(\mu,g\mu)-5b\ge\fr{1}{2}d_{Y_\ell}(\mu,g\mu)$. Combining this with (\ref{eqn:BBKL}) gives 
\[d_{Y_\ell}(\mu,g\mu)\le 2\>d_{Y_\ell}(g_{j-1}'\mu,g_j'\mu).\]
Compare with \cite[Eqn.\ (5)]{BBKL}. Denoting $g_j=(g_{j-1}')^{-1}g_j'$ and $Z_j=(g_{j-1}')^{-1}(Z_j')$, we have 
\[\begin{array}{rcl}
\sum_{\ell={i_{j-1}+1}}^{i_j}d_{Y_\ell}(\mu,g\mu)&\le& 2\sum_{\ell}d_{Y_\ell}(g_{j-1}'\mu,g_j'\mu)\\[2mm]
&=&2\sum_{\ell} d_{(g_{j-1}')^{-1}(Y_\ell)}(\mu,g_j\mu)\\[2mm]
&\le&2\sum_{Y\sbs Z_j}\{d_Y(\mu,g_j\mu)\}_\be\\[2mm]
&\le&2\ka\>d_{\ca M(Z_j)}(\mu,g_j\mu)
\end{array}\]
The last inequality uses the distance formula for the marking complex (\ref{eqn:marking-distance}). (We know that the distance formula applies because the sum on the left-hand side above is nonzero.) 

\paragraph{Conclusion.} Now summing over $j$, we obtain 
\[|g|\le K_2\sum_{i=1}^nd_{Y_i}(\mu,g\mu)\le 2\kappa K_2\sum_{j=1}^kd_{\ca M(Z_j)}(\mu,g_j\mu)\]
which is the desired inequality. Therefore, if $K=2\kappa K_2$ and $|g|$ is sufficiently large with $d_S(\mu,g\mu)\le\fr{|g|}{K}$, then we obtain $Z_1,\ldots,Z_k\sbs S$ and $g=g_1\cdots g_k$ with the desired properties.
\end{proof}

\section{Purely pseudo-Anosov implies convex cocompact} \label{sec:ccc}

This section contains the proof of Theorem \ref{thm:ccc}. 

Fix $G<\bb G$ finitely-generated, purely pseudo-Anosov. To show that $G$ is convex cocompact in $\Mod(S)$, it suffices to show that the orbit map $G\ra\ca C(S)$ is a q.i.\ embedding; furthermore, by  Proposition \ref{prop:contains-reducible}, it suffices to show the orbit map $G\ra\ca P(V)$ is a q.i.\ embedding. We proceed in the following steps. 

\begin{itemize}
\item \un{Step 1.} We show that $\ca P(V)$ is quasi-isometric to a certain coned-off Cayley graph $\Cone(\Ga_{\bb G},\bb G_E)$. 
\item \un{Step 2.} We show that $G\ra\Cone(\Ga_{\bb G},\bb G_E)$ is a q.i.\ embedding, using in particular a result of \cite{abbott-manning} and the fact that $\bb G$ is virtually free. 
\end{itemize}

\subsection{Step 1: quasi-isometry type of $\ca P(V)$} 

Let $\Ga_{\bb G}$ be the Cayley graph of $\bb G$ with respect to the standard generating set (Figure \ref{fig:generators}). As usual, we view $\Ga_{\bb G}$ as a metric space by giving each edge length 1. Given a subgroup $H<\bb G$, we denote $\Cone(\Ga_{\bb G},H)$ be the space obtained from $\Ga_{\bb G}$ by coning off each translate of $H\sbs\Ga_{\bb G}$. Specifically, $\Cone(\Ga_{\bb G},H)$ is the graph obtained from $\Ga_{\bb G}$ by adding an additional vertex $*_{gH}$ for each coset $gH \in\bb G/H$ and edges from $*_{gH}$ to the vertices corresponding to $gH$. 

More generally, given a $\bb G$-space $X$ with basepoint $x_0$ and a subgroup $H<\bb G$, we denote $\Cone(X,H)$ the space obtained by coning off $gH(x_0)$ for each $gH\in\bb G/H$. 

\begin{prop}[$\ca P(V)$ as a coned-off Cayley graph]\label{prop:cone}
Fix a primitive disk $E\subset V$, and let $\bb G_E$ be its stabilizer in $\bb G$. Then the complexes $\Cone(\Ga_{\bb G},\bb G_E)$ and $\ca P(V)$ are quasi-isometric. 
\end{prop}

\begin{proof}
Recall from \S\ref{sec:complexes} that $\ca R(V,\what V)$ denotes the reducing sphere complex. Choose as basepoint the reducing sphere $P$ pictured in Figure \ref{fig:reducing-sphere}, and let $E=E_2$ be the primitive disk pictured in Figure \ref{fig:handlebody}. 

\paragraph{Claim 1.} The complex $\Cone(\ca R(V,\what V),\bb G_E)$ is isomorphic to a simplicial subdivision of $\ca P(V)$. 

{\it Proof of Claim $1$.} This follows from the work of \cite{cho} that was discussed in \S\ref{sec:complexes}: By Remark \ref{rmk:cone}, $\ca P(V)$ is obtained from $\ca R(V,\what V)$ by coning off, for each primitive disk $D$, the subcomplex $\ca R_D\sbs\ca R(V,\what V)$ spanned by reducing spheres that are disjoint from $D$. By the proof of Proposition \ref{prop:primitive-stabilizer}, $\bb G_E$ acts transitively on the vertices of $\ca R_E$, so coning $\ca R_E$ is the same as coning off the orbit of $P$ under $\bb G_E$. Since $\bb G$ acts transitively primitive disks, we conclude that $\ca P(V)$ is obtained from $\ca R(V,\what V)$ by coning the $g\bb G_E(P)$ for each coset $g\bb G_E$. This proves Claim 1. 

Replacing each triangle of $\ca R(V,\what V)$ by the cone on its vertices defines a graph $T$ with a quasi-isometry $T\hra\ca R(V,\what V)$, and $T$ is a tree by \cite[Thm. 1]{akbas} and \cite[\S6]{cho}. Furthermore, $T$ is the Bass--Serre tree for the splitting of $\bb G$ in (\ref{eqn:splitting}) (indeed Akbas and Cho obtain the splitting for $\bb G$ from $T$). 

\paragraph{Claim 2.} The tree $T$ is quasi-isometric to $\Cone(\Ga_{\bb G},\bb G_P\cap\Ga_E)$. 

{\it Proof of Claim $2$.} As mentioned above, $T$ is a Bass--Serre tree for the splitting $G=A*_CB$ in (\ref{eqn:splitting}). In particular, $T$ can be obtained from $\Ga_{\bb G}$ by collapsing cosets of the vertex groups to points. Up to quasi-isometry, this is the same as coning off cosets of the vertex groups. From this one deduces that $T$ is quasi-isometric to $\Cone(\Ga_{\bb G},\bb G_P)$ (since $\bb G_P$ is one of the vertex groups, and the other vertex group is finite). 

Recall that $\bb G_P$ is generated by $\{\alpha,\beta,\gamma\}$ by \cite[\S2]{scharlemann}, and $\bb G_E$ is generated by $\{\alpha,\beta,\gamma\delta\}$ by Proposition \ref{prop:primitive-stabilizer}. It follows from the presentation (\ref{eqn:splitting}) for $\bb G$ that $\bb G_E\cap\bb G_P=\pair{\alpha,\beta}$. In particular, $\bb G_E\cap\bb G_P$ has finite index in $\bb G_P$, so $\Cone(\Ga_{\bb G},\bb G_P)$ and $\Cone(\Ga_{\bb G},\bb G_E\cap \bb G_P)$ are quasi-isometric. This proves Claim 2.

Finally, observe that 
$\Cone(\Cone(\Ga_{\bb G},\bb G_E\cap\bb G_P),\bb G_E)$ and $\Cone(\Ga_{\bb G},\bb G_E)$ are quasi-isometric (this holds generally). Combining all of the above quasi-isometries, gives the desired conclusion: 
\[\begin{array}{rcl}
\ca P(V)&\sim& \Cone(\ca R,\bb G_E)\sim\Cone(T,\bb G_E)\\[2mm]
&\sim&\Cone(\Cone(\Ga_{\bb G},\bb G_E\cap\bb G_P),\bb G_E)\sim \Cone(\Ga_{\bb G},\bb G_E).\qedhere\end{array}\]
\end{proof}

\subsection{Step 2: the orbit map $H\ra\Cone(\Ga_{\bb G},\bb G_E)$} 

In this step we show (Proposition \ref{prop:manning-abbott}) that the orbit map $G\ra\Cone(\Ga_{\bb G},\bb G_E)$ is a q.i.\ embedding. Combined with Step 1, this proves Theorem \ref{thm:ccc}. 

To show $G\ra\Cone(\Ga_{\bb G},\bb G_E)$ is a q.i.\ embedding, we first show $G\hra\Ga_{\bb G}$ is a q.i.\ embedding.

\begin{lem}\label{lem:qi-emb}\
\begin{enumerate}[(i)]
\item Every finitely generated subgroup $H<\bb G$ is quasi-isometrically embedded. 
\item The Goeritz group $\bb G$ is virtually free.
\end{enumerate}
\end{lem}

\begin{proof}
First we explain why (ii) implies (i). This is a well-known consequence of Marshall Hall's theorem \cite[Thm.\ 5.1]{hall}, which states that if $K$ is a finitely generated subgroup of a free group $F$, then $K$ is a free factor of a finite-index subgroup, i.e.\ there exists a finite-index $F'<F$ so that $F'=K*L$ (for some subgroup $L$). By this result, $K$ is isometrically embedded in $F'$ (e.g.\ use the retract $F'\onto K$), and hence $K\hra F$ is a q.i.\ embedding. Clearly a similar argument works for virtually free groups. 

Now we prove (ii). A finitely generated group is virtually free if and only if it can be expressed as a the fundamental group of a finite graph of finite groups; see \cite[Ch.\ II.2.6]{serre-trees} and \cite[Thm.\ 7.3 ff.]{scott-wall}. We show that $\bb G$ has this structure. 

By the presentations of $\bb G$ given by \cite[Thm.\ 2]{akbas} and \cite[\S5]{cho}, $\bb G$ is an amalgamated product as in (\ref{eqn:splitting}). 
Any such group is the fundamental group of a graph of groups of the form pictured in Figure \ref{fig:graph-of-groups}. This proves the lemma.
\end{proof}

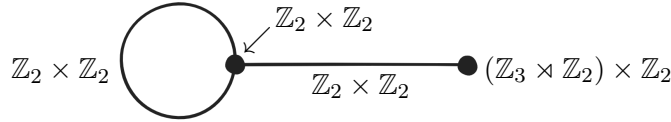
\begin{figure}[h!]
\labellist
\pinlabel $\Z_2\ti\Z_2$ at 530 850
\pinlabel $\Z_2\ti\Z_2$ at 820 835
\pinlabel $\Z_2\ti\Z_2$ at 785 900
\pinlabel $(\Z_3\rt\Z_2)\ti\Z_2$ at 1030 850
\pinlabel $\swarrow$ at 720 880
\endlabellist
\centering
\includegraphics[scale=.4]{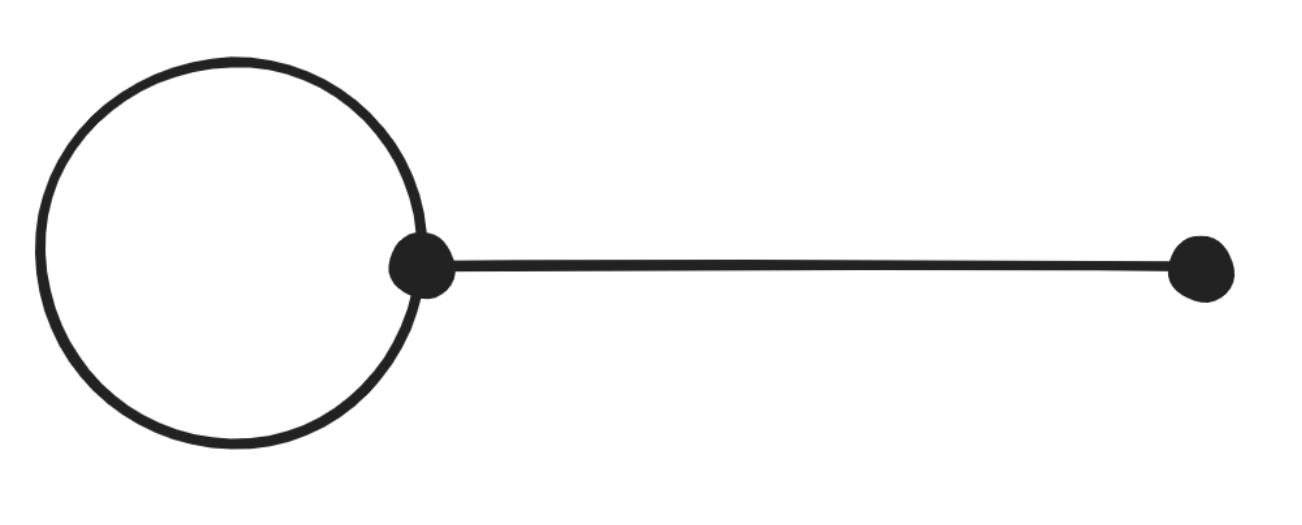}
\caption{A graph of groups description of $\bb G$.} 
\label{fig:graph-of-groups}
\end{figure}

We use Lemma \ref{lem:qi-emb} to prove Proposition \ref{prop:manning-abbott}, which implies that $G\ra\Cone(\Ga_{\bb G},\bb G_E)$ is a q.i.\ embedding (since $G$ is purely pseudo-Anosov).

\begin{prop}\label{prop:manning-abbott}
Fix a subgroup $H<\bb G$. Assume that $H$ acts freely on the set of primitive disks. Then the orbit map $H\ra\Cone(\Ga_{\bb G},\bb G_E)$ is a q.i.\ embedding. 
\end{prop}

\begin{proof}

By Lemma \ref{lem:qi-emb} any finitely generated $H<\bb G$ is quasi-isometrically embedded in $\bb G$, hence quasi-convex \cite[III.$\Ga$.3.6]{bridson-haefliger}. By \cite[Corollary 6.13]{abbott-manning}, to show that $H\ra\Cone(\Ga_{\bb G},\bb G_E)$ is a q.i.\ embedding it suffices to show that the limit set $\Lam(H)\sbs\pa\bb G$ is disjoint from the limit sets $\Lam(g\bb G_Eg^{-1})=\Lam(g\bb G_E)$ for each $g\in\bb G$. 

Since $\Lam(G_1)\cap\Lam(G_2)=\Lam(G_1\cap G_2)$ in $\pa\bb G$ for quasi-convex subgroups $G_1,G_2<\bb G$ of a hyperbolic group $\bb G$ \cite[Lem.\ 2.6]{GMRS}, we conclude that $\Lam(H)\cap\Lam(g\bb G_Eg^{-1})$ must be empty, since otherwise $H\cap g\bb G_Eg^{-1}$ would be nontrivial for some $g\in\bb G$, which contradicts the assumption that $H$ acts freely on the set of primitive disks. 
\end{proof}

\section{Characterization of Goeritz pseudo-Anosov elements}\label{sec:nielsen-thurston}

This section contains the proof of Theorem \ref{thm:pseudo}. For this, we fix $g\in\bb G$ and assume $g$ is not pseudo-Anosov. If $g$ has finite order, then $g$ is conjugate into one of the vertex groups in (\ref{eqn:splitting}) by the theory of graphs of groups \cite[Cor.\ 3.8]{scott-wall}. These vertex groups appear as (ii) and (iii) in the statement of Theorem \ref{thm:pseudo}. Therefore, the most interesting case of Theorem \ref{thm:pseudo} is when $g$ has infinite order. 

Assuming $g\in\bb G$ is reducible and infinite order, to prove Theorem \ref{thm:pseudo}, it suffices to show that $g$ (or a power) stabilizes either a finite set of primitive disks or an infinite-diameter genus-1 subsurface (these are classified in \S\ref{sec:holes}). We prove this with a case-by-case analysis of the canonical reduction system $\CRS(g)$. As such, the proof will also immediately prove Corollary \ref{cor:CRS}. 

Recall that a \emph{reduction system} of a reducible mapping class $g$ is a multicurve that is preserved by $g$. Reduction systems for $g$ are partially ordered by inclusion, and the \emph{canonical reduction
system} $\CRS(g)$ is defined as the intersection of all maximal reduction systems of $g$. See \cite[\S13.2]{farb-margalit} for more information and also \cite[Lem.\ 2.6]{BLM} for basic properties of $\CRS(g)$.

In general, there are the following possibilities for a multicurve, like $\CRS(g)$, on a genus-2 surface: 
\begin{enumerate}
\item[(A)] a separating curve
\item[(B)] a nonseparating curve
\item[(C)] two curves: one separating, one nonseparating
\item[(D)] two nonseparating curves
\item[(E)] three curves: one separating, two nonseparating
\item[(F)] three nonseparating curves
\end{enumerate} 

To analyze these possibilities we use some results of Oertel \cite{oertel} about the structure of homeomorphisms of handlebodies: 
\begin{itemize}
\item Fact 1. A multitwist of $S=\pa V$ extends to the handlebody $V$ if and only if there is a collection of essential disks and annuli whose boundary is the multicurve \cite[Thm.\ 1.11]{oertel} (an essential annulus is incompressible and not boundary parallel). 
\item Fact 2. If $X\sbs S$ is an essential subsurface that is preserved by a homeomorphism $g:V\ra V$, then the \emph{characteristic compression body} $Q_X\sbs V$ is also preserved by $g$. See \cite[Cor.\ 2.2]{oertel} for this statement and \cite[\S2]{oertel} for the definition of the characteristic compression body; the main example we will use is that if $X\cong\Sigma_{1,1}$ supports a disk, then $Q_X$ is a solid torus. 
\end{itemize}

\paragraph{Case (A).} Suppose that $\CRS(g)=\{c\}$, where $c$ is separating. We are done if $c$ is a reducing curve (since then we are in case (ii) of Theorem \ref{thm:pseudo}), so assume not. Let $X_1,X_2\cong \Sigma_{1,1}$ be the subsurfaces that $c$ bounds in $S$. We can write $g^2=h_1\circ h_2\circ T_c^n$, where $h_i$ is supported on $X_i$, and $\rest{h_i}{X_i}\in\Mod(X_i)$ is either the identity or pseudo-Anosov. Suppose that $X_1$ is primitively compressible in $V$. Then the characteristic compression body $Q_{X_1}$ is a solid torus. From this we conclude that $X_1$ supports a unique primitive disk $D$, and hence either $\{D\}$ or $\{D,gD\}$ is preserved by $g$ (depending on whether $g(X_1)=X_1$ or $g(X_1)=X_2$). We conclude similarly if $X_1$ (or $X_2$) is primitively compressible in $\what V$. Thus it remains to consider the case $X_1$ and $X_2$ are primitively incompressible in both $V$ and $\what V$. Since $c$ is not reducing, we can assume (without loss of generality) that $c$ does not bound a disk in $V$. Since $T_c^n\notin\bb G$ (by Fact 1), we can assume that $h_1$, say, is pseudo-Anosov. Then $X_1$ is an infinite-diameter hole for $\ca P(V)$ that is preserved by $g^2$. 

\paragraph{Case (B).} Suppose that $\CRS(g)=\{c\}$, where $c$ is nonseparating. Let $X=S\setminus n(c)\cong \Sigma_{1,2}$. We can write $g^2=h\circ T_c^n$, where $h$ is supported on $X$ and preserves each component of $\pa X$, and $\rest{h}{X}\in\Mod(X)$ is either the identity or a pseudo-Anosov. By Fact 1, $T_c^n\notin\bb G$, so $h\neq\id$. If $c$ does not bound a primitive disk, then $X$ is a hole for $\ca P(V)$, and $X$ has infinite diameter since $\rest{g}{X}=h$ is pseudo-Anosov. Since $X\cong \Sigma_{1,2}$, this contradicts Theorem \ref{thm:holes} (classification of holes). Therefore, $c$ bounds a primitive disk, which is fixed by $g$. 

\paragraph{Case (C).} Suppose that $\CRS(g)=\{c_1,c_2\}$, where $c_1$ is separating and $c_2$ is nonseparating. Let $X,Y\cong \Sigma_{1,1}$ be the subsurfaces that $c_1$ bounds, and assume $Y$ contains $c_2$.  We can write $g^2=h\circ T_{c_1}^{n_1}\circ T_{c_2}^{n_2}$, where $h$ preserves $X$, and $\rest{h}{X}\in\Mod(X)$ is either pseudo-Anosov or the identity. We consider these cases separately. 

First assume that $h$ is pseudo-Anosov. Then $X$ has infinite diameter with respect to $\ca P(V)$. Observe that $Y$ must support a primitive disk; otherwise $X$ is an infinite diameter hole for $\ca P(V)$, which implies (by Theorem \ref{thm:holes} and \S\ref{sec:fig8-symmetry}) that $g$ acts as a pseudo-Anosov on both $X$ and $Y$ contrary to our assumption about $\CRS(g)$. Let $D$ be a primitive disk supported in $Y$. The canonical compression body $Q_Y$ is a solid torus preserved by $g$, and $D$ is the unique disk (up to isotopy) in this solid torus, so $g(D)=D$.  

If $\rest{h}{X}\in\Mod(X)$ is the identity, then $g^2$ is a multitwist. By Fact 1, $c_1$ and $c_2$ either bound disks in $V$ or bound an essential annulus in $V$, and the same holds for $\what V$. In any case, (after properly orienting $c_1,c_2$) the homology class $c_1-c_2\in H_1(S)$ is a nonzero element in the kernel of the homomorphism $H_1(S)\ra H_1(V)\oplus H_1(\what V)$. By Lemma \ref{lem:mayer-vietoris} this implies that $H_2(S^3)\neq0$, a contradiction. 

\paragraph{Case (D).} Suppose that $\CRS(g)=\{c_1,c_2\}$, where both $c_1$ and $c_2$ are nonseparating. Since $g$ preserves $\CRS(g)$, we are done if at least one of $c_1$ and $c_2$ bounds a primitive disk in $V$ or $\what V$, so we assume that neither bounds a primitive disk. Denoting $X=S\setminus n(c_1\cup c_2)$, we can write $g^4=h\circ T_{c_1}^{n_1}\circ T_{c_2}^{n_2}$, where $h$ is supported on $X$ and preserves each component of $\pa X$, and $\rest{h}{X}\in\Mod(X)$ is either pseudo-Anosov or the identity. In fact, $\rest{h}{X}$ is not pseudo-Anosov because otherwise $X\cong \Sigma_{0,4}$ would be an infinite-diameter hole for $\ca P(V)$, contrary to Theorem \ref{thm:holes}. Then $g^4$ is a multitwist, and by applying Fact 1 in the same way as in Case (C), we conclude that $H_2(S^3)\neq0$, a contradiction. 

\paragraph{Cases (E) and (F).} Suppose that $\CRS(g)=\{c_1,c_2,c_3\}$, as in (E) or (F). Each complementary component of $\CRS(g)$ in $S$ is a pair of pants, which has finite mapping class group. Thus a power of $g$ is a nontrivial multitwist about the curves in $\CRS(g)$. Then, as in Cases (C) and (D), we apply Fact 1 arrive at a contradiction. Here there are two cases. In the first case $c_1$, say, bounds disks in both $V$ and $\what V$, and $c_2$ and $c_3$ bound annuli in each of $V$ and $\what V$; this implies that $c_1$ is separating, so both $c_2,c_3$ are nonseparating, and this implies that $c_2-c_3$ is a nontrivial element in the kernel of $H_1(S)\ra H_1(V)\oplus H_1(\what V)$. In the second case, without loss of generality, $c_1$ bounds a disk in $V$, $c_3$ bounds a disk in $\what V$, $c_2$ and $c_3$ bound an annulus in $V$, and $c_1$ and $c_2$ bound an annulus in $\what V$. At most one of the $c_i$ is separating, so $c_1-c_2+c_3$ is a nontrivial element of the kernel of $H_1(S)\ra H_1(V)\oplus H_1(\what V)$. 

This completes the proof of Theorem \ref{thm:pseudo}. \qed

\appendix 

\section{Classification of genus-1 fibered knots} 

Recall that a knot $K\subset S^3$ is \emph{genus-$1$ fibered} if $S^3\setminus K$ fibers as a once-punctured torus bundle over the circle. In this appendix we give a proof of the following classical result, first proved in \cite{fico}.  

\begin{thm}[Gonz\'alez-Acu\~na]\label{thm:gof}
The genus-$1$ fibered knots $K\subset S^3$ are the trefoil, its mirror, and the figure-$8$ knot. 
\end{thm}

The author has not been able to access \cite{fico} and does not know the proof given there. The proof below is almost surely different because (for example) we use Casson's invariant, which was defined in 1985, after the publication of \cite{fico}. One might consider the argument below to be a ``modern proof". 

First we reformulate the result. Let $F=\Sigma_{1,1}$, and denote $\mathring F=F\setminus\partial F$. By the discussion in \S\ref{sec:fibered-link}, given a genus-1 fibering
\[\mathring F\ra S^3\setminus K\ra S^1\] there is a Heegaard decomposition 
\[S^3=(F\ti I)\cup_\phi (F\ti I)\]
whose gluing, which matches horizontal boundaries of the two $I$-bundles, is encoded by a mapping class $\phi\in\Mod(F)$. See Figure \ref{fig:I-bundles}. Under a homeomorphism $M_\phi\cong S^3$, the circle $(\partial F)\times 1$ corresponds to a fibered knot. To prove Theorem \ref{thm:gof}, we consider all 3-manifolds of the form $M_\phi=(F\ti I)\cup_\phi(F\ti I)$, determine when $M_\phi\cong S^3$, and determine the knot $(\partial F)\times 1$ in $S^3$ 

For clarity, we denote the two $I$-bundles by $F\times I$ and $\what F\times I$. Fix curves $a,b$ giving a homology basis for $F=\Sigma_{1,1}$ as pictured in Figure \ref{fig:S11}. We write $T_a,T_b$ for the (right) Dehn twists about these curves. 

\begin{figure}
\labellist
\pinlabel $a$ at 70 605
\pinlabel $b$ at 165 610
\endlabellist
\centering
\includegraphics[scale=.6]{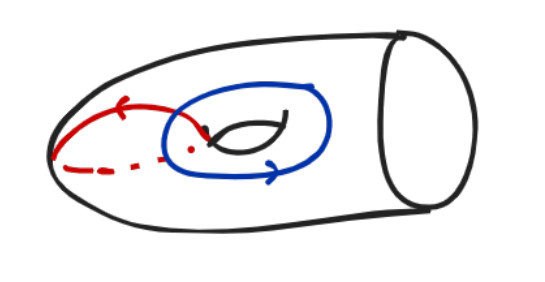}
\caption{Surface $F=\Sigma_{1,1}$ and a homology basis $a,b$. }
\label{fig:S11}
\end{figure}

We deduce Theorem \ref{thm:gof} from Theorem \ref{thm:I-bundle}. 

\begin{thm}[Building $S^3$ from $I$-bundles]\label{thm:I-bundle}
Fix $F=\Sigma_{1,1}$ and $\phi\in\Mod(F)$. Let $M_\phi=(F\ti I)\cup_\phi(\what F\ti I)$ be the closed 3-manifold described above. If $M_\phi\cong S^3$ then $\phi$ is conjugate in $\Mod( \Sigma_{1,1})$ to either $T_aT_b^{-1}$ or $(T_aT_b)^{\pm1}$. The respective fibered knots are the figure-8 and the trefoil and its mirror. 
\end{thm}

\begin{rmk}\label{rmk:inverse}The $I$-bundle structures on $M_{T_aT_b}$ and $M_{(T_aT_b)^{-1}}$ on $S^3$ differ by an orientation-reversing homeomorphism (Remark \ref{rmk:relation}). This explains why the corresponding fibered knots are mirrors. 
\end{rmk}

\begin{rmk}\label{rmk:capping} Recall (e.g.\ \cite[\S3.6]{farb-margalit}) that $\Mod(F)$ is a central extension
\[1\ra\Z\ra\Mod(F)\xra{p}\Mod(\mathring F)\ra1.\]
The kernel of $p$ is generated by the boundary Dehn twist $T_{\pa F}$. Furthermore, there is an isomorphism $\Mod(\mathring F)\cong\SL_2(\Z)$ given by the action on $H_1(\mathring F)\cong\Z^2$. In particular, working in the basis $a,b$,
\[p(T_aT_b^{-1})=\left(\begin{array}{cc}2&1\\1&1\end{array}\right)\>\>\>\text{ and }\>\>\>p(T_aT_b)=\left(\begin{array}{rr}0&1\\-1&1\end{array}\right).\]
The matrix $p(T_aT_b)$ has order 6, and in $\Mod( F)$ we have the relation $(T_aT_b)^6=T_{\pa F}$.
\end{rmk}

\begin{proof}[Proof of Theorem \ref{thm:I-bundle}]\

\paragraph{Step 1 (homology 3-spheres).} To reduce our search, we first show that if $H_1(M_\phi)=0$, then the action of $\phi$ on $H_1(F)$ is conjugate in $\SL_2(\Z)$ to one of 
\begin{equation}\label{eqn:matrices}\left(\begin{array}{cc}2&1\\1&1\end{array}\right)\>\>\>\text{ or }\>\>\>\left(\begin{array}{rr}0&1\\-1&1\end{array}\right)^{\pm1}.\end{equation}

Apply Mayer--Vietoris to the decomposition $M=(F\ti I)\cup(\what F\ti I)$. 
Denoting $\Sigma\cong\Sigma_2$ be the common boundary of $F\ti I$ and $\what F\ti I$, we have the following exact sequence
\[H_1(\Si)\xra{j} H_1(F\ti I)\oplus H_1(\what F\ti I)\xra{s} H_1(M_\phi)\xra{0} H_0(\Si).\]
It is not hard to show that the connecting homomorphism $H_1(M_\phi)\ra H_0(\Si)$ is zero. Then $H_1(M_\phi)=0$ if and only if $j$ is surjective, and since $j$ is a homomorphism $\Z^4\ra\Z^4$, $j$ is surjective if and only if $j$ is an isomorphism. 

Next we express $j$ in coordinates. Denote by $a_1,b_1\sbs \what F\ti 1\cong\what F=F$ the curves in Figure \ref{fig:S11}, which give a basis for $H_1(\what F\ti I)$. We use $\phi^{-1}(a_1),\phi^{-1}(b_1)\sbs F\ti 1$ as basis for $H_1(F\ti I)$. Let $a_0,b_0\sbs \what F\ti 0$ denote the parallel copies of $a_1,b_1$, and observe that $a_0,b_0,a_1,b_1\sbs\pa(\what F\ti I)=\Si$ is a basis for $H_1(\Si)$. With respect to these choices, $j$ has matrix
\[\left(\begin{array}{cc}
I&I\\
\tau&I
\end{array}\right),\]
where $I$ denotes the $2\times 2$ identity matrix, and $\tau$ is the matrix for the action of $\phi$ on $H_1(F)$. Then $j$ is an isomorphism if and only if 
\[\pm 1=\det\left(\begin{array}{cc}
I&I\\
\tau&I
\end{array}\right)=\det(I-\tau)=2-\tr(\tau).\] 
Consequently, either $\tr(\tau)=3$ or $\tr(\tau)=1$. It is well-known that this implies $\tau$ is conjugate to one of the matrices in (\ref{eqn:matrices}). We briefly explain this below. 

If $\tr(\tau)=3$, then $\tau$ is conjugate in $\SL_2(\Z)$ to $\left(\begin{array}{cc}2&1\\1&1\end{array}\right)$. This is because $\tau$ is a hyperbolic matrix with characteristic polynomial $x^2-3x+1$, and conjugacy classes of these matrices correspond to narrow ideal classes in $\Z[\la]$, where $\la$ is a root of the characteristic polynomial \cite{wallace}. Here $\Z[\la]=\Z\left[\fr{1+\sqrt{5}}{2}\right]$ is the ring of integers in $\Q(\sqrt{5})$. It is well-known that the narrow class number of this ring is $1$. 

If $\tr(\tau)=1$, then since $\tau$ satisfies its characteristic polynomial $x^2-x+1$, it follows that $\tau$ has order $6$. From the action of $\SL_2(\Z)$ on the dual tree to the Farey graph, it follows $\tau$ is conjugate to either $\left(\begin{array}{rr}0&1\\-1&1\end{array}\right)$ or its inverse.

\paragraph{Step 2} ($S^3$ recognition). As noted in Remark \ref{rmk:capping}, elements of $\Mod(F)$ that act on the same on $H_1(F)$ differ by the central Dehn twist $T_{\pa F}$. We want to show two statements: 
\begin{enumerate}
\item[(i)] If $\phi=T_aT_b^{-1}$ or $(T_aT_b)^{\pm1}$, then $M_\phi\cong S^3$. The fibered knots in these cases are the figure-$8$ and the trefoil, respectively. 
\item[(ii)] If $\phi$ is as in (i) and $\psi=\phi\circ T_{\pa F}^n$ where $n\neq0$, then $M_\psi\not\cong S^3$. 
\end{enumerate}

\boxed{(i)} Since $M_\phi$ and $M_\phi^{-1}$ are homeomorphic (Remark \ref{rmk:inverse}), it suffices to prove (i) for $\phi=T_aT_b$ and $\phi=T_aT_b^{-1}$.  
We can conclude quickly by showing that $M$ has a Heegaard diagram that is equivalent to the standard Heegaard diagram for $S^3$.

We first explain this for $\phi=T_aT_b$. We start with the Heegaard diagram, given by the collection of simple closed curves $(x_1,x_2;y_1,y_2)$ on $\Sigma$ pictured in Figure \ref{fig:heegaard-diagrams} (left). Now observe that the curves $(z,x_2;y_1,y_2)$ provide another Heegaard diagram for $M_\phi$, one that is equivalent to the standard Heegaard for $S^3$. To properly understand these claims and Figure \ref{fig:heegaard-diagrams}, see Figure \ref{fig:Sigma-vs-S}, which explains how we are translating between a curve the ``standard" genus-2 surface with curves on $\Sigma=\pa(F\ti I)$. (In particular, be careful not to confuse $\Sigma$ with $S\subset S^3$.) In Figure \ref{fig:heegaard-diagrams}, the curves $y_1,y_2$ are obtained by applying $\phi$ to $x_1,x_2$ on the genus-1 subsurface on the ``right" side of $\Sigma$. Then $y_1,y_2$ bound vertical disks in $\what F\times I$ by construction. The curve $z$ bounds a disk in $F\times I$ because it is disjoint from $x_1$ and $x_2$.

\begin{figure}[h!]
\labellist
\pinlabel $L$ at 300 610
\pinlabel $L$ at 85 545
\endlabellist
\centering
\includegraphics[scale=.5]{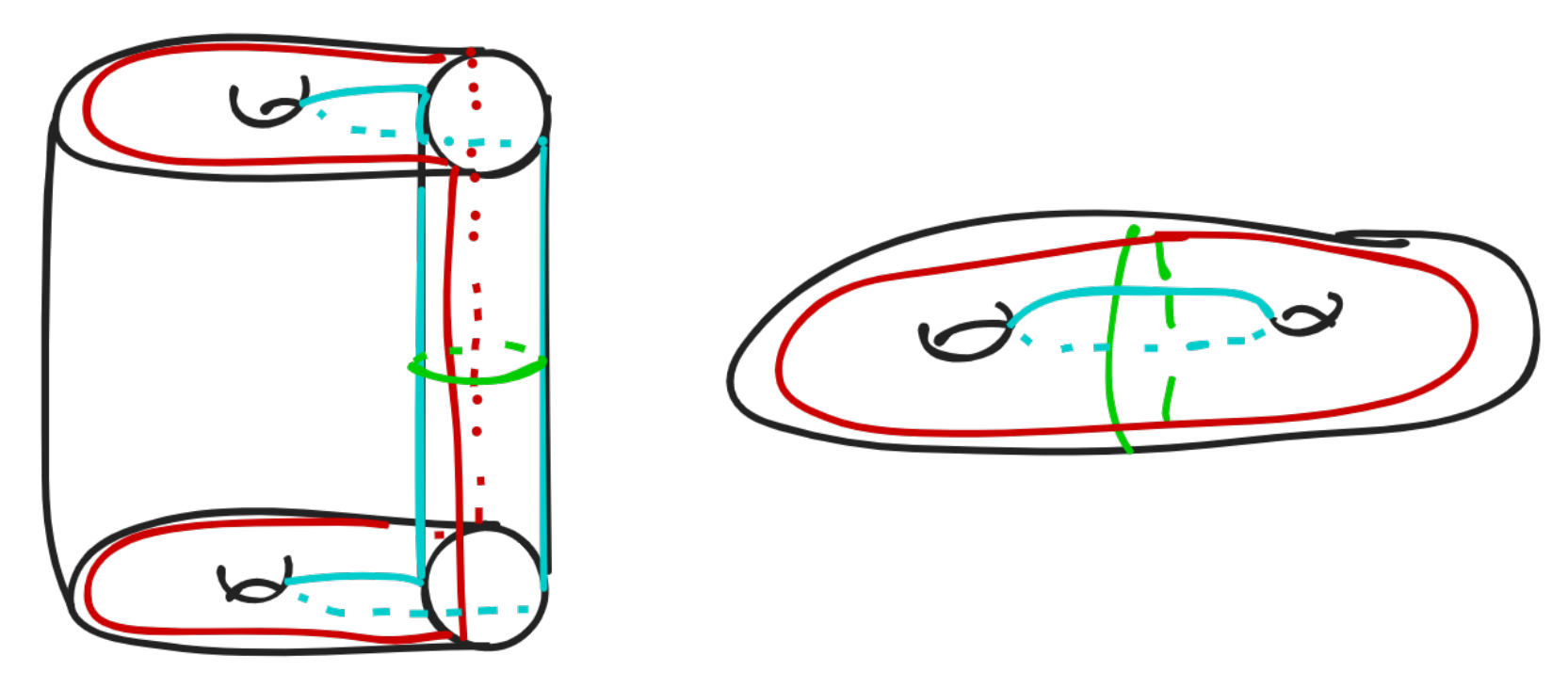}
\caption{Curves on $\Sigma=\pa(F\ti I)$. These curves bound vertical disks.}
\label{fig:Sigma-vs-S}
\end{figure}

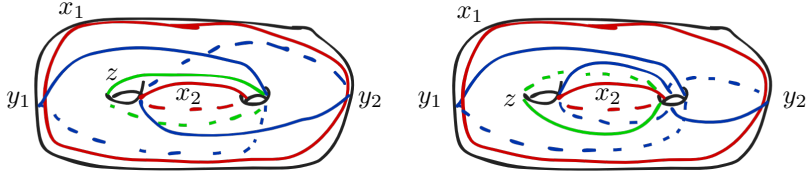
\begin{figure}[h!]
\labellist
\small
\pinlabel $x_1$ at 320 548
\pinlabel $x_2$ at 408 483
\pinlabel $y_1$ at 280 480
\pinlabel $y_2$ at 545 480
\pinlabel $z$ at 350 500
\pinlabel $x_1$ at 620 545
\pinlabel $x_2$ at 724 483
\pinlabel $y_1$ at 590 480
\pinlabel $y_2$ at 865 480
\pinlabel $z$ at 650 480
\endlabellist
\centering
\includegraphics[scale=.5]{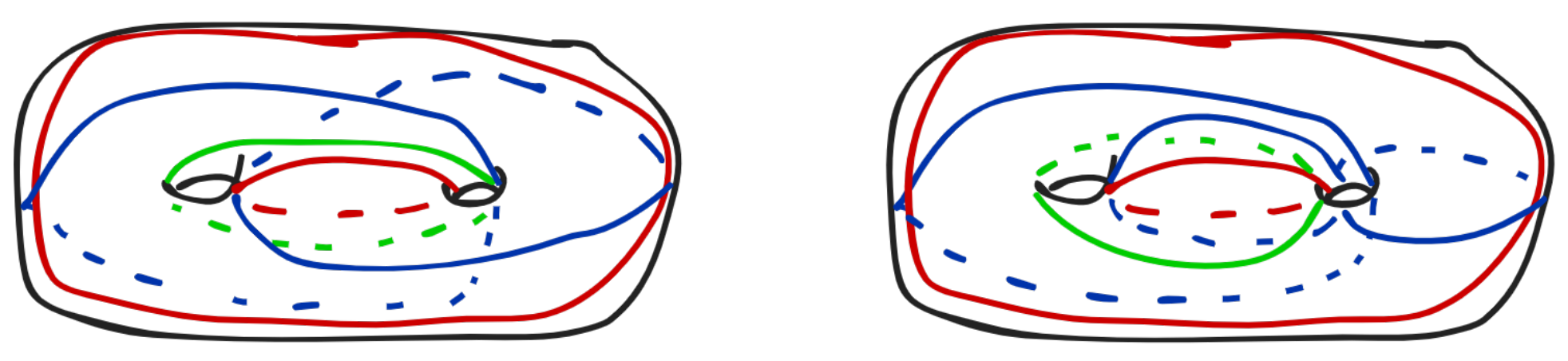}
\caption{Curves on $\Sigma$ (not $S$!) that give Heegaard diagrams for $(F\ti I)\cup(\what F\ti I)$. The cases $\phi=T_aT_b$ (left) and $\phi=T_aT_b^{-1}$ (right).} 
\label{fig:heegaard-diagrams}
\end{figure}


The argument for $\phi=T_aT_b^{-1}$ is similar. We start with the Heegaard diagram $(x_1,x_2;y_1,y_2)$ in Figure \ref{fig:heegaard-diagrams} (right). Then we observe that $(x_1,x_2;y_1,z)$ is also a Heegaard diagram for $M_\phi$, and it is equivalent to the standard Heegaard for $S^3$. 

The last sentence of (i) will be proved in the process of proving (ii). 

%

\boxed{(ii)} Fix $\phi\in\{T_aT_b^{-1},(T_aT_b)^{\pm1}\}$ and $n\neq0$ and define $\psi=\phi\circ T_{\pa F}^n$. We show $M_\psi$ is not homeomorphic to $S^3$. This can be deduced using Casson's invariant $\lambda(M)\in\Z$ for homology 3-spheres and its interpretation in terms of the Torelli group, as we now explain. Replacing the $\phi$ by $\psi$ changes the gluing changes the gluing of the handlebodies $F\ti I,\what F\ti I$ by a separating twist in $\Mod(\Sigma)$. By \cite[Lem.\ 3.4]{morita_casson} this implies that $M_\psi$ is obtained from $M_\phi=S^3$ by $(-1)$-surgery on $K:=(\pa F)\ti\{1\}$. Then by a theorem of Casson (see \cite[pg. xii]{akbulut-mccarthy}), one has
\[\lambda(M_\psi)=\frac{1}{2}\Delta_K''(1),\]
where $\De_K''$ is the second derivative of the Alexander polynomial. Since $\lambda(S^3)=0$, to show $M_\psi\neq S^3$, it suffices to determine $K$ and show $\De_K''(1)\neq0$. 

To identify $K$, we find a homeomorphism from $M_\phi=(F\ti I)\cup (\what F\ti I)$ to the standard Heegaard splitting $S^3=V\cup \what V$, and determine the image of $K$ as a curve on $S=\pa V$.

First we treat the case $\phi=T_aT_b^{-1}$. To draw $K$ on $V$ in this case, we first draw $K$ on $\Si=\pa(F\ti I)$ together with the curves in Figure \ref{fig:heegaard-diagrams} that give a Heegaard diagram that is equivalent to the standard one. Using this, we can we can easily transport $K$ to $S=\pa V$; the resulting curve is the figure-8 knot. See Figures \ref{fig:knot2} and \ref{fig:fig8}.

\begin{figure}[h!]
\labellist
\pinlabel $x_1$ at 690 395
\pinlabel $x_2$ at 790 362
\pinlabel $y_1$ at 665 355
\pinlabel $y_2$ at 715 360
\pinlabel $K\sbs \Sigma$ at 800 300
\pinlabel $x_1$ at 905 355
\pinlabel $x_2$ at 1135 362
\pinlabel $K\sbs S$ at 1040 300
\endlabellist
\centering
\includegraphics[scale=.6]{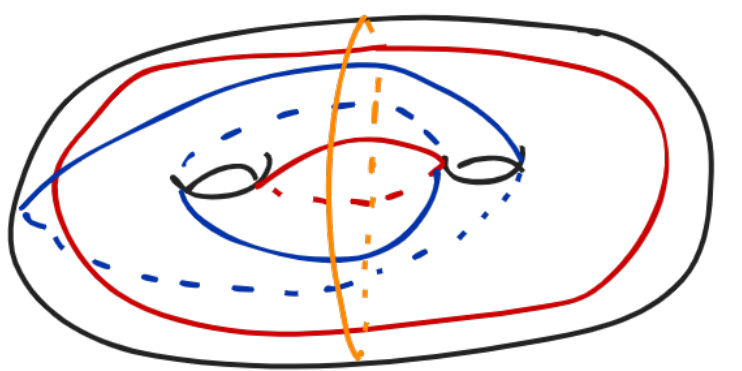}\hspace{.2in}\includegraphics[scale=.6]{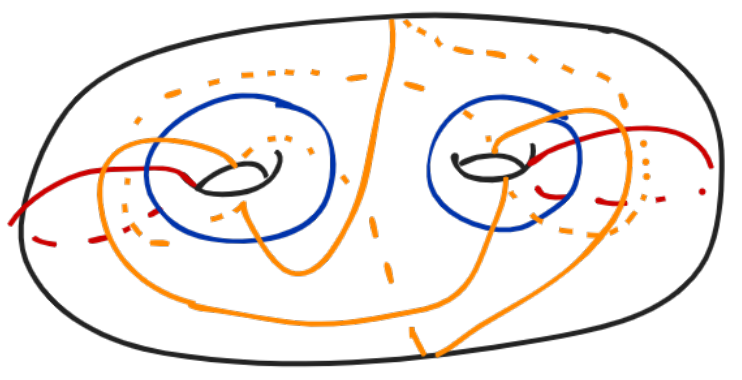}\\
\caption{Left: $\Sigma=\pa(F\ti I)$, a Heegaard diagram (equivalent to the standard one on $S^3$), and the curve $K=(\pa F)\ti1$. Right: $S=\pa V$, the standard Heegaard diagram, and the corresponding curve $K$.}
\label{fig:knot2}
\end{figure}

\begin{figure}[h!]
\labellist
\endlabellist
\centering
\includegraphics[scale=.6]{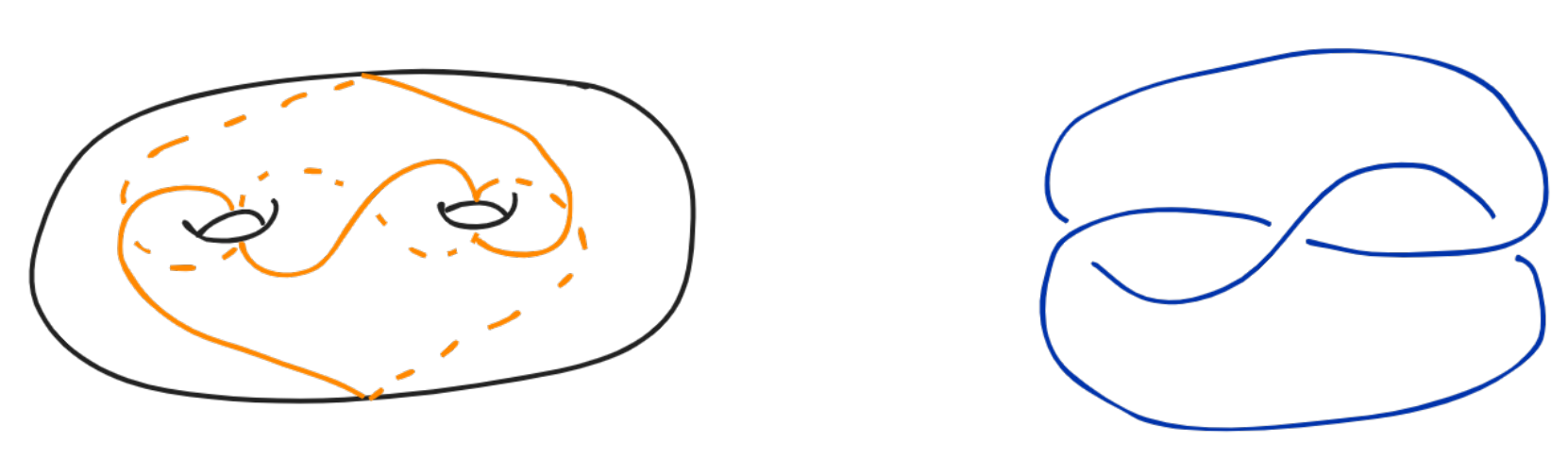}
\caption{Left: $K=\pa X$ in the case $\phi=(T_aT_b)^{\pm1}$, pictured on $S=\pa V$. Right: the trefoil knot.}
\label{fig:knot1}
\end{figure}


For the $K$ the figure-8 knot, $\De_K=3-x-x^{-1}$, and $\De_K''(1)=-2\neq0$. This shows that $M_\psi\not\cong S^3$ when $\psi=\phi\circ T_{\pa}^n$ $(n\neq0)$ and $\phi=T_aT_b^{-1}$. 

The case $\phi=(T_aT_b)^{\pm1}$ can be argued similarly. If we draw $K\subset\Sigma$ on $S$ we arrive at the curve Figure \ref{fig:knot1}, which is the trefoil knot. For $K$ the trefoil knot $\Delta_K=x+x^{-1}-1$, and $\De_K''(1)=2\neq0$ as desired. This shows that $M_\psi\not\cong S^3$ in the case $\psi=\phi\circ T_{\pa}^n$ $(n\neq0)$ and $\phi=(T_aT_b)^{\pm1}$.


This completes the proof of Theorem \ref{thm:I-bundle}.
\end{proof}

\bibliographystyle{alpha}
\bibliography{goeritz}

\section*{Notation guide} 

$V,\what V$ \Tab genus-2 handlebodies in Heegaard splitting of $S^3$

$S$ \Tab genus-2 Heegaard surface embedded in $S^3$

$\Sigma_{g,n}$ \Tab compact surface of genus $g$ with $n$ boundary components

$\ca C(S)$ \Tab curve complex of $S$

$\ca D(V)$ \Tab disk complex of $V$

$\ca P(V)$ \Tab complex of primitive disks in $V$

$\ca R(V,\what V)$ \Tab reducing sphere complex

$D,E$ \Tab primitive disks in $V$

$\what D,\what E$ \Tab primitive disks in $\what V$

$\bb G$ \Tab the genus-2 Goeritz group

$n(c)$ \Tab regular neighborhood of (multi)curve $c\subset S$

$T_c$ \Tab the right Dehn twist about simple closed curve $c$

$X,Y,Z$ \Tab essential subsurfaces of $S$

$i(a,b)$ \Tab geometric intersection number of simple closed curves $a,b$

$(B,a,b)$ \Tab surgery bigon/boundary compression

$\alpha,\beta,\gamma,\delta$ \Tab standard generators for the genus-2 Goeritz group

$d_{\ca C(S)}(a,b)$ \Tab distance in the 1-skeleton of the curve complex $\ca C(S)$

$d_X(a,b)$ \Tab subsurface projection distance

$T$ \Tab an $I$-bundle embedded in $V$

$\pa_hT,\pa_vT$ \Tab horizontal and vertical boundaries of $I$-bundle $T$

\vspace{.2in}
Bena Tshishiku \\
Department of Mathematics, Brown University\\ 
\emph{Email address}: \texttt{bena\_tshishiku@brown.edu}

\end{document}